\documentclass[a4paper,12pt,twoside]{book}
\usepackage{latexsym}
\usepackage[T1]{fontenc}
\usepackage[latin1]{inputenc}
\usepackage[english]{babel}
\usepackage{amsmath, amscd, amsthm, amssymb}
\usepackage{latexsym}
\usepackage[all]{xy}
\usepackage[dvips]{graphicx}
\usepackage{url}
\usepackage{babel}
\DeclareMathOperator{\XY}{\UseComputerModernTips}
\begin{document}

\thispagestyle{empty}
\begin{center}
  \vspace{5mm}
  \LARGE
  \textbf{Pfaffian Calabi-Yau threefolds, Stanley-Reisner schemes and mirror symmetry} \\
  \Large
  \vspace{25mm}
  \large
  \textbf{Ingrid Fausk} \\
  \vspace{45mm}
  \large
  {\textsc{DISSERTATION PRESENTED FOR THE DEGREE\\
        OF PHILOSOPHIAE DOCTOR\\
   }}
  \vspace{14mm}
  \centerline{\includegraphics[width=4cm,height=4cm]{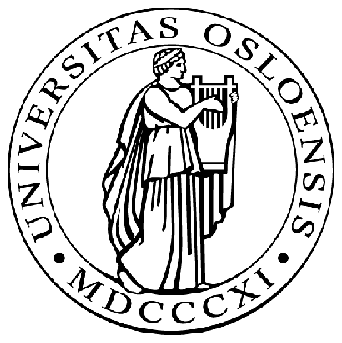}}
  \vspace{14mm}
  \large
  \textsc{DEPARTMENT OF MATHEMATICS\\
          UNIVERSITY OF OSLO\\}
  \vspace{7mm}
  \text{April 2012}
\end{center}

\chapter*{Acknowledgements}

I would like to thank my supervisor, Jan Christophersen, for suggesting this research project, and for his patience and support along the way. Each time I walked out his door, I felt more confident and optimistic than I did when I entered it. Furthermore, I would like to thank him for introducing me to
the field of algebraic geometry in general and to the topics relevant for this thesis in particular. I am grateful to the members of the Algebra and
Algebraic geometry group and the Geometry and Topology group at the University of Oslo for providing
a great environment in which to learn and thrive.
I would like to thank my second supervisor, Klaus Altmann, for hospitality
during my stay at the Freie Universit\"{a}t Berlin.

This thesis was completed during a happy period of my life. I wish to express my love and gratitude to my husband and fellow mathematician Halvard Fausk, and to our daughters, Astrid and Riborg.

\setcounter{chapter}{-1} 

\chapter*{Introduction}

Let $X$ be a smooth complex projective variety of dimension $d$. We call $X$ a {\it Calabi-Yau manifold} if
\begin{itemize}
\item[1.] $H^i(X, \mathcal{O}_X) = 0$ for every $i$, $0 < i < d$, and

\item[2.] $K_X := \wedge^d \Omega_X^1 \cong \mathcal{O}_X$, i.e., the canonical bundle is trivial.
\end{itemize}
\noindent
By the second condition and Serre duality we have

\begin{displaymath}
\text{dim}H^0(X, K_X) = \text{dim}H^d(X, \mathcal{O}_X) = 1
\end{displaymath}
i.e., the geometric genus of $X$ is 1.

Let $\Omega_X^p := \wedge^p \Omega_X^1$ and let $H^q(\Omega_X^p)$ be the $(p,q)$-th {\it Hodge cohomology group} of $X$ with {\it Hodge number} $h^{p,q}(X) := \text{dim}_{\mathbb{C}}H^q(\Omega_X^p)$. The Hodge numbers are important invariants of $X$. There are some symmetries on the Hodge numbers. By complex conjugation we have $H^q(\Omega_X^p) \cong H^p(\Omega_X^q)$ and by Serre duality we have $H^q(\Omega_X^p) \cong H^{d-q}(\Omega_X^{d-p})$. By the {\it Hodge decomposition}

\begin{displaymath}
H^k(X, \mathbb{C}) \cong  \textstyle\bigoplus\nolimits_{p + q = k}  H^q(\Omega_X^p)
\end{displaymath}
\noindent
we have

\begin{displaymath}
h^k(X) =
\underset{p + q = k}
\sum h^{p,q}(X) =  \underset{i = 0}{\overset{k}\sum} h^{i,k-i}(X) \ .
\end{displaymath}

The topological {\it Euler characteristic} of $X$ is an important invariant. It is defined as follows

\begin{displaymath}
\chi(X):=
\underset{k=0}{\overset{2d}\sum} (-1)^k h^k(X) \ .
\end{displaymath}
The conditions for $X$ to be Calabi-Yau assert that $h^{i,0}(X) = 0$ for $0<i<d$ and that $h^{0,0}(X) = h^{d,0}(X) = 1$.

We consider Calabi-Yau manifolds of dimension 3 in this text, these are simply called {\it Calabi-Yau threefolds}. In this case the relevant Hodge numbers are often displayed as a {\it Hodge diamond}.

\begin{center}
$h^{0,0}$\\
$h^{1,0}$ \ \ $h^{0,1}$ \\
$h^{2,0}$ \ \ $h^{1,1}$ \ \ $h^{0,2}$ \\
$h^{3,0}$ \ \ $h^{2,1}$ \ \ $h^{1,2}$ \ \ $h^{0,3}$\\
$h^{3,1}$ \ \ $h^{2,2}$ \ \ $h^{1,3}$ \\
$h^{3,2}$ \ \ $h^{2,3}$ \\
$h^{3,3}$ \\
\end{center}
By the properties mentioned above, the Hodge diamond reduce to

\begin{center}
1\\
0 \ \ 0 \\
0 \ \ $h^{1,1}$ \ \ 0 \\
1 \ \ $h^{2,1}$ \ \ $h^{1,2}$ \ \ 1 \\
0 \ \ $h^{2,2}$ \ \ 0 \\
0 \ \ 0 \\
1\\
\end{center}
\noindent
with the equalities $h^{1,1} =h^{2,2}$ and $h^{1,2} = h^{2,1}$ as explained above. In this case, the Euler characteristic of $X$ is

$$\chi(X) = 2(h^{1,1}(X)  - h^{1,2}(X))$$

Physicists have discovered a phenomenon for Calabi-Yau threefolds, known as {\it mirror symmetry}. This is conjectured to be a correspondence between families of Calabi-Yau threefolds $X$ and $X^{\circ}$ with the isomorphisms \[ H^q(X, \wedge ^p \Theta_X) \cong H^q(X^{\circ}, \Omega^p_{X^{\circ}})\] and vice versa, where $\Theta_X$ is the tangent sheaf of $X$. Since $\wedge^p \Theta_X $ is isomorphic to $\Omega^{3-p}_X$, this gives the numerical equality $h^{p,q}(X) = h^{p,3-q}(X^{\circ})$, and hence $\chi(X) = - \chi(X^{\circ})$, which we will verify for some examples in this thesis. These symmetries correspond to reflecting the Hodge diamond along a diagonal.

For trivial reasons, the mirror symmetry conjecture, as stated above, fails for the Calabi-Yau threefolds where $h^{2,1}(X) = 0$, since Calabi-Yau manifolds are K\"{a}hler, so $h^{1,1}(X)> 0$.

A {\it nonlinear sigma model} consists of  a Calabi-Yau threefold $X$ and a {complexified K\"ahler class} $\omega = B + iJ$ on $X$, where $B$ and $J$ are elements of $H^2(X,\mathbb{R})$, with $J$ a K\"{a}hler class. The {\it moduli}, i.e. how one can deform the complex structure and the complexified structure $\omega$, is governed by $H^1(\Theta_X)$ and
$ H^1(\Omega_X)$, respectively. The isomorphisms $H^1(\Theta_X) \cong H^1(\Omega_{X^{\circ}})$ and $H^1(\Theta_{X^{\circ}})
\cong H^1(\Omega_{X})$ give a local isomorphism between the complex moduli space of $X$ and the
K\"ahler moduli space of $\omega^{\circ}$, and between the complex moduli space of $X^{\circ}$ and the
K\"ahler moduli space of $\omega$. These local isomorphisms are collectively called the {\it mirror map}. A general reference on Calabi-Yau manifolds and mirror symmetry is the book by Cox and Katz~\cite{coxkatz}.

In this thesis we study projective Stanley-Reisner schemes obtained from triangulations of 3-spheres, i.e.~$X_0 := \text{Proj}(A_K)$ for $K$ a triangulation of a 3-sphere and $A_K$ its Stanley-Reisner ring. These schemes are embedded in $\mathbb{P}^n$ for various $n$. We obtain Calabi-Yau 3-folds by smoothing (when a smoothing exists) such Stanley-Reisner schemes.

The first mirror construction by Greene and Plesser for the general quintic hypersurface in $\mathbb{P}^4$ will be reviewed in Chapter 1.

In Chapter 2 we give a method for computing the Hodge number $h^{1,2}(\tilde{X})$ of a small resolution $\tilde{X} \rightarrow X$, where $X$ is a deformation of a Stanley-Reisner scheme $X_0$ with the only singularities of $X$ being nodes. We use results on cotangent cohomology, and a lemma by Kleppe~\cite{kleppe}, which in our case states that $T^1_X \cong T^1_{A,0}$ for $X = \text{Proj}(A)$, i.e. the module of embedded (in $\mathbb{P}^n$) deformations of $X$ is isomorphic to the degree 0 part of the module of first order deformations of the ring $A$. We compute the Hodge number $h^{1,2}(\tilde{X})$ as the dimension of the kernel of the evaluation morphism $T^1_{A,0} \rightarrow \oplus_i T^1_{A_i}$, where $A_i$ is the local ring of a node $P_i$. We use this method in the only non-smoothable example in Chapter 3, where we construct a Calabi-Yau 3-fold with $h^{1,2}(\tilde{X})= 86$ from a small resolution of a variety with one node.

 Gr\"unbaum and Sreedharan~\cite{sreedharan} proved that there are 5 different combinatorial types of triangulations of the 3-sphere with 7 vertices. In Chapter 3 we compute the Stanley-Reisner schemes of these triangulations. They are Gorenstein and of codimension 3, and we use a structure theorem by Buchsbaum and Eisenbud~\cite{buchseisen} to describe the generators of the Stanley-Reisner ideal as the principal Pfaffians of its skew-symmetric syzygy matrix. This approach combined with results by Altmann and Christophersen~\cite{altchrCotangent} on deforming combinatorial manifolds, gives a method for computing the versal deformation space of the Stanley-Reisner scheme of such a triangulation. As we mentioned above, we get a non-smoothable Stanley-Reisner scheme in one case. In the four smoothable cases, we compute the Hodge numbers of the smooth fibers, following the exposition in \cite{roedland}. We also compute the auto\-morphism groups of the triangulations, and consider subfamilies invariant under this action.

 R\o dland constructed in \cite{roedland} a mirror of the 3-fold in $\mathbb{P}^6$ of degree 14 generated by the principal pfaffians of a general $7\times 7$ skew-symmetric matrix with general linear entries, done by orbifolding. B\"ohm constructed in~\cite{boehm} a mirror candidate of the 3-fold in $\mathbb{P}^6$ of degree 13 generated by the principal pfaffians of a $5\times 5$ skew-symmetric matrix with general quadratic forms in one row (and column) and linear terms otherwise. This was done using tropical geometry. In Chapter 4 we describe how the R\o dland and B\"ohm mirrors are obtained from the triangulations in Chapter 3, and in Chapter 5 we verify that the Euler characteristic of the B\"ohm mirror candidate is what it should be.

In general, the mirror constructions we consider in this thesis are obtained in the following way. We consider the automorphism group $G := \text{Aut}(K)$ of the simplicial complex $K$. The group $G$ induces an action on $T^1_{X_0}$, the module of first order deformations of the Stanley-Reisner scheme $X_0$ in the following way. Since an element of $T^1_{X_0}$ is represented by a homomorphism $\phi \in \text{Hom} (I/I^2, A)$, an action of $g\in G$ can be defined by $(g \cdot \phi ) f = g \cdot \phi(g ^{-1}\cdot f)$, where $f \in I$ is a representative for a class in the quotient $I/I^2$.

There is also a natural action of the torus $(\mathbb{C}^*)^{n+1}$ on $X_0 \subset \mathbb{P}^{n}$ as follows. An element $\lambda = (\lambda_0,\ldots, \lambda_{n}) \in (\mathbb{C}^*)^{n+1}$ sends a point $(x_0,\ldots ,x_n) $ of $\mathbb{P}^n$ to $(\lambda_0x_0,\ldots , \lambda_n x_n)$. The subgroup $ \{ (\lambda,\ldots, \lambda) | \lambda \in \mathbb{C}^* \}$ acts as the identity on $\mathbb{P}^n$, so we have an action of the quotient torus $T_n := (\mathbb{C}^*)^{n+1}/\mathbb{C}^*$. Since $I_{X_0}$ is generated by monomials it is clear that $T_n$ acts on $X_0$.

We compute the family of first order deformations of $ X_0$. When the general fiber is smooth, we consider a subfamily, invariant under the action of $G$, where the general fiber $X_t$ of this subfamily has only isolated singularities. We compute the subgroup $H \subset T_n$ of the quotient torus which acts on this chosen subfamily, and  consider the singular quotient $Y_t = X_t/H$. The mirror candidate of the smooth fiber is constructed as a crepant resolution of $Y_t$. In Chapter 4 we perform these computations in order to reproduce the R\o dland and B\"ohm mirrors.

In Chapter 5 we verify that the Euler characteristic of the B\"ohm mirror candidate is 120.  This is as  expected since the cohomology computations in Chapter 3 give Euler characteristic -120 for the original manifold obtained from smoothing the Stanley-Reisner scheme of the triangulation.

We compute the Euler characteristic of the B\"{o}hm mirror using toric geometry. A crepant resolution is constructed locally in 4 isolated $Q_{12}$ singu\-larities. These 4 singularities and two other points are fixed under the action of the group $G$, which is isomorphic to the dihedral group $D_4$. The subgroup $H$ of the quotient torus acting on the chosen subfamily is isomorphic to $\mathbb{Z}/13\mathbb{Z}$. Denote one of these singularities by $V$. The singularity is embedded in $\mathbb{C}^4/H$, which is represented by a cone $\sigma$ in a lattice $N$ isomorphic to $\mathbb{Z}^4$. A resolution $X_{\Sigma} \rightarrow \mathbb{C}^4/H$ corresponds to a regular subdivision of $\sigma$. This subdivision is computed using the Maple package convex~\cite{convex}, and it has 53 maximal cones which are spanned by 18 rays. The following diagram commutes, where $\widetilde{V}$ is the strict transform of $V$.

$$ \XY
  \xymatrix@1{\widetilde{V}\,\ar@{^{(}->}[r] \ar[d] & X_{\Sigma}\ar[d] \\
   V\, \ar@{^{(}->}[r] & \mathbb{C}^4/H }$$

 Each ray $\rho$ in $\Sigma$, aside from the 4 generating the cone $\sigma$, determines an exceptional divisor $D_{\rho}$ in $X_{\Sigma}$. Hence there are 14 exceptional divisors in $X_{\Sigma}$. For every ray $\rho$, the exceptional divisor $D_{\rho}$ is a smooth, complete toric 3-fold and comes with a fan $\text{Star}(\rho)$ in a lattice $N(\rho)$ and a torus $T_{\rho}$ corresponding to these lattices. The subvariety will only intersect 10 of these exceptional divisors $D_{\rho}$. In 9 of these 10 cases the intersection is irreducible and in one case the intersection has 4 components, but one of these is the intersection with another exceptional divisor. All in all the exceptional divisor $E$ in $\tilde{V}$ has 12 components $E_1,\ldots,E_{12}$.

To compute the type of the components $E_i$, several different techniques are needed depending upon the complexity of $D_{\rho}$. In some cases the intersection $\tilde{V}\cap T_{\rho}$ is a torus. In some cases $D(\rho)$ is a locally trivial $\mathbb{P}^1$ bundle over a smooth toric surface. In some cases $E_i$ is an orbit closure in $X_{\Sigma}$ corresponding to a 2-dimensional cone in $\Sigma$. In one case we construct a polytope which has $\text{Star}(\rho)$ as its normal fan.

The space $E$ is a normal crossing divisor. We compute the intersection complex by looking at the various intersections $\tilde{V} \cap D_{\rho_1} \cap D_{\rho_2}$ and $\tilde{V} \cap D_{\rho_1} \cap D_{\rho_2} \cap D_{\rho_3}$, and we compute the Euler characteristic of $E$. For the two other quotient singularities we use the McKay correspondence by Batyrev \cite{bat} in order to find the euler characteristic. We put all this together in order to get the Euler characteristic of the resolved variety.

Computer algebra programs like Macaulay\,2~\cite{M2}, Singular~\cite{GPS05} and Maple \cite{maple} have been used extensively throughout my studies, partly for handling expressions with many parameters and getting overview, but also for proving results. The code is not always included, but it is hoped that enough information is provided in order for the computations to be verified by others.

\tableofcontents

\chapter{Preliminaries}
\section{Simplicial Complexes and Stanley-Reisner schemes}\label{simplicial}

Throughout this thesis we will work over the field of complex numbers $\mathbb{C}$. We will first give some basic definitions. Let $[n] = \{ 0,\ldots,n \}$ be the set of all positive integers from 0 to $n$, and let $\Delta_n$ denote the set of all subsets of $[n]$. We view a {\it simplicial complex} as a subset $K$ of $\Delta_n$ with the property that if $f \in K$, then all the subsets of $f$ are also in $K$. The elements of $K$ are called {\it faces} of $K$. Let $p\in \Delta_n$. In the polynomial ring $R = \mathbb{C}[x_0,\ldots , x_n]$, let $x_p$ be defined as the monomial $\Pi_{i \in p} x_i$. We define the set of ``non-faces'' of $K$ to be the complement of K in $\Delta_n$, i.e. $M_K = \Delta_n \setminus K$. The {\it Stanley-Reisner ideal} $I_K$ is defined as the ideal generated by the monomials corresponding to the "non-faces" of $K$, i.e.

$$I_K = \displaystyle\langle x_p \in R  \mid  p \in M_K \displaystyle\rangle \,\,.$$
\noindent
The {\it Stanley-Reisner ring} is defined as the quotient ring $A_K = R/I_K$. The projective scheme

$$\mathbb{P}(K) := \text{Proj}\, (A_K)$$
\noindent
is called the {\it projective Stanley-Reisner scheme}.

We will need the following definitions. For an face $f \in K$, we define the {\it link} of $f$ in $K$ as the set

$$\text{link}\,(f,K) := \displaystyle\{ g \in K \mid g\displaystyle\cap f = \emptyset \text{\,\,and}\,\, g\cup f \in K \displaystyle\} \, .$$
We set $[K] \subset [n]$ to be the vertex set $[K] = \{ i \in [n] : {i} \in K \}$. The {\it closure} of $f$ is defined as $\overline{f} = \{g \in \Delta_n : g \subseteq f \}$. The {\it boundary} of $f$ is defined as $\partial f= \{ g \in \Delta_n : g \subset f \text{\,\,proper subset} \}$. The {\it join} of two complexes $X$ and $Y$ is defined by

$$X*Y = \displaystyle\{ f \displaystyle\sqcup g \mid f\in X \,\, g\in Y \displaystyle\}\,\, ,$$
\noindent
where the symbol $\sqcup$ denotes disjoint union. The geometric realization of $K$, denoted $|K|$, is defined as

$$|K| := \displaystyle\{ \alpha : [n] \rightarrow [0,1] : \text{supp}(\alpha) \in X \text{\,and\,\,} \underset{i}{\textstyle\sum}\alpha(i) = 1\displaystyle\}\,\, ,$$
where $\text{supp}(\alpha) := \{ i: \alpha(i)\neq 0\}$ is the support of the function $\alpha$. The real number $\alpha(i)$ is called the $i$th {\it barycentric coordinate of $\alpha$}. One can define a metric topology on $K$ by defining the distance $d(\alpha, \beta)$ between two elements $\alpha$ and $\beta$ as

$$d(\alpha, \beta) = \sqrt{\sum_{i \in K}(\alpha(i) - \beta(i))^2}\,\, .$$
\noindent
For a general reference on simplicial complexes, see the book by Spanier~\cite{spanier}.

The schemes $\mathbb{P}(K)$ are singular. In fact, $\mathbb{P}(K)$ is the union of projective spaces, one for each {\it facet} (maximal face) in the simplicial complex $K$, intersecting the same way as the facets intersect in $K$. The proof of this statement is combinatorial: Let $p \in \Delta_n$ be a set with the property that $p\cap q \neq \emptyset$ for all $q \in M_K$ and suppose also that $p \neq [n]$. Then the complement $p^c := [n] - p$ is a face of $K$, and $p^c \neq \emptyset$. Note that if $p$ is a minimal set with the property mentioned above, then $p^c$ is a facet. Recall that $x_p$ is defined as the monomial $x_p := \Pi_{i \in p} x_i$, and that the Stanley-Reisner ideal of $K$ is generated by the monomials $x_q$ with $q\in M_K$. If $x_i = 0$ for all $i \in p$, then all the monomials $x_q$ are zero, since each $x_q$ contains a factor $x_i$ when $p$ has the property mentioned above and $i\in p$. Hence the scheme $\mathbb{P}(K)$ is the union of projective spaces which are defined by such $p$, i.e. given by $x_i = 0$ for all $i \in p$. These projective spaces are of dimension $| p^c| -1$, and they are in one to one correspondence with the faces $p^c$.

We will now mention some special triangulations of spheres which will be of importance in this thesis. The most basic triangulation of the $n-1$-sphere is the boundary $\partial \Delta_n$ of the n-simplex $\Delta_n$ (more precisely, with the definition of boundary of a face given above, it is the boundary of the unique facet $[n] = \{0,\ldots,n\}$ of $\Delta_n$.) For $n = 1$ it is the union of two vertices. For $n = 2$ it is the boundary of a triangle, denoted $E_3$. All triangulations of $\mathbb{S}^1$ are boundaries of $n$-gons, denoted  $E_n$, for $n\geq 3$. The boundary of the 3-simplex $\partial \Delta_3$ is the boundary of a regular pyramid. From now on, we will for simplicity omit the word "boundary", and we will denote the triangulations of spheres as triangles, n-gons, pyramids etc. Other basic triangulations of $\mathbb{S}^2$ are the suspension of the triangle $\Sigma E_3$ (double pyramid) and the {\it octahedron} $\Sigma E_4$ (double pyramid with quadrangle base). Let $C_{k}$ be the chain of $k$ 1-simplices, i.e. $\{ \{0, 1 \},\{1,2\}, \ldots ,\{k-1, k \} \}$. Let $\Delta_1$ be the set of all subsets of $\{ n-2,n-1 \}$. Then we define (the boundary of) the {\it cyclic polytope}, $\partial C(n,3)$, as the union $(\overline{C_{n-3}} * \partial \Delta_1) \cup J$, where $J$ is the join $\Delta_1 * \{\{0\},\{n-3\}\}$ (see the book by Gr\"unbaum \cite{grunbaum} for details).

\section{Deformation Theory}

Given a scheme $X_0$ over $\mathbb{C}$, a {\it family of deformations}, or simply a {\it deformation} of $X_0$ is defined as a cartesian diagram of schemes

$$\XY
 \xymatrix@1{
X_0\ar[r]\ar[d]                 & \mathcal{X}\ar[d]^{\pi}\\
\text{Spec}(\mathbb{C})\ar[r] & S\\
 }$$
where $\pi$ is a flat and surjective morphism and $S$ is connected. The scheme $S$ is called the {\it parameter space} of the deformation, and $\mathcal{X}$ is called the {\it total space}. When $S = \text{Spec} B$ with $B$ an artinian local $\mathbb{C}$-algebra with residue field $\mathbb{C}$ we have an {\it infinitesimal deformation}. If in addition the ring $B$ is the ring of dual numbers, $B = \mathbb{C}[\epsilon]/(\epsilon^2)$, the deformation is said to be of {\it first order}. A {\it smoothing} is a deformation where the general fiber $\mathcal{X}_t$ of $\pi$ is smooth. For a general reference on deformation theory, see e.g. the book by Hartshorne~\cite{hartshorneDef} or the book by Sernesi~\cite{sernesi}.

For a construction of the cotangent cohomology groups in low dimensions, see e.g. Hartshorne~\cite{hartshorneDef}, where {\it cotangent complex} and the cotangent cohomology groups $T^i(A/S,M)$ are constructed for $i = 0,1$ and $2$, where $S\rightarrow A$ is a ring homomorphism and $M$ is an $A$-module. This is part of the cohomology theory of Andr{\'e} and Quillen, see e.g. the book by Andr{\'e} \cite{andre}.

We will be interested in the case with $M = A$ and $S = \mathbb{C}$, and in this case the {\it cotangent modules} will be denoted $T^n_{A}$. We will consider the first three of these. The module $T^0_{A}$ describes the derivation module $\text{Der}_{\mathbb{C}}(A,A)$. The module $T^1_{A}$ describes the first order deformations, and the $T^2_{A}$ describes the obstructions for lifting the first order deformations.

Let $R$ be a polynomial ring over $\mathbb{C}$ and let $A$ be the quotient of $R$ by an ideal $I$. The module $T^1_{A}$ is the cokernel of the map

$$\text{Der}(R,A) \rightarrow \text{Hom}_R(I, A) \cong \text{Hom}_{A}(I/I^2, A)\ ,$$
\noindent
where a derivation $\phi: R \rightarrow A$ is mapped to the restriction $\phi | I : I \rightarrow A$. Let

$$\XY
 \xymatrix@1{
0\,\ar[r] & \text{Rel}\,\,\ar[r] & F\,\ar[r]^j & R\,\ar[r] & A}$$
\noindent
be an exact sequence presenting $A$ as an $R$ module with $F$ free. Let $\text{Rel}_0$ be the submodule of $\text{Rel}$ generated by the Koszul relations; i.e. those of the form $j(x)y - j(y)x$. Then $\text{Rel}/\text{Rel}_0$ is an $A$ module
and we have an induced map

$$\text{Hom}_A(F/\text{Rel}_0 \otimes_R A, A) \rightarrow \text{Hom}_A(\text{Rel}/\text{Rel}_0, A)\,\, .$$
\noindent
The module $T^2_{A}$ is the cokernel of this map.

The $T^i$ functors are compatible with localization, and thus define sheaves.

\newtheorem{defCotCohom}{Definition}[section]
\begin{defCotCohom}
Let $\mathcal{S}$ be a sheaf of rings on a scheme $X$, $\mathcal{A}$ an $\mathcal{S}$-algebra and $\mathcal{M}$ an $\mathcal{A}$-module. We define the sheaf $\mathcal{T}^i_{\mathcal{A}/\mathcal{S}}(\mathcal{M})$ as the sheaf associated to the presheaf

$$U \mapsto T^i(\mathcal{A}(U)/\mathcal{S}(U); \mathcal{M(U)})$$\end{defCotCohom}

 Let $X$ be a scheme $\mathcal{A} = \mathcal{O}_X$, $\mathcal{M} = \mathcal{A}$ and $S = \mathbb{C}$, and denote by $\mathcal{T}^i_X$ the sheaf $\mathcal{T}^i_{\mathcal{O}_X/\mathbb{C}}$. The modules $T_X^i$ are defined as the hyper-cohomology of the cotangent complex on $X$.

 For projective schemes, we will be interested in the deformations that are embedded in $\mathbb{P}^n$, and the following lemma will be useful.

\newtheorem{t1}{Lemma}[section]
\begin{t1}
If $A$ is the Stanley-Reisner ring of a triangulation of a 3-sphere and $X = \text{Proj\,} A$, then there is an isomorphism
\begin{displaymath}
T_X^1 \cong T_{A,\,0}^1\,\, .
\end{displaymath}
\label{lemma:kleppe}
\end{t1}

\begin{proof}
See the article by Kleppe~\cite{kleppe}, Theorem 3.9, which in the case $\mu =0$, $i = 1$ and $n>1$ (and in our notation) states that there is a canonical morphism

$$T^1_{A,0} \rightarrow T_X^1$$
\noindent
which is a bijection if $\text{depth}_m A > 3$, where $m$ is the ideal $\coprod_{i >0}A_i$. Note that the Stanley-Reisner ring corresponding to a triangulation of a sphere is Gorenstein (see Corollary 5.2, Chapter II, in the book by Stanley~\cite{stanley}). If $A$ is the Stanley-Reisner ring of a triangulation of a 3-sphere a, we have $\text{depth}_m A = 4$, hence the morphism above is a bijection.
\end{proof}

When the simplicial complex $K$ is a triangulation of the sphere, i.e. $|K| \cong \mathbb{S}^n$, a smoothing of $X_0$ yields an elliptic curve, a K3 surface or a Calabi-Yau 3-fold when $n = 1$, 2 or 3, respectively. We will prove this in the $n= 3$ case.

\newtheorem{cala}{Theorem}[section]
\begin{cala}A smoothing, if it exists, of the Stanley-Reisner scheme of a triangulation of the 3-sphere yields a Calabi-Yau 3-fold.\label{theorem:cala}\end{cala}

\begin{proof}Sheaf cohomology of $X_0$ is isomorphic to simplicial cohomology of the complex $K$ with coefficients in $\mathbb{C}$, i.e. $h^i(X_0,\mathcal{O}_{X_0}) = h^i(K, \mathbb{C})$. This is proved in Theorem 2.2 in the article by Altmann and Christophersen~\cite{altchrDeforming}. The semicontinuity theorem (see Chapter III, Theorem 12.8 in ~\cite{hartshorne}) implies that $h^i(X_t, \mathcal{O}_{X_t} )= 0$ for all $t$ when $h^i(X_0, \mathcal{O}_{X_0} )= 0$. Third, the Stanley-Reisner scheme $X_0$ of an oriented combinatorial manifold has trivial canonical bundle $\omega_{X_0}$, hence $\omega_{X_t}$ is trivial for all $t$. This is proved in the article by Bayer and Eisenbud~\cite{bayereisenbud}, Theorem 6.1.
\end{proof}

\section{Results on deforming Combinatorial Manifolds}\label{resultsDeforming}

A method for computing the $T^i$ is given in the article by Altmann and Christophersen~\cite{altchrDeforming}. If $K$ is a simplicial complex on the set $\{ 0, \ldots , n \}$ and $A: = A_K$ is the Stanley-Reisner ring associated to $K$, then the $T^1_{A}$ is $\mathbb{Z}^{n+1}$ graded. For a fixed ${\bf c} \in \mathbb{Z}^{n+1}$ write ${\bf c = a - b}$ where ${\bf a} = (a_0, \ldots , a_{n})$ and ${\bf b} = (b_0 ,\ldots b_{n})$ with $a_i,b_i \geq 0$ and $a_ib_i = 0$. Let $x^\mathbf{a}$ be the monomial $x_0^{a_0}\cdots x_{n}^{a_n}$. We define the support of ${\bf a}$ to be $a = \{  i\in [ n]   |  a_i \neq 0 \}$. Thus if ${\bf a} \in \{ 0,1 \}^{n+1}$, then we have $x_a = x^{\bf a}$. If $a, b \subset \{ 0, \ldots ,n \}$ are the supporting subsets corresponding to ${\bf a}$ and ${\bf b}$, then $a \cap b = \emptyset$. The graded piece $T_{A, \bf{c}}^1$ depends only on the supports $a$ and $b$, and vanish unless $a$ is a face in $K$, $\mathbf{b} \in \{0,1 \}^n$ and $b \subset [ \text{link} (a,K) ]$.

 The module $\text{Hom}_R (I_0, A)_{\mathbf{c}}$ sends each monomial $x_p$ in the generating set of the Stanley-Reisner ideal $I_0$ defining $A = R/I_0$ to the monomial $\frac{x_p x^\mathbf{a}}{x^\mathbf{b}}$ when $\mathbf{b}\subset \mathbf{p}$, and 0 otherwise. This corresponds to perturbing the generator $x_p$ of $I_0$ to the generator $x_p + t\frac{x_p x^\mathbf{a}}{x^\mathbf{b}}   $ of a deformed ideal $I_t$.

If $|K| \cong \mathbb{S}^3$, then the link of every face $f$, $| \text{link} (f)|$, is a sphere of dimension $2 - \text{dim}(f)$. We will need some results on how to compute the module $T^1_{A}$ for these Stanley-Reisner schemes. We will list results from~\cite{altchrCotangent}. We write $T^1_{<0}(X)$ for the sum of the graded pieces $T^1_{A,\mathbf{c}}$ with $\mathbf{a} = 0$, i.e. $a = \emptyset$.

\newtheorem{alch}{Theorem}[section]
\begin{alch}
If $K$ is a manifold, then

$$T^1_{A} = \underset{\mathbf{a}\in \,\mathbb{Z}^n \text{\,with\,}\, a\in \,X}{\displaystyle\sum} T^1_{<\,0}(\text{link}\,(a,X))$$
where $T^1_{<0}(\text{link}(a,X))$ is the sum of the one dimensional $T^1_{\emptyset - b}(\text{link}(a,X))$ over all $b \subseteq [\text{link}(a,X)]$ with $|b| \geq 2$ such that $\text{link}(a,X) = L * \partial b$ if $b$ is not a face of $\text{link}(a,X)$, or $\text{link}(a,X) = L * \partial b \cap \partial L * \overline{b}$ if $b$ is a face of $\text{link}(a,X)$. In the first case $| L |$ is a $(n - |b| + 1)$-sphere, in the second case $| L |$ is a $(n - |b |+ 1)$-ball
\end{alch}
The following proposition lists the non trivial parts of $T^1_{<0}(\text{link}(a,X))$.

\newtheorem{listofnontrivial}[alch]{Proposition}
\begin{listofnontrivial}If $K$ is a manifold, then the contributions to $T^1_{<0}(\text{link}(a,X))$ are the ones listed in Table 1.
\begin{table}
\begin{center}
\begin{tabular}{|c|c|c|}
\hline Manifold & K & dim $T^1_{<0}$ \\
\hline
\hline two points & $\partial \Delta_1$ & 1\\
\hline
\hline triangle   & $E_3$ & 4\\
\hline quadrangle & $E_4$ & 2\\
\hline
\hline tetraedron & $\partial \Delta_3$ & 11\\
\hline suspension of triangle & $\Sigma E_3$ & 5\\
\hline octahedron & $\Sigma E_4$ & 3\\
\hline suspension of $n$-gon & $\Sigma E_n$, $n\geq5$ & 1\\
\hline cyclic polytope & $\partial C(n,3)$, $n\geq 6$& 1\\
\hline
\end{tabular}
\end{center}\label{table:T1}
\caption{$T^1$ in low dimensions}
\end{table}
Here $\partial C(n,3)$ is the cyclic polytope defined in section~\ref{simplicial}, and $E_n$ is an n-gon.\label{prop:T1result}
\end{listofnontrivial}

A non-geometric way of computing the degree zero part of the $\mathbb{C}$-vector space $T_A^1$ is given in the Macaulay 2 code in Appendix A, when $p$ is an ideal and $T$ is the polynomial ring over a finite field.

\section{Crepant Resolutions and Orbifolds}\label{section:defcrepant}

In this thesis, we will construct Calabi-Yau manifolds by {\it crepant} resolutions of singular varieties. In some cases these singular varieties are {\it orbifolds}. A crepant resolution of a singularity does not affect the dualizing sheaf. In the smooth case, the dualizing sheaf coincides the canonical sheaf, which is trivial for Calabi Yau manifolds. An orbifold is a generalization of a manifold, and it is specified by local conditions. We will give precise definitions below.

\newtheorem{deforbifold}{Definition}[section]
\begin{deforbifold}
A $d$-dimensional variety $X$ is an {\it orbifold} if every $p \in X$ has a neighborhood analytically equivalent to $0 \in U/G$, where $G \subset GL(n,\mathbb{C})$ is a finite subgroup with no complex reflections other than the identity and $U \subset \mathbb{C}^d$ is a $G$-stable neighborhood of the origin.
\end{deforbifold}
A {\it complex reflection} is an element of $GL(n,\mathbb{C})$ of finite order such that $d - 1$ of its eigenvalues are equal to 1. In this case the group $G$ is called a small subgroup of $GL(n, \mathbb{C})$, and $(U/G,0)$ is called a local chart of $X$ at $p$.

Let $X$ be a normal variety such that its canonical class $K_X$ is $\mathbb{Q}$-Cartier, i.e., some multiple of it is a Cartier divisor, and let $f \colon Y \rightarrow X$ be a resolution of the singularities of $X$. Then

$$K_Y = f^* (K_X) + \textstyle\sum a_iE_i$$
where the sum is over the irreducible exceptional divisors, and the $a_i$ are rational numbers, called the {\it discrepancies}.

\newtheorem{defcansing}[deforbifold]{Definition}
\begin{defcansing}If $a_i \geq 0$ for all $i$, then the singularities of $X$ are called {\it canonical singularities}.
\end{defcansing}

\newtheorem{defcrepant}[deforbifold]{Definition}
\begin{defcrepant}A birational projective morphism $f \colon Y\rightarrow X$ with $Y$ smooth and $X$ with at worst Gorenstein canonical singularities is called a {\it crepant resolution} of $X$ if $f^*K_X = K_Y$ (i.e.~if the {\it discrepancy} $K_Y - f^*K_X$ is zero).
\end{defcrepant}

\section{Small resolutions of nodes}\label{section:isolsing}

Let $X$ be a variety obtained from deforming a Stanley-Reisner scheme obtained from a triangulation of the 3-sphere, where the only singularity of $X$ is a node. If there is a plane $S$ passing through the node, contained in $X$, then there exists a crepant resolution $\pi \colon \tilde{X} \rightarrow X$ with $\tilde{X}$ smooth. To see this, consider a smooth point of $X$. As $S$ is smooth, $S$ is a complete intersection, i.e., defined by only one equation. The blow-up along $S$ will thus have no effect as the blow-up will take place in $X\times \mathbb{P}^0$ outside the singular points. The singularity will be replaced by $\mathbb{P}^1$. The resolution is small (in contrast to the big resolution where the singularity is replaced by $\mathbb{P}^1\times \mathbb{P}^1$), i.e.

\begin{displaymath}
\text{codim} \{ x \in X \mid \text{dim} f^{-1}(x) \geq r \} > 2r
\end{displaymath}
\noindent
for all $r > 0$, hence, the dualizing sheaf is left trivial. The resolved manifold $\tilde{X}$ is Calabi-Yau. This result can be generalized to the case with several nodes, and $S$ a smooth surface in $X$ passing through the nodes. For details, see the article by Werner~\cite{werner}, chapter XI.

\chapter{The Quintic Threefold}\label{4simplex}

It is well known that a smooth quintic hypersurface $X \subset \mathbb{P}^4$ is Calabi-Yau. A smooth quintic hypersurface can be obtained by deforming the projective Stanley-Reisner scheme of the boundary of the 4-simplex. Since the only non-face of $\partial \Delta_4$ is $\{0,1,2,3,4 \}$, the Stanley-Reisner ideal $I$ is generated by the monomial $x_0 x_1 x_2 x_3 x_4$ and the Stanley-Reisner ring is

$$A = \mathbb{C}[x_0,\ldots x_4]/(x_0x_1x_2x_3x_4 )\,\, .$$
\noindent
The automorphism group $\text{Aut}(K)$ of the simplicial complex is the symmetric group $S_5$.

Following the outline described in section~\ref{resultsDeforming}, we compute the family of first order deformations. The deformations correspond to perturbations of the monomial $x_0x_1x_2x_3x_4$. Section~\ref{resultsDeforming} describes which choices of the vectors $\bf a$ and $\bf b$ with support $a$ and $b$ give rise to a contribution to the module $T^1_{X}$.

The link of a vertex $a$ is the tetrahedron $\partial \Delta_3$. The only $b$ with $a\cap b = \emptyset$ and $b$ not face is if $|b| = 4$. The case where $b$ is a face and $|b| = 3$ gives 4 choices for each vertex $a$. The case where $b$ is a face and $|b| = 2$ gives 6 choices for each vertex $a$. All in all, the links of vertices give rise to $5\times 11 = 55$ dimensions of the degree 0 part of $T^1_{A}$ (as a $\mathbb{C}$ vector space).

The link of an edge $a$ is the triangle $\partial \Delta_2$. The only $b$ with $a\cap b = \emptyset$ and $b$ not face is if $|b| = 3$. In this case, there are two possible choices of ${\bf a}$ with support $a$ corresponding to a degree 0 element of $\text{Hom}_R (I_0, A)$. The case where $b$ is a face and $|b| =2$ gives 3 choices for each edge $a$. All in all, the links of edges give rise to $10\times 5 = 50$ dimensions of the degree 0 part of $T^1_A$.

We represent each orbit under the action of $S_5$ by a representative $a$ and $b$, and all the orbits are listed in Table~\ref{table:quinticT1}. Note that the monomials $x_i x_j x_k x_l^2$ are derivations, hence give rise to trivial deformations.

\begin{table}
\begin{center}
\begin{tabular}{|c|c|c|c|}
\hline $a$ & $b$ & perturbation & $\#$ in $S_5$-orbit\\
\hline
$\{ 0 \}$ & $\{1,2,3,4\}$ & $x_0^5$ & 5 \\
\hline
$\{ 0 \}$ & $\{1,2,3\}$ & $x_0^4x_4$ & 20 \\
\hline
$\{ 0 \}$ & $\{1,2\}$ & $x_0^3x_3x_4$ & 30 \\
\hline
$\{ 0,1 \}$ & $\{ 2,3,4 \} $ & $x_0^3x_1^2$ & 20\\
\hline
$\{ 0,1 \}$ & $\{ 2,3 \} $ & $x_0^2x_1^2x_4$ & 30\\
\hline
\end{tabular}
\end{center}
\caption{$T^1_{X_0}$ is 105 dimensional for the quintic threefold $X_0$}
\label{table:quinticT1}
\end{table}
We now choose the one parameter $S_5$-invariant family corresponding to $\mathbf{a}$ a vertex (i.e. support $a = \{ j \}$) and $b = [ \text{link}(a,X) ]$, i.e.

$$X_t = \{(x_0,\ldots x_4) \in \mathbb{P}^4 \mid f_t = 0 \}\,\, ,$$
\noindent
where $f_t = tx_0^5 + tx_1^5 + tx_2^5 + tx_3^5 + tx_4^5 + x_0x_1x_2x_3x_4$. To simplify computations, we set

$$f_t = x_0^5 + x_1^5 + x_2^5 + x_3^5 + x_4^5 -5tx_0x_1x_2x_3x_4 \, .$$
This can be viewed as a family $\mathcal{X} \rightarrow \mathbb{P}^1$ with

$$\mathbb{P}(A) = \mathcal{X}_{\infty} = \{  (x_0, \ldots , x_4) \mid \textstyle\prod_i x_i = 0\}$$
\noindent
our original Stanley-Reisner scheme. The natural action of the torus $(\mathbb{C}^*)^{5}$ on $\mathcal{X}_{\infty} \subset \mathbb{P}^{6}$ is as follows. An element $\lambda = (\lambda_0,\ldots, \lambda_{4}) \in (\mathbb{C}^*)^{5}$ sends a point $(x_0,\ldots ,x_4) $ of $\mathbb{P}^4$ to $(\lambda_0x_0,\ldots , \lambda_4 x_4)$. The subgroup $ \{ (\lambda,\ldots, \lambda) | \lambda \in \mathbb{C}^* \}$ acts as the identity on $\mathbb{P}^4$, so we have an action of the quotient torus $T_4 := (\mathbb{C}^*)^{5}/\mathbb{C}^*$. Since $\mathcal{X}_{\infty}$ is generated by a monomial, it is clear that $T_4$ acts on $\mathcal{X}_{\infty}$.

We compute the subgroup $H \subset T_4$ of the quotient torus acting on $\mathcal{X}_{t}$ as follows. Let the element $\lambda = (\lambda_0, \ldots, \lambda_4)$ act by sending $(x_0,\ldots ,x_4)$ to $(\lambda_0x_0, \ldots ,\lambda_4 x_4)$. For $\lambda$ to act on $X_t$, we must have

$$\lambda_0^5 = \lambda_1^5 = \cdots = \lambda_4^5 = \Pi_{i=0}^4 \lambda_{i}\,\, , $$
 \noindent
 hence $\lambda_i = \xi^{a_i}$ where $\xi$ is a fixed fifth root of 1, and $\sum_i a_i = 0 \,(\text{mod }5)$. Hence $H$ is the subgroup of $(\mathbb{Z}/5\mathbb{Z})^5/(\mathbb{Z}/5\mathbb{Z})$ given by


$$\left\{ (a_0, \ldots , a_4 )  \mid \textstyle\sum  a_i = 0 \right\} \,\, .$$
\noindent
 This group acts on $X_t$ diagonally by multiplication by fifth roots of unity, i.e. $(a_0, \ldots a_4) \in  (\mathbb{Z}/5\mathbb{Z})^5$ acts by $$(x_0, \ldots , x_4) \mapsto (\xi^{a_0} x_0, \ldots , \xi^{a_4} x_4 )$$ where $\xi$ is a fixed fifth root of unity.
We would like to understand the singularities of the space $Y_t := X_t /H$. For the Jacobian to vanish in a point $(x_0,\ldots x_4)$ we have to have $x_i^5 = tx_0x_1x_2x_3x_4$, and hence $\Pi x_i^5 = t^5 \Pi x_i^5$. Thus either $t^5 = 1$ or one of the $x_i$ is zero. But if one $x_i$ is zero, then they all are, and thus $(x_0,\ldots, x_4)$ does not represent a point in $\mathbb{P}^4$. If $t^5 \neq 1$, then $X_t$ is nonsingular. If $t^5 = 1$, then $X_t$ is singular in the points $(\xi^{a_0}, \ldots , \xi^{a_4})$ with $\sum a_i  = 0$ modulo 5. Projectively, these points can be written

$$(1, \xi^{-a_0 + a_1}, \xi^{-a_0 + a_2}, \xi^{-a_0 + a_3}, \xi^{3a_0 - a_1 - a_2 - a_3})\,\, .$$
\noindent
This consists of 125 distinct singular points.

From now on assume that $|t| < 1$. The quotient $X_t/H$ is singular at each point $x$ where the stabilizer $H_x$ is nontrivial. A point in $\mathbb{P}^4$ has nontrivial stabilizer in $H$ if at least two of the coordinates are zero. The points of the curves

$$C_{ij} = \{ x_i = x_j = 0 \} \cap X_t $$
have stabilizer of order 5. For example, the stabilizer of a point of the curve $C_{01}$ is generated by $(2,0,1,1,1)$. The points of the set $$P_{ijk} = \{ x_i = x_j = x_k = 0\} \cap X_t$$ have stabilizer of order 25.

It follows from this that the singular locus of $Y_t$ consists of 10 such curves $C_{ij}/H$. We have $C_{ij}/H = \text{Proj}(R^H)$ where

$$R = \mathbb{C}[x_0,\ldots,x_4]/(x_i, x_j, f_t) \,\,\, .$$
For example, for $C_{01}$ the ring $R$ is

$$\mathbb{C}[x_2,x_3, x_4]/(x_2^5 + x_3^5 + x_4^5)\,\, .$$
\noindent
An element $(a_0,\ldots,a_4)\in H$ now acts on this ring by

$$(x_2,x_3,x_4)\mapsto (\xi^{a_2}x_2, \xi^{a_3}x_3, \xi^{a_4}x_4)\,\, ,$$
\noindent
so we have an action of $(\mathbb{Z}/5\mathbb{Z})^3$ on $R$. For a monomial $x_2^ix_3^jx_4^k$ to be invariant under this group action, we have to have $i = j = k = 0\, \text{ mod }5$, hence

$$R^H = \mathbb{C}[y_0, y_1, y_2]/(y_0 + y_1 + y_2)\,\, ,$$
\noindent
where $y_i = x_{i+2}^5$, and $\text{Proj}(R^H) \cong \mathbb{P}^1$. The curves $C_{ij}$ intersect in the points $P_{ijk}/H$.

The singularity $P_{ijk}/H$ locally looks like $\mathbb{C}^3/(\mathbb{Z}/5\mathbb{Z} \oplus \mathbb{Z}/5\mathbb{Z})$, where the element $(a,b) \in \mathbb{Z}/5\mathbb{Z}\oplus \mathbb{Z}/5\mathbb{Z}$ acts by sending $(u,v,w) \in \mathbb{C}$ to $(\xi^au, \xi^b v, \xi^{-a-b} w)$. To see this, consider for example the set $P := P_{012}$. This set consists of 5 points projecting down to the same point in $Y_t$. A neighborhood $U$ of one of these 5 points projects down to $U/H \subset Y_t$. By symmetry, the other singularities $P_{ijk}$ are similar. The set $P$ is defined by the equations $x_0 = x_1 = x_2 = 0$ and $x_3^5 + x_4^5 = 0$. We consider an affine neighborhood of $P$, so we can assume $x_4 = 1$. Set $y_i = \frac{x_i}{x_4}$. Then we have

$$f = y_0^5 + y_1^5 + y_2^5 + y_3^5 + 1 - 5ty_0y_1y_2y_3\,\,  .$$
The points $x_0 = x_1 = x_2 = x_3^5 + x_4^5= 0$ now correspond to $y_0 = y_1 = y_2 = y_3^5 + 1 = 0$. Now set $z_3 = y_3 + 1$ and $z_i = y_i$ for $i = 0,1,2$. Then we have

$$f = z_0^5 + z_1^5 + z_2^5 + z_3^5u - 5tz_0z_1z_2v\,\,  .$$
where $u = 5 - 10z_3 + 10z_3^2 - 5z_3^3 + z_3^4$ and $v = z_3 - 1$ are units locally around the origin. For a fixed $z_3$ with $(z_3 - 1)^5 = -1$, the group $H$ acts on the coordinates $z_0, z_1, z_2$ by $z_i \mapsto \xi^{a_i}z_i$ with $a_0 + a_1 + a_2 = 0 (\text{mod }5)$, hence we get the quotient $\mathbb{C}^3/(\mathbb{Z}/5\mathbb{Z})^2$ with the desired action.

We can describe this situation by toric methods, i.e. we can find a cone $\sigma^{\nu}$ with

$$\mathbb{C}^3/(\mathbb{Z}/5\mathbb{Z})^2 = \text{Proj} \,\mathbb{C}[y_1,y_2,y_3]^H = U_{\sigma^{\nu}}$$
where $U_{\sigma^{\nu}}$ is the toric variety associated to $\sigma^{\nu}$. For a general reference on toric varieties, see the book by Fulton~\cite{fulton}. A monomial $y_1^{\alpha}y_2^{\beta}y_3^{\gamma}$ maps to $\xi^{a\alpha + b\beta - (a + b)\gamma}y_1^{\alpha}y_2^{\beta}y_3^{\gamma}$, hence the monomial is invariant under the action of $H$ if

$$a\alpha + b\beta - (a + b)\gamma = 0 \,(\text{mod}\,5)\,\,\text{for all}\,(a,b)\,\, ,$$
\noindent
i.e. $\alpha = \beta = \gamma \,(\text{mod}\, 5)$. Let $M \subset \mathbb{Z}^3$ be the lattice

$$M:= \{ (\alpha, \beta, \gamma) | \alpha = \beta = \gamma \, (\text{mod}\, 5) \}\,\, .$$
\noindent
The cone $\sigma^{\nu}$ is the first octant in $M\otimes_{\mathbb{Z}}\mathbb{R} \cong \mathbb{Z}^3\otimes_{\mathbb{Z}}\mathbb{R}$. A basis for $M$ is

$$\begin{bmatrix}
1\\
1\\
1\\
\end{bmatrix},
\begin{bmatrix}
5\\
0\\
0\\
\end{bmatrix},
\begin{bmatrix}
0\\
5\\
0\\
\end{bmatrix}.$$
We have

$$\mathbb{C}[M\cap \sigma^{\nu}] = \mathbb{C}[u^5, v^5, w^5, uvw] = \mathbb{C}[x,y,z,t]/(xyz - t^5)\,\, .$$
\noindent
A basis for the dual lattice $N = \text{Hom}(M, \mathbb{Z})$ is

$$\begin{bmatrix}
1/5\\
0\\
-1/5\\
\end{bmatrix},
\begin{bmatrix}
0\\
1/5\\
-1/5\\
\end{bmatrix},
\begin{bmatrix}
0\\
0\\
1\\
\end{bmatrix},$$
\noindent
and the cone $\sigma$ is the first octant in $\mathbb{R}^3 = N\otimes_{\mathbb{R}}\mathbb{R}$. The semigroup $\sigma \cap N$ is spanned by the vectors $1/5 \cdot (\alpha_1, \alpha_2, \alpha_3)$ with $\alpha_i \in \mathbb{Z}$ and $\sum_i \alpha_i = 5$. Figure~\ref{figure:regular} shows a regular subdivision $\Sigma$ of $\sigma$. The inclusion $\Sigma \subset \sigma$ induces a birational map $X_{\Sigma}\rightarrow U_{\sigma}$ on toric varieties. This gives a resolution of a neighborhood of each point $P_{ijk}$. In the local picture in figure~\ref{figure:regular} we have introduced 18 exceptional divisors, where 6 of these blow down to $P_{ijk}$. In addition 12 of the exceptional divisors blow down to the curves $U_{\sigma}\cap C_{ij}$, $U_{\sigma}\cap C_{ik}$ and $U_{\sigma}\cap C_{jk}$, 4 for each of the three curves intersecting in $P_{ijk}$. This gives $10\times 6 + 10 \times 4 = 100$ exceptional divisors.

\begin{figure}
  \begin{center}
  \includegraphics[width=6.5cm]{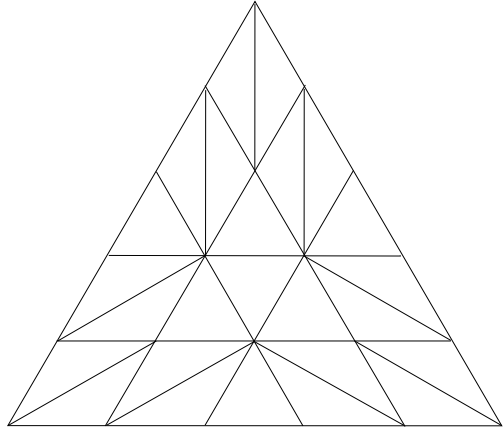}\end{center}
  \caption{Regular subdivision of a neighborhood of the point $P_{ijk}$}\label{figure:regular}
\end{figure}

By this sequence of crepant resolutions we get the desired mirror family $X^\circ_t$. We have $h^{1,1}(X_t) = 1$, $h^{1,2}(X_t) = 101$, $h^{1,1}(X_t^\circ) = 101$ and $h^{1,2}(X_t^\circ) = 1$. For additional details, see the book by Gross, Huybrechts and Joyce~\cite{grosshuybrechtsjoyce}, section 18.2. or the article by Morrison~\cite{morrison}.

\chapter{Hodge numbers of a small resolution of a deformed Stanley-Reisner scheme}\chaptermark{Hodge numbers of a small resolution}\label{ch:cohom}

Let $X = \text{Proj}\,( A)$ be a singular fiber of the versal deformation space of a Stanley-Reisner scheme, with the only singularities of $X$ being a finite number of nodes. Let $\tilde{X}\rightarrow X$ be a small resolution of the singularities. Let $A_i$ be the local rings $\mathcal{O}_{X, P_i}$ where $P_i$ is a node. The Hodge number $h^{1,2}(\tilde{X})$ is the dimension of the kernel of the map $T^1_{A,0}\rightarrow \oplus T^1_{A_i}$. We will prove this in this chapter, and in the next chapter we will apply this result to the non-smoothable case in Section \ref{ex1section}.

We have $\text{dim }H^1(\Theta_{\tilde{X}}) = h^{1,2}(\tilde{X})$ since $H^2(\tilde{X}, \Omega^1) \cong H^1(\tilde{X}, (\Omega^1)^{\nu} \otimes \omega)' \cong H^1(\tilde{X}, \Theta_{\tilde{X}})'$ where the first isomorphism is Serre duality and the second follows from the fact that $\omega_{\tilde{X}}$ is trivial. A general equation for the node is $f = \sum_{i=1}^n x_i^2$. Then we have

\begin{displaymath}
T^1_{A_i} \cong \mathbb{C}[x_1, \ldots, x_n]/(f, \partial f / \partial x_1, \ldots, \partial f / \partial x_n)  \cong \mathbb{C}\,\, .
\end{displaymath}
\noindent
Recall that if $\mathcal{S}$ is a sheaf of rings on a scheme $X$, $\mathcal{A}$ an $\mathcal{S}$-algebra and $\mathcal{M}$ an $\mathcal{A}$-module, we defined the sheaf $\mathcal{T}^i_{\mathcal{A}/\mathcal{S}}(\mathcal{M})$ as the sheaf associated to the presheaf

$$U \mapsto T^i(\mathcal{A}(U)/\mathcal{S}(U); \mathcal{M(U)})$$
\noindent
In this section, let $\mathcal{A} = \mathcal{O}_X$, $\mathcal{M} = \mathcal{A}$ and $S = \mathbb{C}$, and denote by $\mathcal{T}^i_X$ the sheaf $\mathcal{T}^i_{\mathcal{O}_X/\mathbb{C}}(\mathcal{O}_X)$.

\newtheorem{localglobal}{Theorem}[section]
\begin{localglobal}

There is an exact sequence

\begin{displaymath}
\XY
 \xymatrix@1{
0\ar[r] & H^1(\Theta_{\tilde{X}}) \ar[r] & T_{A,\,0}^1\ar[r] & \oplus T_{A_i}^1
 }\,\, ,
\end{displaymath}
\label{th:sequence}
where the map on the right hand side consists of the evaluations of an element of $T_{A,\,0}^1$ in the points $P_i$, and is easy to compute.

\end{localglobal}

\begin{proof}
There is a local-to-global spectral sequence with $E_2^{p,q} = H^p(X, \mathcal{T}_X^q)$ converging to the cotangent cohomology $T^{p+q}_X$. Since $\mathcal{T}^0$ is the tangent sheaf $\Theta_X$, the beginning of the 5-term exact sequence of this spectral sequence is

\begin{displaymath}
\XY
 \xymatrix@1{
0\ar[r] & H^1(\Theta_{X}) \ar[r] & T_{X}^1\ar[r] & H^0(\mathcal{T}_{X}^1)
 }\,\, .
\end{displaymath}
For a general reference on spectral sequences, see e.g. the book by McCleary \cite{mccleary}. By Lemma \ref{lemma:isolsing} we have $H^0(\mathcal{T}_X^1) = \oplus T_{A_i}^1$. For a sheaf $\mathcal{F}$ on $\tilde{X}$, the small resolution $\pi :\tilde{X} \rightarrow X$ gives
a Leray spectral sequence  $H^p(X, R^q\pi_*\mathcal{F})$ converging to $H^n(\tilde{X}, \mathcal{F})$. With $\mathcal{F} = \Theta_{\tilde{X}}$, the beginning of the 5-term exact sequence is

\begin{displaymath}
\XY
 \xymatrix@1{
0\ar[r] & H^1(X, \pi_* \Theta_{\tilde{X}})\ar[r] & H^1(\tilde{X}, \Theta_{\tilde{X}})\ar[r] & H^0(X, R^1 \pi_* \Theta_{\tilde{X}})
} \,\, .
\end{displaymath}
By Lemma \ref{lemma:leray} the last term is zero and $\pi_*\Theta_{\tilde{X}} \cong \Theta_X$, hence we get the isomorphisms $H^1(X, \Theta_X) \cong H^1(X, \pi_* \Theta_{\tilde{X}}) \cong H^1(\tilde{X}, \Theta_{\tilde{X}})$. Lemma \ref{lemma:kleppe} states that $T_X^1 \cong T_{A,0}^1$.
\end{proof}

\newtheorem{isolsing}[localglobal]{Lemma}
\begin{isolsing}If $X$ has only isolated singularities, then $\mathcal{T}_X^1 \cong \oplus T^1_{(X,p)}$.
\label{lemma:isolsing}
\end{isolsing}

\begin{proof}The sheaf $\mathcal{T}_X^1$ is associated to the presheaf $U \mapsto T^1_U$. If $U$ contains no singular points, then $T_U^1 = 0$.
\end{proof}

\newtheorem{leray}[localglobal]{Lemma}
\begin{leray}
We have $\pi_* \Theta_{\tilde{X}} \cong \Theta_X$ and $R^1 \pi_* \Theta_{\tilde{X}} = 0$.
\label{lemma:leray}
\end{leray}

\begin{proof}
$R^1\pi_*\Theta_{\tilde{X}}$ has support in the nodes, so this computation can be done locally. Take an affine neighborhood $V$ of a node, and take the locally small resolution of the node. The node is given by the equation $xy-zw = 0$ in $\mathbb{C}^4$, and $\tilde{V}$ is the blow-up along the ideal $(x,z)$. Hence, $\tilde{V} \subset \{ xU-Tz = 0 \} \subset \mathbb{C}^4 \times \mathbb{P}^1$, where  $(U,T)$ are the coordinates on $\mathbb{P}^1$. We prove first that $H^1(\tilde{V}, \Theta_{\tilde{V}}) = 0$ using \v{C}ech-cohomology. Consider the two maps $U_1$ and $U_2$ given by $T \neq 0$ and $U \neq 0$ respectively. In $U_1$ we have $z = xu$, $y = uw$ and

$$xy - zw = xy - xuw = x(y-uw) = 0\,\, ,$$
where $u$ is the coordinate $U/T$, so the strict transform is given by $y-uw = 0$. Similarly, on the map $U_2$ we get $x = tz$, $w = ty$ and

$$xy - zw = tyz - zw = z(ty-w) = 0\,\, ,$$
where $t$ is the coordinate $T/U$, so the strict transform is given by $ty - w = 0$. On the intersection $U_1 \cap U_2$ we have $t = \frac{1}{u}$, $y=uw$, $z = ux$. The affine coordinate ring of $U_1 \cap U_2$ is $\mathbb{C}[x,u,w, \frac{1}{u}] \cong \mathbb{C}[x,t,w, \frac{1}{t}]$. The differentials $\frac{\partial}{\partial y}$, $\frac{\partial}{\partial z}$, and  $\frac{\partial}{\partial t}$ restricted to the intersection $U_1 \cap U_2$ can be computed as

\begin{equation*}\frac{\partial}{\partial y} = \frac{1}{u}\frac{\partial}{\partial w}\end{equation*}
\begin{equation*}\frac{\partial}{\partial z} = \frac{1}{u}\frac{\partial}{\partial x}\end{equation*}
\begin{equation*}\frac{\partial}{\partial t} = u \left(x\frac{\partial}{\partial x} + w\frac{\partial}{\partial w} - u \frac{\partial}{\partial u} \right)\end{equation*}
\noindent
To see this, apply $\frac{\partial}{\partial y}$, $\frac{\partial}{\partial z}$ and $\frac{\partial}{\partial t}$ on $x$, $w$ and $u$, and keep in mind that we have the relations $x = tz$, $w = ty$ and $tu = 1$.

We prove surjectivity of the map $d: C^0(\tilde{V}) \rightarrow C^1(\tilde{V})$ which sends $(\alpha, \beta) \in \Theta(U_1)\times \Theta(U_2)$ to $(\alpha - \beta) | U_1 \cap U_2$. The elements which do not intersect the image of $\Theta(U_1)\times \{0 \}$ under $d$ are of the form

$$\sum \frac{f_k(x,u,w)}{u^k} \frac{\partial}{\partial x} + \sum \frac{g_k(x,u,w)}{u^k} \frac{\partial}{\partial w} + \sum \frac{h_k(x,u,w)}{u^k} \frac{\partial}{\partial u}$$
\noindent
where $f_k$, $g_k$ and $h_k$ have no term with degree higher than $k-1$ in the variable $u$. The differential $d$ maps

$$-t^{k-1}\frac{\partial}{\partial z} \mapsto \frac{1}{u^k} \frac{\partial}{\partial x}\,\, ,$$
\noindent
and hence
$$p_k(y,z,t) \frac{\partial}{\partial z} \mapsto \frac{f_k(x,u,w)}{u^k}\frac{\partial}{\partial x}\,\, ,$$
where $p_k$ is given by $p_k(y,z,t) = -f_k\left( tz,\frac{1}{t}, ty\right) t^{k-1}$. Similarly we have

$$q_k(y,z,t) \frac{\partial}{\partial y} \mapsto \frac{g_k(x,u,w)}{u^k}\frac{\partial}{\partial w}$$
where $q_k$ is given by $q_k(y,z,t) = -g_k\left( tz, \frac{1}{t}, ty \right) t^{k-1}$. For the last term, we have

$$r_k(y,z,t) \left( y\frac{\partial}{\partial y} + z\frac{\partial}{\partial z} - t\frac{\partial}{\partial t}\right) \mapsto \frac{h_k(x,u,w)}{u^k}\frac{\partial}{\partial u}$$
\noindent
where $r_k(y,z,t) = h_k(tz,\frac{1}{t},ty)t^{k+1}$. Hence $d$ is surjective, and $H^1(\tilde{V}, \Theta_{\tilde{V}}) = 0$.

We construct an isomorphism $\Theta_V \rightarrow H^0(\tilde{V}, \Theta_{\tilde{V}})$ as follows. Consider the map

$$\phi: \Theta_{V} \rightarrow \text{Der}(\mathcal{O}_V, \mathcal{O}_{U_1})\oplus \text{Der}(\mathcal{O}_V, \mathcal{O}_{U_2})$$
given by $D \mapsto (\phi_1 D, \phi_2 D)$, where $\phi_1$ and $\phi_2$ are the inclusions of $\mathcal{O}_V$ into $\mathcal{O}_{U_1}$ and $\mathcal{O}_{U_1}$, respectively. On the generator set $\{ x,y,z,w \}$ they take the following values

\begin{equation*}\phi_1(x) = x,\,\,  \phi_1(y) = uw,\,\, \phi_1(z) = ux,\,\, \phi_1(w) = w\end{equation*}
\begin{equation*}\phi_2(x) = tz,\,\, \phi_2(y) = y ,\,\, \phi_2(z) = z ,\,\, \phi_2(w) = ty\end{equation*}
There is also a map

$$ \text{Der}(\mathcal{O}_{U_1}, \mathcal{O}_{U_1})\oplus \text{Der}(\mathcal{O}_{U_2}, \mathcal{O}_{U_2})  \rightarrow \text{Der}(\mathcal{O}_V, \mathcal{O}_{U_1}) \oplus \text{Der}(\mathcal{O}_V, \mathcal{O}_{U_2})\,\, ,$$
\noindent
which is given by $(D_1, D_2)\mapsto (D_1\phi_1, D_2\phi_2)$. The elements which come from $\Theta_V$ can be lifted to $\oplus_{i=1}^{2} \text{Der}(\mathcal{O}_{U_i}, \mathcal{O}_{U_i})$, and we get elements in $H^0(\tilde{V}, \Theta_{\tilde{V}})$. To see this, note that a generator set for the sheaf $\Theta_V$\label{sidecohom} is

\begin{equation*}
\begin{split}
E = x \frac{\partial}{\partial x} & +  y \frac{\partial}{\partial y}
+ z \frac{\partial}{\partial z} + w\frac{\partial}{\partial w}\\
& y \frac{\partial}{\partial y} - x \frac{\partial}{\partial x}\\
& w \frac{\partial}{\partial y} + x \frac{\partial}{\partial z}\\
& z \frac{\partial}{\partial y} + x \frac{\partial}{\partial w}\\
& y \frac{\partial}{\partial z} + w \frac{\partial}{\partial x}\\
& y \frac{\partial}{\partial w} + z \frac{\partial}{\partial x}\\
& w \frac{\partial}{\partial w} - z \frac{\partial}{\partial z}\\
\end{split}
\end{equation*}
\noindent
They are mapped to the following in $\oplus_{i=1}^2 \text{Der}(\mathcal{O}_V, \mathcal{O}_{U_i})$

\begin{equation*}
\left( x \frac{\partial}{\partial x}  +  uw \frac{\partial}{\partial y}
+ ux \frac{\partial}{\partial z} + w\frac{\partial}{\partial w},\, tz \frac{\partial}{\partial x}  + y \frac{\partial}{\partial y} + z \frac{\partial}{\partial z} + ty\frac{\partial}{\partial w} \right)
\end{equation*}
\begin{equation*}
\left( uw \frac{\partial}{\partial y} - x \frac{\partial}{\partial x},\,
y \frac{\partial}{\partial y} - tz \frac{\partial}{\partial x}  \right)
\end{equation*}
\begin{equation*}
\left( w \frac{\partial}{\partial y} + x \frac{\partial}{\partial z},\,
ty \frac{\partial}{\partial y} + tz \frac{\partial}{\partial z}  \right)
\end{equation*}
\begin{equation*}
\left( ux \frac{\partial}{\partial y} + x \frac{\partial}{\partial w},\,
z \frac{\partial}{\partial y} + tz \frac{\partial}{\partial w} \right)
\end{equation*}
\begin{equation*}
\left( uw \frac{\partial}{\partial z} + w \frac{\partial}{\partial x},\,
y \frac{\partial}{\partial z} + ty \frac{\partial}{\partial x} \right)
\end{equation*}
\begin{equation*}
\left( uw \frac{\partial}{\partial w} + ux \frac{\partial}{\partial x},\,
y \frac{\partial}{\partial w} + z \frac{\partial}{\partial x} \right)
\end{equation*}
\begin{equation*}
\left( w \frac{\partial}{\partial w} - ux \frac{\partial}{\partial z},\,
ty \frac{\partial}{\partial w} - z \frac{\partial}{\partial z} \right)
\end{equation*}

These 7 elements can be lifted to the following elements in $\oplus_{i=1}^2 \text{Der}(\mathcal{O}_{U_i}, \mathcal{O}_{U_i})$.

\begin{equation*}\left( x \frac{\partial}{\partial x}  +  w\frac{\partial}{\partial w}, y \frac{\partial}{\partial y} + z \frac{\partial}{\partial z}\right)\end{equation*}
\begin{equation*}\left( u \frac{\partial}{\partial u} - x \frac{\partial}{\partial x}, y \frac{\partial}{\partial y} - t \frac{\partial}{\partial t}\right)\end{equation*}
\begin{equation*}\left( \frac{\partial}{\partial u}, t \left( y\frac{\partial}{\partial y} + z \frac{\partial}{\partial z} - t \frac{\partial}{\partial t} \right) \right)\end{equation*}
\begin{equation*}\left( x \frac{\partial}{\partial w}, z \frac{\partial}{\partial y}\right)\end{equation*}
\begin{equation*}\left( w \frac{\partial}{\partial x}, y \frac{\partial}{\partial z}\right)\end{equation*}
\begin{equation*}\left( u \left(x \frac{\partial}{\partial x} + w \frac{\partial}{\partial w} - u\frac{\partial}{\partial u} \right), \frac{\partial}{\partial t}\right)\end{equation*}
\begin{equation*}\left( w \frac{\partial}{\partial w} - u \frac{\partial}{\partial u}, t \frac{\partial}{\partial t} - z \frac{\partial}{\partial z}\right)\end{equation*}
\noindent
We can construct an inverse map $g:H^0(\tilde{V}, \Theta_{\tilde{V}}) \rightarrow  \Theta_{V}$ by $$g(D_1, D_2) = \frac{1}{2}(D_1\phi_1 + D_2\phi_2)\,\, .$$
\noindent
Since $\pi_*(\Theta_{\tilde{V}}) \cong R^0\pi_*(\Theta_{\tilde{V}}) \cong H^0(\tilde{V},\Theta_{\tilde{V}})$ and $R^1\pi_*(\Theta_{\tilde{V}}) \cong H^1(\tilde{V},\Theta_{\tilde{V}})$, we get the desired result.
\end{proof}

\chapter{Stanley-Reisner Pfaffian Calabi-Yau 3-folds in $\mathbb{P}^6$}\chaptermark{S-R Pfaffian Calabi-Yau 3-folds in $\mathbb{P}^6$}\label{chapter:srpfaffians}

\section{Triangulations of the 3-sphere with 7 vertices}

In this chapter we look at the triangulations of the 3-sphere with 7 vertices. Table~\ref{table:pol} is copied from the article by Gr\"unbaum and Sreedharan~\cite{sreedharan}, where all the combinatorial types of triangulations of the 3-sphere with 7 or 8 vertices are listed. For each such combinatorial type (from now on referred to as a \textit{triangulation}) we compute the versal deformation space of the corresponding Stanley-Reisner scheme, and we check if the general fiber is smooth. In the smoothable cases, we compute the Hodge numbers of the general fiber. We also compute the automorphism group of the triangulation, and we compute the subfamily of the versal deformation space invariant under this group action. In the non-smoothable case, we construct a small resolution of the nodal singularity of the general fiber.

Let $M = [m_{ij}]$ be a skew-symmetric $d \times d$ matrix (i.e., $m_{ij} = -m_{ji}$) with entries in a ring $R$. One can associate to $M$ an element $\text{Pf}(M)$ in $R$ called the {\it Pfaffian} of $M$: When $d = 2n$ is even, we define the Pfaffian of $M$ by the closed formula

\begin{displaymath}
\text{Pf}(M) = \frac{1}{2^n n!}\underset{\sigma \in S_{2n}}{\textstyle\sum} \text{sgn}(\sigma) \underset{i =1}{\overset{n}{\textstyle\prod}}m_{\sigma(2i-1),\sigma(2i)}
\end{displaymath}
where $S_{2n}$ is the symmetric group on $2n$ elements, and $\textit{sgn}(\sigma)$ is the signature of $\sigma$. When $d$ is odd, we define $\text{Pf}(M) =0$.

The Pfaffian of a skew-symmetric matrix has the property that the square of the pfaffian equals the determinant of the matrix, i.e.

\begin{displaymath}
\text{Pf}(M)^2 = \text{det}(M)\,\,\, .
\end{displaymath}
\noindent
In this chapter the ring $R$ will be the polynomial ring $\mathbb{C}[x_1,\ldots,x_7]$, and we will study ideals generated by such Pfaffians. In this case, the sign of the Pfaffian can be chosen arbitrarily, so it suffices to compute the Pfaffian as one of the square roots of the determinant.

\begin{table}
\begin{center}
\begin{tabular}{|c|c|cc|}
\hline Polytope & Number of facets & List of facets & \\
\hline
&&&\\
$P_1^7$ & 11 & A: 1256 & H: 1367 \\
        &    & B: 1245 & J: 2367 \\
        &    & C: 1234 & K: 2345 \\
        &    & D: 1237 & L: 2356 \\
        &    & E: 1345 &         \\
        &    & F: 1356 &         \\
        &    & G: 1267 &         \\
        &    &         &         \\
$P_2^7$ & 12 & A: 1245 & H: 2356 \\
        &    & B: 1246 & J: 2347 \\
        &    & C: 1256 & K: 2367 \\
        &    & D: 1345 & L: 2467 \\
        &    & E: 1346 & M: 3467 \\
        &    & F: 1356 &         \\
        &    & G: 2345 &         \\
        &    &         &         \\
$P_3^7$ & 12 & A: 1246 & H: 1347 \\
        &    & B: 1256 & J: 2346 \\
        &    & C: 1257 & K: 2356 \\
        &    & D: 1247 & L: 2357 \\
        &    & E: 1346 & M: 2347 \\
        &    & F: 1356 &         \\
        &    & G: 1357 &         \\
        &    &         &         \\
$P_4^7$ & 13 & A: 2467 & H: 1456 \\
        &    & B: 2367 & J: 1247 \\
        &    & C: 1367 & K: 1237 \\
        &    & D: 1467 & L: 1345 \\
        &    & E: 2456 & M: 2345 \\
        &    & F: 2356 & N: 1234 \\
        &    & G: 1356 &         \\
        &    &         &         \\
$P_5^7$ & 14 & A: 1234 & H: 1567 \\
        &    & B: 1237 & J: 2345 \\
        &    & C: 1267 & K: 2356 \\
        &    & D: 1256 & L: 2367 \\
        &    & E: 1245 & M: 3467 \\
        &    & F: 1347 & N: 3456 \\
        &    & G: 1457 & O: 4567 \\
\hline
\end{tabular}
\caption{Polytopes $P_i^7$, $i=1,\ldots ,5$}
\label{table:pol}
\end{center}
\end{table}

For a sequence $i_1, \ldots, i_m$, $1 \leq i_j \leq d$, the matrix obtained from $M$ by omitting rows and columns with indices $i_1, \ldots ,i_m$ is again skew-symmetric; we write $\text{Pf}^{i_1,\ldots ,i_m}(M)$ for its Pfaffian. The elements $\text{Pf}^{i_1,\ldots ,i_m}(M)$ are called Pfaffians of order $d-m$. The Pfaffians of order $d-1$ of a $d\times d$ matrix $M$ are called the {\it principal Pfaffians} of $M$.

\section{Computing the versal family}

The Stanley-Reisner rings obtained from the triangulations in Table \ref{table:pol} are Gorenstein of codimension 3 (in fact, the Stanley-Reisner ring corresponding to any triangulation of a sphere is Gorenstein, see Corollary 5.2, Chapter II, in the book by Stanley~\cite{stanley}). Buchsbaum and Eisenbud proved in Theorem 2.1 (and its proof) in their article~\cite{buchseisen} that Gorenstein codimension 3 ideals are generated by the principal Pfaffians of their skew-symmetric syzygy matrix.

The Stanley-Reisner ideals obtained from the triangulations in Table~\ref{table:pol} are generated by $d = 3,5$ or $7$ monomials. In each case, the following resolution can be computed.

\newtheorem{eksaktsekvensGenerell}{Lemma}[section]
\begin{eksaktsekvensGenerell}
For the Stanley-Reisner ideals $I_0$ obtained from the triangulations in Table~\ref{table:pol}, there is a free resolution of the Stanley-Reisner ring $A = R/I_0$
$$
\XY
 \xymatrix@1{
0\ar[r] & R\ar[r]^{f} & R^d \ar[r]^{M} & R^d \ar[r]^{f^T}
& R \ar[r] & A \ar[r] & 0
 }\,\,\, ,
$$
\noindent
where $f$ is a vector with entries the generators of $I_0$, $M$ is an skew-symmetric $d\times d$ syzygy matrix and $I_0$ is generated by the principal pfaffians of $M$.\label{lemma:eksaktsekvensGenerell}
\end{eksaktsekvensGenerell}

In the sections 3.4 - 3.8 we compute the degree zero part of the $\mathbb{C}$-vector space $T_A^1$ as described in Section~\ref{resultsDeforming}. This gives us a new perturbed ideal $I_1$, with $k$ parameters, one for each choice of ${\bf a}$ and ${\bf b}$ that contribute to $T_A^1$ of degree zero. We get a perturbed vector $f^1$ with entries the generators of $I_1$, and we get a new matrix $M^1$ by perturbing the entries of the matrix $M$ in such a way that skew-symmetry is preserved, keeping the entries homogeneous in $x_1, \ldots ,x_7$ such that $M^1\cdot f^1 = 0$ mod $t^2$, where $t$ is the ideal $(t_1,..,t_{k})$. This gives the first order embedded (in $\mathbb{P}^6$) deformations.

It has not yet been possible for computers to deal with free resolutions over rings with many parameters. Finding the matrix $M^1$ can however be done manually, by considering the parameters one by one, perturbing the entries of the matrix $M$ keeping skew-symmetry preserved. The principal pfaffians of the matrix $M^1$ give the versal family up to all orders. This follows from Theorem 9.6 in the book by Hartshorne~\cite{hartshorneDef}. Versality follows from the fact that the Kodaira-Spencer map is surjective, see Proposition 2.5.8 in the book by Sernesi~\cite{sernesi}.

We have computed this family explicitly for these five triangulations from Table~\ref{table:pol}.

\section{Properties of the general fiber}

We will now compute the degrees of the varieties obtained from the triangulations in Table~\ref{table:pol}.

\newtheorem{degree}{Lemma}[section]
\begin{degree}
The number of maximal facets of a triangulation equals the degree of the associated variety.
\end{degree}
\begin{proof}Let $d$ be the dimension $d = \text{dim} R/I$. The Hilbert series is\begin{displaymath}\textstyle\sum_{i=-1}^{d-1} \frac{f_i t^{i+1}}{(1-t)^{i+1}} = \frac{1}{(1-t)^d}\textstyle\sum_{i=-1}^{d-1} (1-t)^{d-i-1} f_i t^{i+1}\end{displaymath} where $f_i$ is the number of facets of dimension $i$ and $f_{-1} = 1$, see the book by Stanley~\cite{stanley}. The maximal facets have dimension $d-1$. Inserting $t=1$ in the numerator yields the degree $f_{d-1}$.
\end{proof}
The triangulations in Table~\ref{table:pol} give rise to varieties of degree 11, 12, 13 and 14. The degree is invariant under deformation, so in the smoothable cases we can construct Calabi-Yau 3-folds of degree 12, 13 and 14. The following theorem will be proved in Sections 3.4 -- 3.8.

\newtheorem{summarySeven}[degree]{Theorem}
\begin{summarySeven}
Some invariants of the general fiber of the versal deformation space of the Stanley-Reisner rings of the triangulations in Table~\ref{table:pol} are given in Table~\ref{table:summary}.
\end{summarySeven}

\begin{table}
\begin{center}
\begin{tabular}{|c|c|c|c|c|}
\hline
          &        &                   &         &            \\
Polytope  & Degree & General fiber     & Hodge   & Dimension     \\
          &        & in the versal     & numbers & of the versal \\
          &        & deformation space &         & base space    \\
          &        &                   &         &               \\
\hline
&&&&\\
$P_1^7$ & 11 & Isolated nodal           & non-           & 92\\
        &    & singularity with         & smoothable     &   \\
        &    & small resolution         &                &   \\
        &    &                          &                &   \\
$P_2^7$ & 12 & Complete intersection    & $h^{1,1} = 1$  & 79\\
        &    & type $(2,2,3)$           & $h^{1,2} = 73$ &   \\
        &    &                          &                &   \\
$P_3^7$ & 12 & Complete intersection    & $h^{1,1} = 1$  & 79\\
        &    & type $(2,2,3)$           & $h^{1,2} = 73$ &   \\
        &    &                          &                &   \\
$P_4^7$ & 13 & Pfaffians of $5\times 5$ & $h^{1,1} = 1$  & 67\\
        &    & matrix with general      & $h^{1,2} = 61$ &   \\
        &    & quadratic terms in       &                &   \\
        &    & first row/column         &                &   \\
        &    & and general linear       &                &   \\
        &    & terms otherwise          &                &   \\
        &    &                          &                &   \\
$P_5^7$ & 14 & Pfaffians of $7\times 7$ & $h^{1,1} = 1$  & 56\\
        &    & matrix with general      & $h^{1,2} = 50$ &   \\
        &    & linear terms             &                &   \\
        &    &                          &                &   \\
\hline
\end{tabular}
\caption{Polytopes $P_i^7$, $i=1,\ldots ,5$, and their deformations}
\label{table:summary}
\end{center}
\end{table}
Note that the dimension of the versal base space equals $h^{1,2} + 6$ in the four smoothable cases. Theorem 5.2 in~\cite{altchrDeforming} states that there is an exact sequence

$$0 \rightarrow \mathbb{C}^6 \rightarrow H^0(\Theta_{X_t}) \rightarrow H^1(K,\mathbb{C}) \rightarrow 0$$
\noindent
Since the last term is zero, we have $\text{dim\,} H^0(\Theta_X) = 6$. One would expect that $T^1_{X_0} = h^1(\Theta_{X_t}) + h^0(\Theta_{X_0}),$ where $X_t$ is a general fiber and $X_0$ is the central fiber of the versal deformation space.

After we have resolved the singularity in the non-smoothable case, we get a Calabi-Yau manifold with $h^{1,2}(X) = 86$, and since 86 + 6 = 92, this fits nicely also in the non-smoothable case.

\section{The triangulation $P^7_1$}\label{ex1section}

In this section we consider $P^7_1$, the first triangulation of $\mathbb{S}^3$ from Table~\ref{table:pol}. The Stanley-Reisner ideal of this triangulation is

\begin{displaymath}
I_0 = (x_5 x_7,x_4 x_7, x_4 x_6, x_1 x_2 x_3 x_6, x_1 x_2 x_3 x_5)
\end{displaymath}
\noindent
in the polynomial ring $R = \mathbb{C}[x_1, \ldots ,x_7]$. Let $A = R /I_0$ be the Stanley-Reisner ring of $I_0$. In the minimal free resolution in Lemma~\ref{lemma:eksaktsekvensGenerell}, the vector $f$ and the matrix $M$ are given by

$$f =
\begin{bmatrix}
x_5 x_7 \\
x_4 x_7 \\
x_4 x_6 \\
x_1 x_2 x_3 x_6 \\
x_1 x_2 x_3 x_5 \\
\end{bmatrix}\,\,\,$$
and
$$M =
\begin{bmatrix}
0                    & 0        &  -x_1x_2x_3 &  x_4  & 0   \\
0                    & 0        &  0                    &  -x_5 & x_6\\
 x_1x_2x_3 & 0       &  0                     &  0       &  -x_7\\
-x_4                & x_5    &  0                    &  0       &  0  \\
0                    & -x_6  &  x_7                &  0         &  0  \\
\end{bmatrix}\,\,\, .$$
\noindent

Using the results of section~\ref{resultsDeforming}, we compute the module $T^1_X$, i.e. the first order embedded deformations, of the Stanley-Reisner scheme $X$ of the complex $K:=P^7_1$ by considering the links of the faces of the complex. Various combinations of $a,b \in \{ 1,\ldots ,7 \}$, with $b \subset [ \text{link}(a,K)]$ a subset of the vertex set and $a$ a face of $K$, contribute to $T^1_X$.

\begin{figure}
  \begin{center}
  \includegraphics[width=9cm]{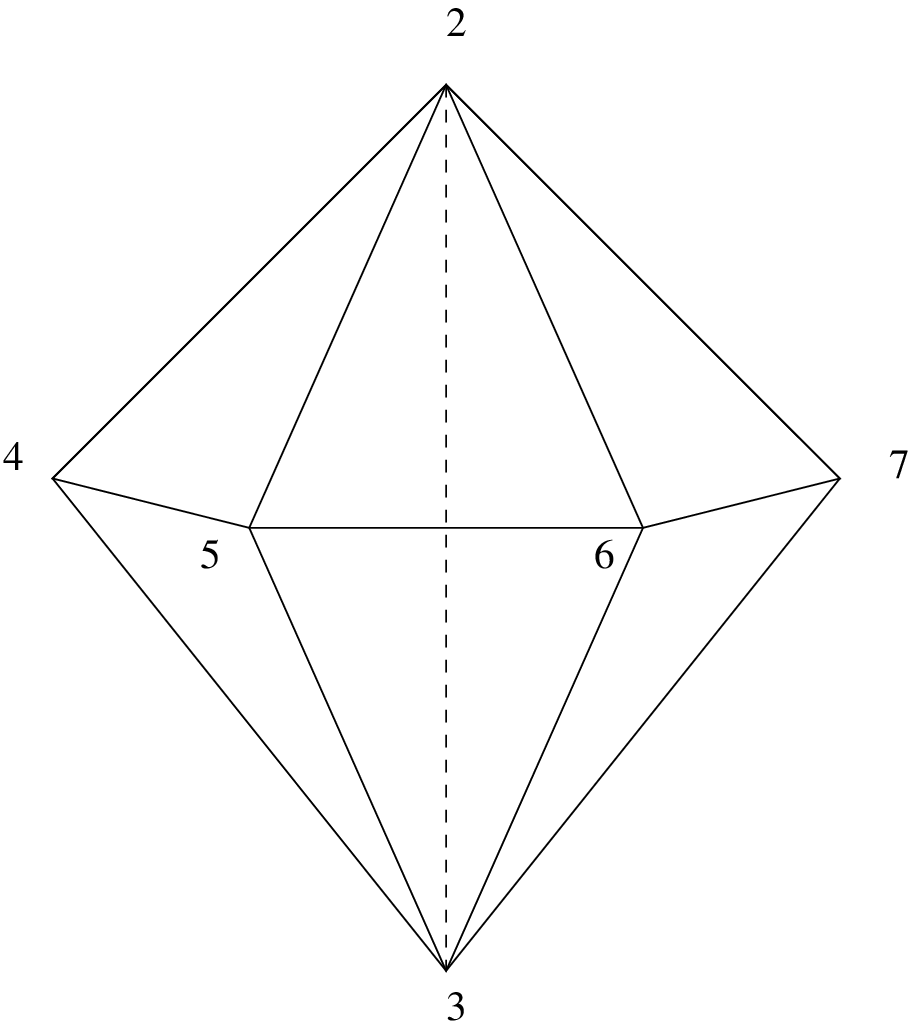}\end{center}
  \caption{The link of the vertex $\{1 \}$ in $P_7^1$}\label{figure:link1}
\end{figure}

The geometric realization $| \text{link}(1,K) |$ of the link of the vertex $\{1\}$ in $K$ is the boundary of a cyclic polytope, and is illustrated in figure~\ref{figure:link1}. The links of the vertices $\{2\}$ and $\{3\}$ are similar.

\begin{figure}
  \begin{center}
  \includegraphics[width=9cm]{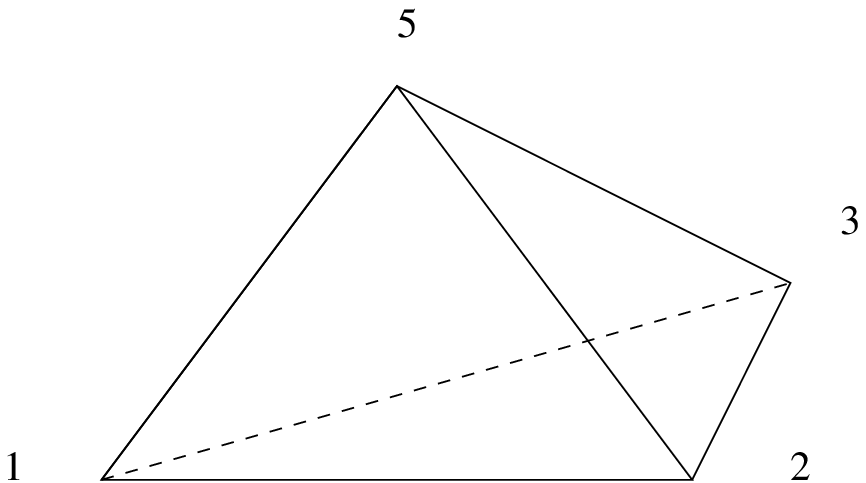}\end{center}
  \caption{The link of the vertex $\{4 \}$ in $P_7^1$}\label{figure:link4}
\end{figure}
\begin{figure}
  \begin{center}
  \includegraphics[width=7.6cm]{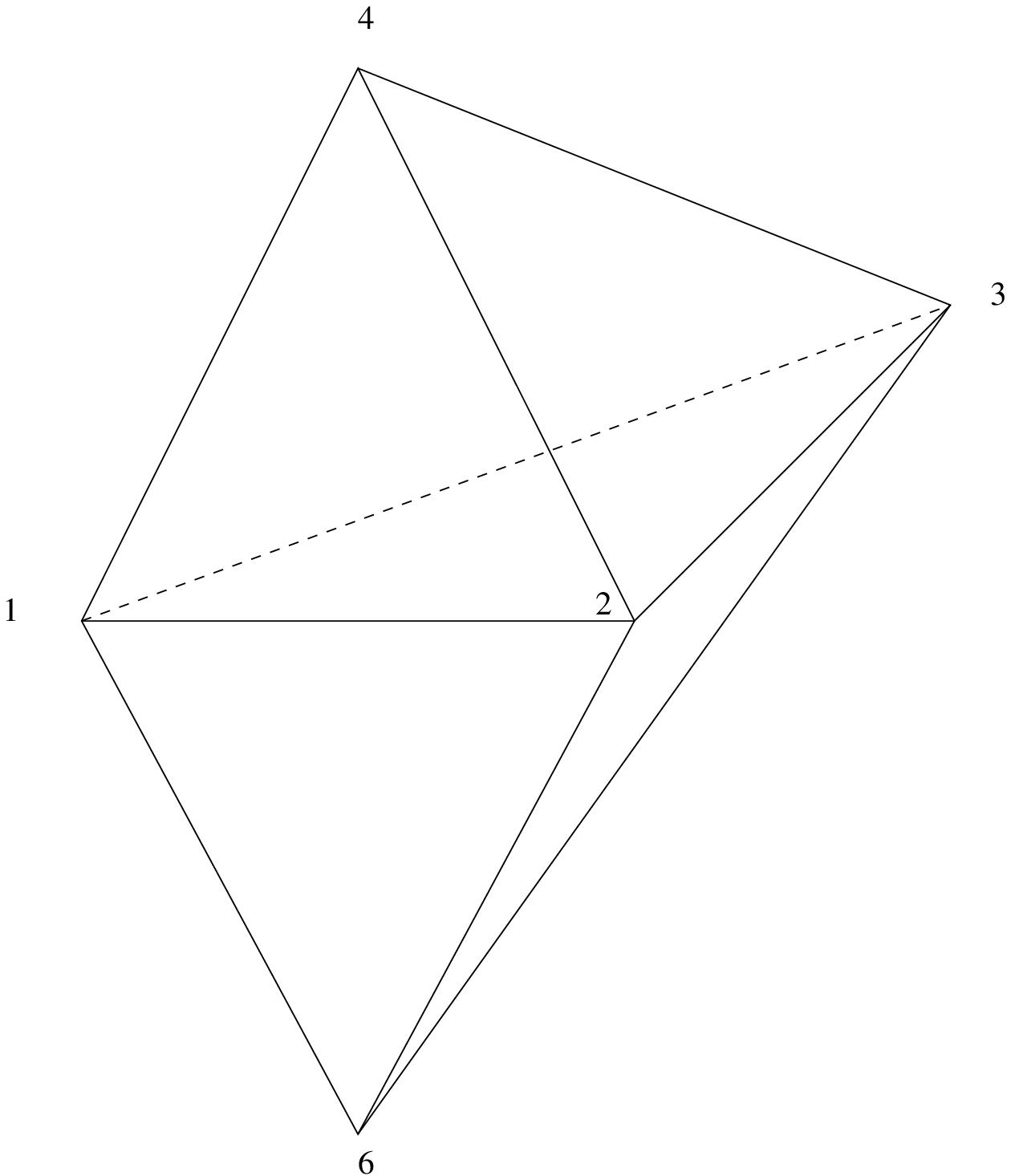}\end{center}
  \caption{The link of the vertex $\{ 5 \}$ in $P_7^1$}\label{figure:link5}
\end{figure}

Two vertices, $\{ 4\}$ and $\{ 7 \}$, give rise to a tetraedron (see figure~\ref{figure:link4}), and two vertices, $\{5\}$ and $\{6 \}$, give rise to a suspension of a triangle (see figure~\ref{figure:link5}). We also consider links of one dimensional faces. In 9 cases, the geometric realization is a triangle. The case of $\{ 1,4 \}\in K$ is illustrated in figure~\ref{figure:link14}. In 6 cases, the link is a quadrangle. The case of $\{ 1,5 \}\in K$ is illustrated in figure~\ref{figure:link15}.

In Proposition~\ref{prop:T1result} the contributions to $T^1_A$ of these different links are listed. In the case with $a=1$, we get a contribution to $T^1_X$ if and only if $b = \{2,3 \}$. As in section~\ref{resultsDeforming}, this gives a homogeneous perturbation the monomial $x_1 x_2 x_3 x_6$ to

$$x_1 x_2 x_3 x_6 + t_1x_1x_2x_3x_6\frac{x^{\mathbf{a}}}{x_b} = x_1 x_2 x_3 x_6 + t_1 x_1^3 x_6\,\,\, ,$$
\noindent
and a homogeneous perturbation of the monomial $x_1x_2x_3x_5$ to

$$x_1 x_2 x_3 x_5 + t_1 x_1 x_2 x_3x_5 \frac{x^{\mathbf{a}}}{x_b} = x_1 x_2 x_3 x_5 + t_1 x_1^3x_5\,\,\, ,$$
 \noindent
 with $\mathbf{a} = (2,0,0,0,0,0,0)$ and hence $x^{\bf{a}} = x_1^2$, and $x_b = x_2x_3$. The other three monomials of the Stanley-Reisner ideal are unchanged. The cases $a = \{2\}$ and $a=\{3\}$ give rise to similar perturbations, with parameters $t_2$ and $t_3$, respectively. In the case $a = \{4\}$, the tetrahedron gives rise to 11 dimensions of $T^1$, one for each $b \subset \{ 1,2,3,5 \}$ with $|b| \geq 2$.  The case $a = \{7\}$ is similar. In each of the two cases $a = \{5\}$ and $a = \{6\}$, the suspension of a triangle gives 5 different choices of $b$ contributing non-trivially to $T^1$. In addition, the 9 triangles give rise to $9\times 4$ permutations, and the 6 quadrangles give rise to $6\times 2$ perturbations. Note that each triangle gives rise to 5 perturbations and not 4 as stated in the table~\ref{table:T1}. To see this, note that since $T^1$ is $\mathbb{Z}^n$ graded, e.g. the case with $a = \{ 1,4 \}$ and $b = \{2,3,5 \} $ gives two different choices of the vector $\mathbf{a}$ in order for the deformation to be embedded in $\mathbb{P}^6$; $\mathbf{a} = (2,0,0,1,0,0,0)$ or $\mathbf{a} =(1,0,0,2,0,0,0)$ both have support $a = \{ 1,4 \}$. Putting all this together, the dimension of $T^1_X$ is

\begin{figure}
  \begin{center}
  \includegraphics[width=5.8cm]{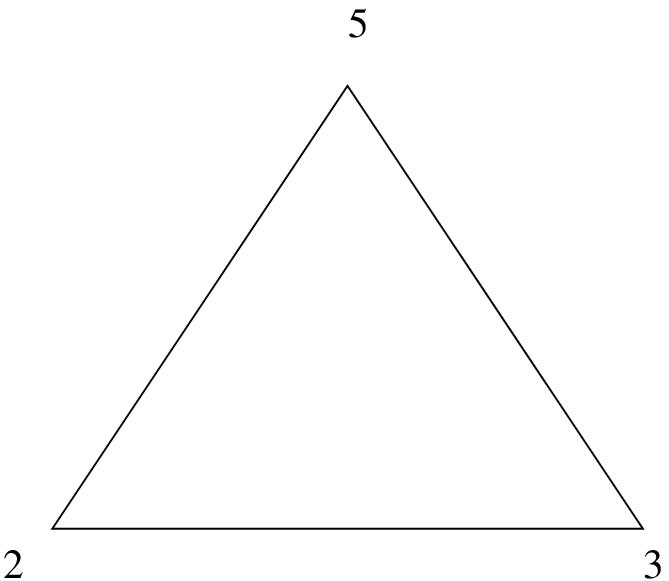}\end{center}
  \caption{The link of the edge $\{1,4 \}$ in $P_7^1$}\label{figure:link14}
\end{figure}
\begin{figure}
  \begin{center}
  \includegraphics[width=4.8cm]{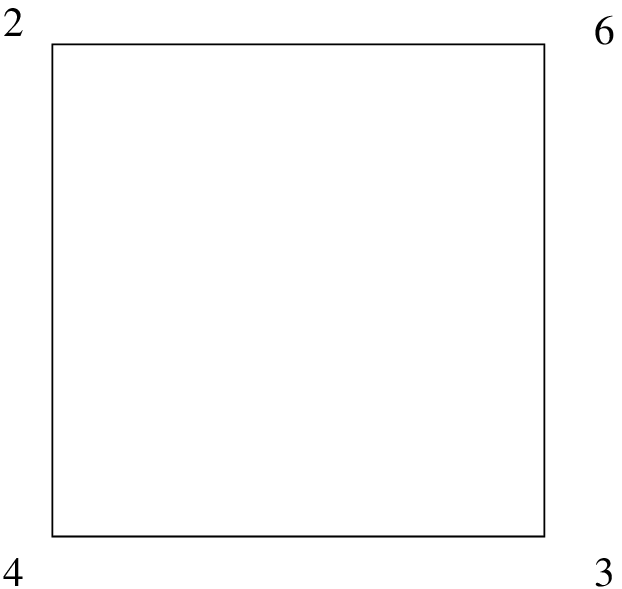}\end{center}
  \caption{The link of the edge $\{1,5 \}$ in $P_7^1$}\label{figure:link15}
\end{figure}

$$3\times 1 + 2\times 11 + 2\times 5 + 9\times 5 + 6\times 2 =92 \,\, .$$
\noindent
This gives 92 parameters $t_1,\ldots, t_{92}$, and the first order deformed ideal $I^1$. The relations between the generators of $I_0$ can be lifted to relations between the generators of $I^1$, and the matrix $M$ lifts to the matrix

\begin{displaymath}M^1 =
\begin{bmatrix}
0  & g_1 & g_2 & l_1 & l_2 \\
-g_1 & 0 & g_3 & l_3 & l_4 \\
-g_2 & -g_3 & 0 & l_5 & l_6 \\
-l_1 & -l_3 & -l_5 & 0 & 0 \\
-l_2 & -l_4 & -l_6 & 0 & 0
\end{bmatrix}\,\,\, ,
\end{displaymath}
\noindent
where $g_1$, $g_2$ and $g_3$ are cubics and $l_1, \ldots, l_6$ are linear forms in the variables $x_1, \ldots, x_7$. This matrix is computed explicitly, and is given in the appendix on page~\pageref{ex1refUttrykk}. The principal pfaffians of $M^1$ give the versal deformation up to all orders.

After a coordinate change we can describe a general fiber $X$ by

\begin{displaymath}\text{rk}
\begin{bmatrix}
x_1 & x_3 & x_5 \\
x_2 & x_4 & x_6
\end{bmatrix}\leq 1 \;\; \text{and}\;\;
\begin{bmatrix}
x_1 & x_3 & x_5 \\
x_2 & x_4 & x_6
\end{bmatrix} \cdot
\begin{bmatrix}
g_3\\
g_2\\
g_1
\end{bmatrix} = 0
\end{displaymath}
\noindent
The first group of equations define the projective cone over the Segre embedding of $\mathbb{P} ^1 \times \mathbb{P}^2$ in $ \mathbb{P}^5$. Call this variety $Y$. There is one singular point on $X_t$, the vertex of the cone $Y$; $P = (0: \cdots : 0:1)$. The singularity is a node. In fact it is locally isomorphic to $x_1x_4 - x_2x_3 = 0$ in $\mathbb{C}^4$. Since $X_t$ is the general fiber in a smooth versal deformation space, $X_0$ cannot be smoothed.

Using the techniques of Section~\ref{section:isolsing}, the intersection of $X$ with the equations $x_3 = x_4 = 0$ gives a smooth surface $S$ containing the point $P$. A crepant resolution $\pi \colon \tilde{X} \rightarrow X$ exists since the only singularity of $X$ is a node, and the plane $S$ passing through the node. Let $\tilde{X}$ be the manifold obtained by blowing up along $S$.

The Macaulay 2 computation on page~\pageref{sideM2code} gives  $\text{dim} T_{A,0}^1 = 86$, and we can compute the evaluation map $T_{A,0}^1\rightarrow T^1_{\mathcal{O}_P}$, which is 0, we have $\text{dim}H^1(\Theta_{\tilde{X}}) = 86$ by Theorem~\ref{th:sequence}.

\section{The triangulation $P^7_2$}\label{ex2section}
In this section we consider $P^7_2$, the second triangulation of
$\mathbb{S}^3$ in Table~\ref{table:pol}. It has Stanley-Reisner ideal

\begin{displaymath}
I_0  = (x_5 x_7, x_1 x_7, x_4 x_5 x_6, x_1 x_2 x_3, x_2 x_3 x_4 x_6)
\end{displaymath}
and the matrix $M$ in the free resolution is

\begin{displaymath}M =
\begin{bmatrix}
0            & 0              & 0             & -x_4x_6 & x_1 \\
0            & 0              & x_2x_3  & 0             & -x_5\\
0            & -x_2x_3  & 0             & x_7        & 0   \\
x_4x_6 & 0              & -x_7        & 0            & 0   \\
-x_1      & x_5          & 0             & 0             & 0
\end{bmatrix}
\end{displaymath}

As in the previous section, we compute the module $T^1_X$, i.e. the first order embedded deformations, of the Stanley-Reisner scheme $X$ of the complex $K:=P^7_2$ by considering the links of the faces of the complex. Various combinations of $a,b \in \{ 1,\ldots ,7 \}$, with $b \subset [ \text{link}(a,K)]$ a subset of the vertex set and $a$ a face of $K$, contribute to $T^1_X$.

The geometric realization $| \text{link}(i,K) |$ of the link $\text{link} (i,K)$ of the vertex $\{i\}$ is the boundary of a cyclic polytope for $i = 2,3,4$ and $6$. For $i = 1$ and $5$, the geometric realization $| \text{link}(i,K) |$ is the suspension of a triangle, and for $i = 7$, $|\text{link}(i,K)|$ is a tetrahedron.

The links of the edges give rise to 8 triangles, 7 quadrangles and 4 pentagons. Hence, the dimension of $T^1_X$ is $4\times 1 + 2\times 5 + 1\times 11 + 8\times 5 + 7\times2 = 79$.

We compute the first order ideal $I_t$ perturbed by 79 parameters. The matrix $M$ lifts to the matrix

\begin{displaymath} M^{1} =
\begin{bmatrix}
0    & -g   & q_1     & -q_2  & x_1    \\
g    & 0    & q_3     & -q_4   & -x_5   \\
-q_1 & -q_3 & 0       & x_7   & t_{38} \\
q_2  &  q_4 & -x_7    & 0     & -t_{33}\\
-x_1 & x_5  & -t_{38} & t_{33}& 0      \\
\end{bmatrix}\,\,\, ,
\end{displaymath}
where $g$ is a cubic and $q_1,\ldots, q_4$ are quadrics in the
variables $x_1,\ldots, x_7$. The exact expressions for these quadrics
are given in the appendix on page~\pageref{ex2refUttrykk}. The versal
deformation space up to all orders is given by the principal pfaffians
of the matrix above. Let $X$ be a general fiber of this family.

\newtheorem{CompleteIn}{Lemma}[section]
\begin{CompleteIn}
The variety $X$ is a complete intersection.\label{lemma:complLemmaet}
\end{CompleteIn}

\begin{proof}
The lower right corner of the matrix $M^1$ is

\begin{displaymath}
W =
\begin{bmatrix}
0 & k\\
-k & 0\\
\end{bmatrix}\,\,\, ,
\end{displaymath}
\noindent
where $k$ is a constant. The matrix $M^1$ can be written on the form

\begin{displaymath}
\begin{bmatrix}
U      & V\\
-V^{T}  & W\\
\end{bmatrix}=
\begin{bmatrix}
I   & VW^{-1}\\
0 & I\\
\end{bmatrix}
\begin{bmatrix}
U + VW^{-1}V^{T}  & 0\\
0  & W\\
\end{bmatrix}
\begin{bmatrix}
I   & 0\\
(VW^{-1})^{T} & I\\
\end{bmatrix}
\end{displaymath}

Now, the ideal of principal pfaffians can be computed as the principal pfaffians of the matrix at the right center above,
hence two of the generators are now zero. The remaining three pfaffians are the elements of the $3\times 3$ matrix $U' = U + VW^{-1}V^{T}$ multiplied by a constant. Hence, the variety is a complete intersection in $\mathbb{P}^{6}$.
\end{proof}

The five principal pfaffians can be reduced to three, two quadrics and a cubic. The smoothness of a general fiber can be checked for a good choice of the $t_i$ using a computer algebra
package like Macaulay 2~\cite{M2} or Singular~\cite{GPS05}. We will compute the cohomology of the smooth fiber, following the exposition in R\o dland's thesis~\cite{roedland}. The
following lemma will be useful.

\newtheorem{exact2}[CompleteIn]{Lemma}
\begin{exact2}
There is an exact sequence

\begin{displaymath}
\XY
 \xymatrix@1{
0\ar[r] & \mathcal{O}_{\mathbb{P}^6}(-7) \ar[r]^-{v} &
2\mathcal{O}_{\mathbb{P}^6}(-5) \oplus \mathcal{O}_{\mathbb{P}^6}(-4) \ar[r]^-{U'} &   }
\end{displaymath}
\begin{displaymath}
\XY
 \xymatrix@1{
 2\mathcal{O}_{\mathbb{P}^6}(-2) \oplus
\mathcal{O}_{\mathbb{P}^6}(-3)
\ar[r]^-{v^T}&
 \mathcal{O}_{\mathbb{P}^6}
\ar[r] & \mathcal{O}_{X} \ar[r] & 0  }
\end{displaymath}
where $X$ is the general fiber and $v$ is the column vector with entries the three principal pfaffians of $U^{\prime}$.\label{lemma:exsequence2}\end{exact2}

Since $X$ is Calabi-Yau (see Theorem~\ref{theorem:cala}), we know that $h^{1,0}(X)= h^{2,0}(X) = 0$. We now proceed to find the
remaining Hodge numbers of $X$. Let $ \mathcal{J} := \text{ker} ( i^{\sharp}
: \mathcal{O}_{\mathbb{P}^6} \rightarrow i_* \mathcal{O}_{X})$ denote the
ideal sheaf.

\newtheorem{ideal2}[CompleteIn]{Lemma}
\begin{ideal2}
There is a free resolution
\begin{displaymath}
\XY
 \xymatrix@1{
  0 \ar[r] & \mathcal{G} \ar[r]^-{U' \cdot}& \mathcal{H}
  \ar[r]^-{\Phi} & \mathcal{K} \ar[r]^-{v^{\otimes 2}}
  & \mathcal{J}^2_X \ar[r]& 0}
\end{displaymath}
where the sheaves $\mathcal{G}$, $\mathcal{H}$ and $\mathcal{K}$ are given by

\begin{equation*}\mathcal{G} = \mathcal{O}_{\mathbb{P}^6}(-9) \oplus  2\mathcal{O}_{\mathbb{P}^6}(-10) \end{equation*}
\begin{equation*}\mathcal{H} =2\mathcal{O}_{\mathbb{P}^6}(-6) \oplus
4\mathcal{O}_{\mathbb{P}^6}(-7) \oplus 2\mathcal{O}_{\mathbb{P}^6}(-8)\end{equation*}
\begin{equation*}
\mathcal{K} = 3\mathcal{O}_{\mathbb{P}^6} (-4)\oplus 2
  \mathcal{O}_{\mathbb{P}^6} (-5) \oplus \mathcal{O}_{\mathbb{P}^6}(-6)\end{equation*}
\noindent
The elements of $\mathcal{G}$, $\mathcal{H}$ and $\mathcal{K}$ are
regarded as $5 \times 5$-matrices that are skew-symmetric
matrices, general matrices modulo the identity matrix (or with zero
trace), and symmetric matrices respectively. The three maps are

$$U'\cdot : A \mapsto U'A - I/3 \cdot \text{trace}(U'A)\,\,\, ,$$
$$\Phi: B \mapsto BU' + (U')^TB^T\,\,\, ,$$
\noindent
and $$v^{\otimes 2}: C \mapsto v^TCv\,\,\, .$$
\noindent
If viewed modulo the identity, the last term of the map $U'$ may be dropped.

\label{lemma:idealsheaf2}
\end{ideal2}

\begin{proof}
All the compositions are clearly zero, hence it remains to prove that the kernels are contained in the images. The last map, $v^{\otimes 2}$, is surjective, because $\mathcal{J}^2_X$ is generated by the elements of $v^T v$, i.e. the elements $m_{ij} = v_i v_j$ for $i \leq j$.

The relations on the $m_{ij}$ are no other than $m_{ij} = m_{ji}$ and $m_{ij} v_k = m_{jk}v_i$, hence the sequence is exact at $\mathcal{K}$. Next, consider the map $\Phi: \mathcal{H} \rightarrow \mathcal{K}$. We have

\begin{displaymath}
\Phi(B) = BU' + (U')^TB^T = BU' - U' B^T
\end{displaymath}
and hence

\begin{equation*}\Phi(B) = 0\end{equation*}
\begin{equation*}BU' = U' B^T\end{equation*}
\begin{equation*}U' B^T v= 0\,\,\, .\end{equation*}
For some $b$ we have (by Lemma~\ref{lemma:exsequence2})
\begin{equation*}B^T v = bv\end{equation*}
\begin{equation*}(B^T - I b) v = 0\,\,\, ,\end{equation*}
and for some matrix $W$ we have (by Lemma~\ref{lemma:exsequence2} again)
\begin{equation*}B^T - I b = W U' \end{equation*}
\begin{equation*}B =  - U' W^T + bI\,\,\, .\end{equation*}
\noindent
Since $B = -U' W^T + bI$ equals $-U' W^T $ modulo $I$, we have proved that the sequence is exact at $\mathcal{H}$.

Consider the map $U'\cdot : \mathcal{G} \rightarrow \mathcal{H}$. The image of a skew-symmetric matrix $A$ is zero if and only if $U'A = bI$. However, skew-symmetry yields rank less than $3$. So for A to map to zero, we must have $U' A = 0$. However, using the exact sequence of Lemma~\ref{lemma:exsequence2}, we have that $ U' A = 0 \Rightarrow A = v w^T$ for some vector $w$. However, $A = -A^T = -wv^T$, so $U' A = 0 \Rightarrow U' w = 0 \Rightarrow w = gv \Rightarrow A = g v v^t$. However, for $A = g v v^T$ to be skew-symmetric, $g$ must be zero, making $A = 0$. Hence, the map is injective.
\end{proof}

\newtheorem{cohom2igjen}[CompleteIn]{Proposition}
\begin{cohom2igjen}
The Hodge numbers are

\[ h^{1,1}(X)=  1\,\,  \text{and}\,\, h^{1,2}(X) = 73 \,\,\, , \]
\noindent
where $h^{1,1}(X) := \text{dim}\,H^1(\Omega_{X}) $ and $h^{1,2}(X) := \text{dim}\,H^2(\Omega_{X})$.\label{proposition:hodgeEx2}
\end{cohom2igjen}

\begin{proof}
First, we know that $H^*(\mathcal{O}_{\mathbb{P}^6}(-r)) = 0$ for $0< r<7$. Second, if we have a resolution $0 \rightarrow A_n \rightarrow \cdots \rightarrow A_0 \rightarrow I$ where $H^*(A_i) = 0$ for $i<n$, then $H^p(I) \cong H^{p+n}(A_n)$, and third, $h^6(\mathcal{O}_{\mathbb{P}^6}(-r -7)) = h^0(\mathcal{O}_{\mathbb{P}^6}(r)) = \binom{r + 6}{6}$.

Using these facts on the resolution of $\mathcal{O}_X(-1)$ (twist the entire sequence of Lemma~\ref{lemma:exsequence2} by $-1$) we get $h^p(\mathcal{O}_X(-1)) = h^{p+3}(\mathcal{O}_{\mathbb{P}^6}(-8))$ which is $7$ for $p = 3$, otherwise zero. Using these results and the cohomology of $\mathcal{O}_X$ on the long exact sequence of

\begin{displaymath}
\XY
 \xymatrix@1{
0 \ar[r] & \Omega_{\mathbb{P}^6}| X \ar[r] & 7 \mathcal{O}_X(-1) \ar[r]& \mathcal{O}_X \ar[r] &0 }
\end{displaymath}
we find that $h^0(\Omega_{\mathbb{P}^6} | X) = h^2(\Omega_{\mathbb{P}^6} | X) = 0$, $h^1(\Omega_{\mathbb{P}^6} | X) = h^0(\mathcal{O}_X) = 1$, and $h^3(\Omega_{\mathbb{P}^6} | X) = h^3(7\mathcal{O}_X(-1))- h^3(\mathcal{O}_X) = 48$.

For the ideal sheaf $\mathcal{J}_X$, the above results and the resolution~\ref{lemma:exsequence2} give $h^p(\mathcal{J}_X) = h^{p+2}(\mathcal{O}_{\mathbb{P}^6}(-7))$ which is $1$ for $p=4$, otherwise zero. For $\mathcal{J}_X^2$, the resolution splits into two short exact sequences

\begin{displaymath}
\XY
 \xymatrix@1{0 \ar[r] & \mathcal{G} \ar[r]^-{U' \cdot}& \mathcal{H}
\ar[r] &\text{Im}(\Phi)\ar[r] & 0}
\end{displaymath}
and
\begin{displaymath}
\XY
 \xymatrix@1{
 0\ar[r] & \text{Im}(\Phi)\ar[r] & \mathcal{K} \ar[r]^-{v^{\otimes 2}} & \mathcal{J}^2_X \ar[r]& 0}
\end{displaymath}
From the second, we get $h^p(\mathcal{J}_X^2) = h^{p+1}(\text{Im}(\Phi))$. From the first, the only non-zero part of the long exact sequence is

\begin{displaymath}
\XY
 \xymatrix@1{ 0 \ar[r] & H^5(\text{Im}(\Phi)) \ar[r] & H^6(\mathcal{G})\ar[r] & H^6(\mathcal{H})\ar[r] & H^6(\text{Im}(\Phi))\ar[r] & 0}
 \end{displaymath}
 This makes $h^4(\mathcal{J}_{X}^2) - h^5(\mathcal{J}_{X}^2) = h^5(\text{Im}(\Phi)) - h^6(\text{Im}(\Phi)) = h^6(\mathcal{G}) - h^6(\mathcal{H}) = 2h^6(\mathcal{O}_{\mathbb{P}^6}(-9)) + h^6(\mathcal{O}_{\mathbb{P}^6}(-10)) - 2h^6(\mathcal{O}_{\mathbb{P}^6}(-6)) - 4h^6(\mathcal{O}_{\mathbb{P}^6}(-7)) - 2h^6(\mathcal{O}_{\mathbb{P}^6}(-8)) = 2\cdot 28 + 84 - 2\cdot 0 -4\cdot 1 - 2\cdot 7 = 122$. Since the variety is smooth, we have a short exact sequence

\begin{displaymath}
\XY
 \xymatrix@1{0\ar[r] & \mathcal{J}_X^2\ar[r] &  \mathcal{J}_X\ar[r] & \mathcal{N}_X^{\vee} \ar[r] & 0 }
 \end{displaymath}
and another sequence

\begin{displaymath}
\XY
 \xymatrix@1{0\ar[r] & \mathcal{N}_X^{\vee}\ar[r] &  \Omega_{\mathbb{P}^6}| X \ar[r] & \Omega_X \ar[r] & 0 }
\end{displaymath}
Note that $\mathcal{N}_X^{\vee}$ is a sheaf on $X$, hence $h^p(\mathcal{N}_X^{\vee}) = 0$ for $p>3 = \text{dim}X$. Entering this into the long exact sequences of the first of the two resolutions above, we get $h^5(\mathcal{J}_X^2)= 0$ as both $h^4(\mathcal{N}_X^{\vee})= 0$ and $h^5(\mathcal{J}_X)=0$. Hence we have $h^4(\mathcal{J}_X^2) = 122$. In addition, we get $h^2(\mathcal{N}^{\vee}_X)=0$ and $h^3(\mathcal{N}^{\vee}_X) = 121$. The long exact sequence of the second resolution above yields $h^1(\Omega_X) = 1$ and $h^2(\Omega_X) = 121 -48 = 73$.
\end{proof}

Using Singular~\cite{GPS05} (or any other programming language) we can compute the group of automorphisms of the simplicial complexes. It is a subgroup of $S_7$ and is computed by checking which permutations preserve the maximal facets.
The automorphism group $\text{Aut}(P^7_2) \cong D_{4}$ of the complex $P^7_2$ is the dihedral group on 8 elements, i.e. $\mathbb{Z} \ast \mathbb{Z} $
modulo the relations $a^{2} = b^{2} = 1$, $(ab)^{4} = 1$. It is
generated by the elements

\begin{displaymath} a =
\begin{pmatrix}
1 & 2 & 3 & 4 & 5 & 6 & 7\\
1 & 2 & 3 & 6 & 5 & 4 & 7\\
\end{pmatrix}\,\,
\end{displaymath}
\noindent
and
\begin{displaymath} b =
\begin{pmatrix}
1 & 2 & 3 & 4 & 5 & 6 & 7\\
5 & 4 & 6 & 2 & 1 & 3 & 7\\
\end{pmatrix}\, .
\end{displaymath}
\noindent
This group action on the versal family has 22 orbits. Hence, we have an invariant family with 22 parameters, $s_1,\ldots, s_{22}$.

\section{The triangulation $P^7_3$}\label{ex3section}
In this section we consider the third example, $P^7_3$, from Table~\ref{table:pol}. It has Stanley-Reisner ideal

\begin{displaymath}
I_0 = (x_6x_7,x_4x_5,x_1x_2x_3)\,\, ,
\end{displaymath}
and the syzygy matrix is

\begin{displaymath}M =
\begin{bmatrix}
0          & -x_1x_2x_3  & x_4x_5\\
x_1x_2x_3 & 0          & -x_6x_7 \\
-x_4x_5    & x_6x_7    & 0\\
\end{bmatrix}\,\,\, ,
\end{displaymath}
As in sections~\ref{ex1section} and \ref{ex2section}, we compute the module $T^1_X$, i.e. the first order embedded deformations, of the Stanley-Reisner scheme $X$ of the complex $K:=P^7_3$ by considering the links of the faces of the complex. Various combinations of $a,b \in \{ 1,\ldots ,7 \}$, with $b \subset [ \text{link}(a,K)]$ a subset of the vertex set and $a$ a face of $K$, contribute to $T^1_X$.

The geometric realization $| \text{link}(1,K) |$ of the link of the vertex $\{1\}$ in $K$ is an octahedron, and is illustrated in figure~\ref{figure:QuadrangleSusp}. The links of the vertices $\{2\}$ and $\{3\}$ are similar.

\begin{figure}
  \begin{center}
  \includegraphics[width=5.8cm]{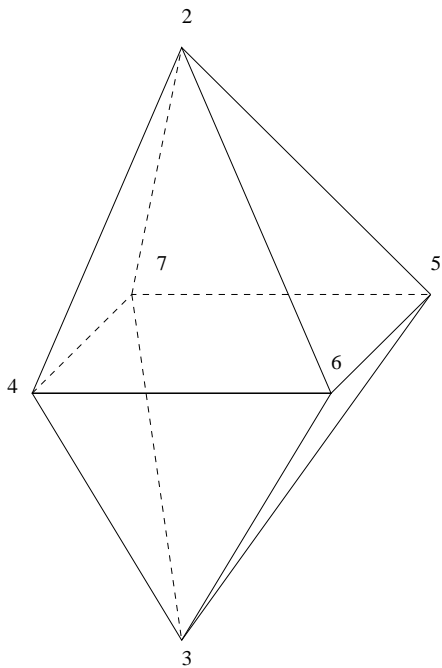}\end{center}
  \caption{The link of the vertex $\{1 \}$ in $P_3^7$}\label{figure:QuadrangleSusp}
\end{figure}

The links of the vertices $\{4\}$, $\{5\}$, $\{6\}$ and $\{7\}$ are the suspension of a triangle. In addition, the links of the edges give rise to 15 quadrangles and 4 triangles. Putting all this together, the dimension of $T^1_X$ is $3\times 3 + 4\times 5 + 15\times 2 + 4\times 5 = 79$.

The perturbed ideal is generated by the elements of the vector

\begin{displaymath}f^1 =
\begin{bmatrix}
 q_2\\
 q_1\\
 g\\
\end{bmatrix}
\end{displaymath}
where the exact expressions for $g$, $q_1$ and $q_2$ are given in the appendix on page~\pageref{ex3refUttrykk}. The syzygy matrix lifts to

\begin{displaymath}M^1 =
\begin{bmatrix}
0    & -g    & q_1\\
g   & 0    & -q_2\\
-q_1 & q_2 & 0  \\
\end{bmatrix}\,\,\, .
\end{displaymath}
\noindent
Thus the ideal generated by $f^1$ gives the versal family up to
all orders. The general fiber $X$ is given by $g=0$, $q_1 =
0$ and $q_2 = 0$, the intersection of $2$ quadrics and a cubic in
$\mathbb{P}^6$, a complete intersection. The smoothness can be checked for a good choice of the $t_i$ using
Singular~\cite{GPS05}. The following lemma will be useful.

\newtheorem{exact3}{Lemma}[section]
\begin{exact3}
The sequence

\begin{displaymath}
\XY
 \xymatrix@1{
0\ar[r] & \mathcal{O}_{\mathbb{P}^6}(-7) \ar[r]^-{F} &
2\mathcal{O}_{\mathbb{P}^6}(-5) \oplus \mathcal{O}_{\mathbb{P}^6}(-4) }
\end{displaymath}
\begin{displaymath}
\XY
 \xymatrix@1{
\ar[r]^-{R^1} & 2\mathcal{O}_{\mathbb{P}^6}(-2)\oplus
\mathcal{O}_{\mathbb{P}^6}(-3) \ar[r]^-{F^t}& \mathcal{O}_{\mathbb{P}^6}
\ar[r] & \mathcal{O}_{X} \ar[r] & 0 }
\end{displaymath}
is exact, where $F$ and $R^1$ are given above, and values for the $t_i$'s are chosen.\label{lemma:exsequence3}\end{exact3}

Since $X$ is Calabi-Yau, we know that $h^{1,0}(X)= h^{2,0}(X) = 0$. The following Proposition can be proved in a similar manner as Proposition~\ref{proposition:hodgeEx2} in the previous section.

\newtheorem{cohom3igjen}[exact3]{Proposition}
\begin{cohom3igjen}The Hodge numbers are

$$h^{1,1}(X) = 1  \,\,\text{and}\,\, h^{1,2}(X) = 73\,\, ,$$
\noindent
where $h^{1,1}(X) := \text{dim}\,H^1(\Omega_{X}) $ and $h^{1,2}(X) := \text{dim}\,H^2(\Omega_{X})$.
\end{cohom3igjen}

Using Singular~\cite{GPS05} (or any other programming language) we can compute the group of automorphisms of the simplicial complexes. It is a subgroup of $S_7$ and is computed by checking which permutations preserve the maximal facets.
The automorphism group $\text{Aut}(P^7_3)$ of the complex $P^7_3$ is $D_{4}\times D_{3}$, where $D_{4}$ is
the dihedral group on 8 elements, i.e. $\mathbb{Z} \ast \mathbb{Z} $
modulo the relations $a^{2} = b^{2} = 1$, $(ab)^{4} = 1$.  The group
$D_{3}$ is the dihedral group on 6 elements, i.e.  $\mathbb{Z} *
\mathbb{Z}$ modulo the relations $c^{3}= 1$, $d^{2}= 1$ and $cdc = d$.
The group $\text{Aut}(P^7_3)$ is generated by the permutations

\begin{displaymath} a =
\begin{pmatrix}
1 & 2 & 3 & 4 & 5 & 6 & 7\\
1 & 2 & 3 & 5 & 4 & 6 & 7\\
\end{pmatrix}
\end{displaymath}
\begin{displaymath} b =
\begin{pmatrix}
1 & 2 & 3 & 4 & 5 & 6 & 7\\
1 & 2 & 3 & 6 & 7 & 4 & 5\\
\end{pmatrix}
\end{displaymath}
\begin{displaymath} c =
\begin{pmatrix}
1 & 2 & 3 & 4 & 5 & 6 & 7\\
2 & 3 & 1 & 4 & 5 & 6 & 7\\
\end{pmatrix}
\end{displaymath}
\begin{displaymath} d =
\begin{pmatrix}
1 & 2 & 3 & 4 & 5 & 6 & 7\\
3 & 2 & 1 & 4 & 5 & 6 & 7\\
\end{pmatrix}
\end{displaymath}
and it has order $48$. If we consider the subfamily invariant under this group action, the original 79 parameters reduce to 10.

\section{The triangulation $P^7_4$}\label{ex4section}

In this section we consider $P^7_4$, the fourth triangulation of $\mathbb{S}^3$ from Table~\ref{table:pol}. The Stanley-Reisner ideal of this triangulation is

$$ I_{0} = (x_5x_7,x_1x_2x_5,x_1x_2x_6,x_3x_4x_6, x_3x_4x_7)\,\,\, . $$
\noindent
in the polynomial ring $R = \mathbb{C}[x_1, \ldots ,x_7]$. Let $A = R /I_0$ be the Stanley-Reisner ring of $I_0$. The minimal free resolution the Stanley-Reisner ring is

\begin{displaymath}
\XY
 \xymatrix@1{
0\ar[r] & R\ar[r]^{f} & R^5 \ar[r]^{M} & R^5 \ar[r]^{f^T}
& R \ar[r] & A \ar[r] & 0
 }\,\,\, ,
\end{displaymath}
\noindent
where $f$ and $M$ are given by

$$f =
\begin{bmatrix}
x_5 x_7 \\
x_1 x_2 x_5 \\
x_1 x_2 x_6 \\
x_3 x_4 x_6 \\
x_3 x_4 x_7 \\
\end{bmatrix} \,\,\, ,$$
\noindent
\begin{displaymath}M =
\begin{bmatrix}
0       & 0    & x_3x_4 & -x_1x_2 &  0\\
0       & 0    & 0      & x_7     & -x_6\\
-x_3x_4 & 0    & 0      & 0       & x_5\\
x_1x_2  & -x_7 & 0      & 0       &  0\\
0       & x_6 & -x_5    & 0       &  0\\
\end{bmatrix}
\end{displaymath}
Computing as in the previous sections, we find the module $T^1_X$, i.e. the embedded versal deformations, of the Stanley-Reisner scheme $X$ of the complex $K:=P^7_4$ by considering the links of the faces of the complex. Various combinations of $a,b \in \{ 1,\ldots ,7 \}$, with $b \subset [ \text{link}(a,K)]$ a subset of the vertex set and $a$ a face of $K$, contribute to $T^1_X$.

The geometric realization $|\text{link}(i,K) |$ of the link of the vertex $\{i\}$ in $K$ for $i = 1,2,3$ and $4$ is the boundary of the cyclic polytope. For $i = 5$ and $7$, the link is the suspension of the triangle, and for $i = 6$, the link is a octahedron. In addition, the links of edges give rise to 4 pentagons, 8 quadrangles and 4 triangles. Putting all this together, the dimension of $T^1_X$ is $4\times 1 + 2\times 5 + 1\times  3 + 6\times5 + 10\times2  = 67$. A general fiber $X$ will be given by the principal pfaffians of the
matrix

\begin{displaymath}M^{1} =
\begin{bmatrix}
 0   &  q_1 &  q_2 &  q_3 & q_4\\
-q_1 &  0   &  l_1 &  l_2 & l_3\\
-q_2 & -l_1 &  0   &  l_4 & l_5\\
-q_3 & -l_2 & -l_4 &  0   & l_6\\
-q_4 & -l_3 & -l_5 & -l_6 & 0\\
\end{bmatrix}
\end{displaymath}
where $q_1, \ldots, q_4$ are general quadrics and $l_1, \ldots, l_6$ are linear terms. The exact expressions for the polynomials in this matrix is given in the appendix. The smoothness of the general fiber can be checked using computer algebra software.

\newtheorem{exact4}{Lemma}[section]
\begin{exact4}
The following sequence

\begin{displaymath}
\XY
 \xymatrix@1{
0\ar[r] & \mathcal{O}_{\mathbb{P}^6}(-7) \ar[r]^-{F} &
\mathcal{O}_{\mathbb{P}^6}(-5)\oplus 4\mathcal{O}_{\mathbb{P}^6}(-4) }
\end{displaymath}
\begin{displaymath}
\XY
 \xymatrix@1{
\ar[r]^-{M^1} & \mathcal{O}_{\mathbb{P}^6}(-2)\oplus
4\mathcal{O}_{\mathbb{P}^6}(-3) \ar[r]^-{F^t} &
\mathcal{O}_{\mathbb{P}^6}
\ar[r] & \mathcal{O}_{X} \ar[r]& 0 }
\end{displaymath}
is exact, where $F$ is the vector with entries the pfaffians
of the matrix $M^1$ mod $t^2$.\label{lemma:exsequence4}
\end{exact4}
Since $X$ is Calabi-Yau, we know that $h^{1,0}(X)= h^{2,0}(X) =
0$. The following Proposition can be proved in a similar manner as Proposition~\ref{proposition:hodgeEx2}.

\newtheorem{cohom4}[exact4]{Proposition}
\begin{cohom4}The Hodge numbers are

$$h^{1,1}(X) = 1\,\,\text{and}\,\, h^{1,2}(X) = 61\,\, ,$$
\noindent
where $h^{1,1}(X):= \text{dim}\,H^1(\Omega_{X})$ and $h^{1,2}(X) := \text{dim}\,H^2(\Omega_{X})$.\label{proposition:ex4cohom}
\end{cohom4}

The group $\text{Aut}(P^7_4)$ of automorphisms of the complex $P^7_4$ is $D_4 $, the dihedral group of 8 elements. It is generated by the permutations

\begin{displaymath} a =
\begin{pmatrix}
1 & 2 & 3 & 4 & 5 & 6 & 7\\
2 & 1 & 3 & 4 & 5 & 6 & 7\\
\end{pmatrix}\,\,
\end{displaymath}
and
\begin{displaymath} b =
\begin{pmatrix}
1 & 2 & 3 & 4 & 5 & 6 & 7\\
3 & 4 & 1 & 2 & 7 & 6 & 5\\
\end{pmatrix}\, .
\end{displaymath}
\noindent
This group action on the versal family has 20 orbits. Hence, we have an invariant family with 20 parameters, $s_1,\ldots, s_{20}$.

\section{The triangulation $P^7_5$}\label{ex5section}

In this section we consider the fifth example, $P^7_5$, from Table~\ref{table:pol}. It has Stanley-Reisner ideal

$$I_0 = (x_1x_3x_5, x_1x_3x_6, x_1x_4x_6, x_2x_4x_6, x_2x_4x_7,
x_2x_5x_7, x_3x_5x_7)\,\,\, ,$$
and the Syzygy matrix is

$$M =
\begin{bmatrix}
0    & 0    & 0    & x_7  & -x_6 & 0    & 0   \\
0    & 0    & 0    & 0    & x_5  & -x_4 & 0   \\
0    & 0    & 0    & 0    & 0    & x_3  & -x_2\\
-x_7 & 0    & 0    & 0    & 0    & 0    & x_1 \\
x_6  & -x_5 & 0    & 0    & 0    & 0    & 0   \\
0    & x_4  & -x_3 & 0    & 0    & 0    & 0   \\
0    & 0    & x_2  & -x_1 & 0    & 0    & 0   \\
\end{bmatrix}\,\,\, .$$
\noindent

Computing as in the previous sections, we find the module $T^1_X$, i.e. the first order embedded deformations, of the Stanley-Reisner scheme $X$ of the complex $K:=P^7_5$ by considering the links of the faces of the complex. Various combinations of $a,b \in \{ 1,\ldots ,7 \}$, with $b \subset [ \text{link}(a,K)]$ a subset of the vertex set and $a$ a face of $K$, contribute to $T^1_X$.

The geometric realization $| \text{link}(i,K) |$ of the link $\text{link} (i,K)$ of a vertex $\{i\}$ is the boundary of a cyclic polytope for $i = 1,\ldots ,7$. We also consider links of one dimensional faces. In 7 cases the geometric realization is a triangle, and in 7 cases the link is a quadrangle. Putting all this together, the dimension of $T^1_X$ is $7\times 1 + 7\times 5 + 7\times 2 = 56$. The full family is displayed in the appendix.

The matrix $M$ lifts to the matrix

$$M^{1} =
\begin{bmatrix}
0    & l_1    & l_2    & x_7     & -x_6    & -l_3   & -l_4   \\
-l_1 & 0      & l_5    & l_6     & x_5     & -x_4   & -l_7   \\
-l_2 & -l_5   & 0      & l_8     & l_9     & x_3    & -x_2  \\
-x_7 & -l_6   & -l_8   & 0       & l_{10}  & l_{11} & x_1   \\
x_6  & -x_5   & -l_9   & -l_{10} & 0       & l_{12} & l_{13}\\
l_3  &  x_4   & -x_3   & -l_{11} & -l_{12} & 0      & l_{14}\\
l_4  &  l_7   & x_2    & -x_1    & -l_{13} & -l_{14} & 0     \\
\end{bmatrix}\,\,\, ,$$
\noindent
where $l_1,\ldots,l_{14}$ are linear forms, whose exact expressions are given in the appendix on page~\pageref{ex5refUttrykk}. The general fiber $X$ is a degree 14 Calabi-Yau 3-fold. The following Proposition can be proved in a similar manner as Proposition \ref{proposition:hodgeEx2}.

\newtheorem{hodge5}{Proposition}[section]
\begin{hodge5}The Hodge numbers are

$$h^{1,1}(X) = 1\,\, \text{and}\,\, h^{1,2}(X) = 50\,\, ,$$
\noindent
where $h^{1,1}(X) := \text{dim}\,H^1(\Omega_{X}) $ and $h^{1,2}(X) = \text{dim}\,H^2(\Omega_{X})$.
\end{hodge5}

The automorphism group of the complex $P^7_5$, is $\text{Aut}(P^7_5) \cong D_7$. It is
generated by the permutations

\begin{displaymath}
a =
\begin{pmatrix}
1 & 2 & 3 & 4 & 5 & 6 & 7\\
2 & 3 & 4 & 5 & 6 & 7 & 1\\
\end{pmatrix}
\end{displaymath}
\begin{displaymath}
b =
\begin{pmatrix}
1 & 2 & 3 & 4 & 5 & 6 & 7\\
7 & 6 & 5 & 4 & 3 & 2 & 1\\
\end{pmatrix}
\end{displaymath}
with relations $a^7 = 1$, $b^2 = 1$, $aba = b$. A calculation gives a 5 parameter invariant deformations under the action of this group. We will consider a one-parameter subfamily of this invariant family in Section~\ref{section:roedlandmirror}.

\chapter{The R\o dland and B\"ohm Mirrors}\label{chapter:boehmroedlandMirror}

In this chapter we will describe how to obtain the R\o dland and B\"ohm mirrors from the triangulations we studied in the previous chapter. The R\o dland mirror is obtained from the complex $P^7_5$, and the B\"ohm mirror is obtained from the complex $P^7_4$. They are given by a crepant resolution of a chosen one-parameter subfamily of the invariant family under the action of the automorphism group of $P^7_i$.

\section{The R\o dland  Mirror Construction}\label{section:roedlandmirror}

Consider the case of $P^7_5$ which we studied in Section~\ref{ex5section}, and let $X_0$ be the Stanley-Reisner scheme associated to this complex. studied in Section~\ref{ex5section}. As seen in the previous chapter, the automorphism group of the complex is $D_7$. Recall from the introductory chapter that the automorphism group induces an action on $T^1_{X_0}$, and that the parameters of the versal family correspond to faces and links contributing to $T^1_{X_0}$. The $D_7$ orbits of these are given in table~\ref{table:roedlandTabell}.

\begin{table}
\begin{center}
\begin{tabular}{|c|c|c|c|}
\hline {a} & {b} & Link & $\#$ in $D_7$-orbit\\
\hline
$\{ 1 \}$ & $\{2,7\}$ & cyclic polytope & 7 \\
\hline
$\{ 1,3 \}$ & $\{2,4,7\}$ & triangle & 14\\
\hline
$\{ 1,3 \}$ & $\{2,7\}$ & triangle & 14 \\
\hline
$\{ 1,3 \}$ & $\{ 4,7 \} $ & triangle & 7\\
\hline
$\{ 3,5 \}$ & $\{ 2,4 \} $ & quadrangle & 14\\
\hline
\end{tabular}
\end{center}
\caption{$T^1_{X_0}$ is 56 dimensional for the Stanley-Reisner scheme $X_0$ of $P^7_5$}
\label{table:roedlandTabell}
\end{table}

All the links of vertices are cyclic polytopes, and all these 7 cyclic polytopes are one orbit under the action of this automorphism group. In addition, we have 7 cases where the link of an edge is a triangle, and we have 7 cases where it is a quadrangle.

The invariant parameters $s_1,\ldots ,s_5$ are achieved by equating the parameters $t_i$ corresponding to the same orbit under the action of the automorphism group on $T^1_{X_0}$. Consider one of the invariant parameters corresponding to the links of edges being triangles, $s := t_{24} = t_{25} = t_{29} = t_{33} = t_{35} = t_{38} = t_{40}$, the one with 7 elements in the orbit, and set the other ones to $0$. In this case, the matrix $M^1$ will reduce to

$$\begin{bmatrix}
0      & 0        & 0        & x_7       & -x_6    & 0        & -sx_4   \\
0      & 0        & sx_7   & 0          & x_5      & -x_4   & 0  \\
0      &  -sx_7 & 0        & sx_5     & 0         & x_3    & -x_2  \\
-x_7 & 0         & -sx_5 & 0          & sx_3    & 0        & x_1   \\
x_6  & -x_5    & 0         & -sx_3   & 0         & sx_1   & 0\\
0      &  x_4    & -x_3    & 0          & -sx_1  & 0         & sx_6\\
sx_4 &  0       & x_2      & -x_1     & 0         & -sx_6   & 0     \\
\end{bmatrix}\,\,\, .$$
Let $X_s$ be the variety generated by the principal pfaffians of this matrix. It is defined by the ideal generated by the polynomials

\begin{equation*}p_1 = \ -x_{{1}}x_{{3}}x_{{5}}+{s}^{2}x_{{6}}{x_{{5}}}^{2}+{s}^{2}{x_{{1}}}^{2}x_{{7}}-sx_{{2}}x_{{3}}x_{{4}} +{s}^{3}x_{{3}}x_{{6}}x_{{7}}\end{equation*}
\begin{equation*}p_2 =  -x_{{1}}x_{{6}}x_{{3}}-sx_{{1}}x_{{2}}x_{{7}}+{s}^{2}{x_{{3}}}^{2}x_{{4}} +{s}^{2}{x_{{6}}}^{2}x_{{5}}+{s}^{3}x_{{1}}x_{{5}}x_{{4}} \end{equation*}
\begin{equation*}p_3 = -x_{{1}}x_{{6}}x_{{4}}+{s}^{2}x_{{3}}{x_{{4}}}^{2}-sx_{{6}}x_{{5}}x_{{7}} \end{equation*}
\begin{equation*}p_4 = {s}^{3}x_{{1}}x_{{4}}x_{{7}}-x_{{2}}x_{{4}}x_{{6}}-x_{{3}}x_{{5}}sx_{{4}} +{s}^{2}x_{{7}}{x_{{6}}}^{2}\end{equation*}
\begin{equation*}p_5 = -x_{{2}}x_{{4}}x_{{7}}+{x_{{4}}}^{2}{s}^{2}x_{{5}}+{s}^{2}x_{{6}}{x_{{7}}}^{2} \end{equation*}
\begin{equation*}p_6 = {x_{{5}}}^{2}{s}^{2}x_{{4}}-x_{{7}}x_{{2}}x_{{5}}-x_{{7}}sx_{{1}}x_{{6}} +{s}^{3}x_{{3}}x_{{4}}x_{{7}} \end{equation*}
\begin{equation*}p_7 = {x_{{7}}}^{2}{s}^{2}x_{{1}}-x_{{3}}x_{{5}}x_{{7}}-sx_{{4}}x_{{5}}x_{{6}}
\end{equation*}
\noindent
This variety has 56 nodes. Choosing the nonzero parameter as one of the other two parameters corresponding to triangles, gives a smooth general fiber, or a general fiber with singular locus of dimension 0 and degree 189, after a Macaulay 2 computation~\cite{M2}.

There is also a natural action of the torus $(\mathbb{C}^*)^{7}$ on $X_0 \subset \mathbb{P}^{6}$ as follows. An element $\lambda = (\lambda_1,\ldots, \lambda_{7}) \in (\mathbb{C}^*)^{7}$ sends a point $(x_1,\ldots ,x_7) $ of $\mathbb{P}^6$ to $(\lambda_1x_1,\ldots , \lambda_7 x_7)$. The subgroup $ \{ (\lambda,\ldots, \lambda) | \lambda \in \mathbb{C}^* \}$ acts as the identity on $\mathbb{P}^6$, so we have an action of the quotient torus $T_6 := (\mathbb{C}^*)^{7}/\mathbb{C}^*$. In order to compute the subgroup $H \subset T_n$ of the quotient torus which acts on this chosen subfamily, consider the diagonal scalar matrix

$$\lambda = \begin{bmatrix}
\lambda_1 & 0 & 0 & 0 & 0 & 0 & 0\\
0 & \lambda_2 & 0 & 0 & 0 & 0 & 0\\
0 & 0 & \lambda_3 & 0 & 0 & 0 & 0\\
0 & 0 & 0 & \lambda_4 & 0 & 0 & 0\\
0 & 0 & 0 & 0 & \lambda_5 & 0 & 0\\
0 & 0 & 0 & 0 & 0 & \lambda_6 & 0\\
0 & 0 & 0 & 0 & 0 & 0 & \lambda_7\\
\end{bmatrix}
$$
which acts on $(x_1,\ldots, x_7)$ by

$$\lambda \cdot (x_1,\ldots, x_7) = (\lambda_1 \cdot x_1, \ldots ,\lambda_7\cdot x_7)\,\,\, .$$
The subgroup acting on $X_s$ is generated by the $\lambda$ with the property that $\lambda \cdot p_i = c_i p_i$ for $i = 1,\ldots 7$, and $c_i$ a constant. From $\lambda p_1 = c_1\cdot p_1$, we obtain the equations

\begin{equation*}\lambda_1\lambda_3\lambda_5 = \lambda_5^2\lambda_6 = \lambda_1^2\lambda_7 = \lambda_2\lambda_3\lambda_4 = \lambda_3\lambda_6\lambda_7 \,\, .\end{equation*}
For convenience, we set $\lambda_1 = 1$, and we get the equations

\begin{equation}\lambda_3\lambda_5 = \lambda_5^2\lambda_6 = \lambda_7 = \lambda_2\lambda_3\lambda_4 = \lambda_3\lambda_6\lambda_7\,\,\,.\label{lambdaequation}\end{equation}
\noindent
Hence we have the following expression for $\lambda_5$, $\lambda_6$ and $\lambda_7$.

\begin{equation}\lambda_5 = \lambda_2\lambda_4\label{lambda5}\end{equation}
\begin{equation}\lambda_6 = \displaystyle\frac{1}{\lambda_3}\label{lambda6}\end{equation}
\begin{equation}\lambda_7 = \lambda_2\lambda_3\lambda_4\label{lambda7}\end{equation}
\noindent
From $\lambda p_2 = c_2\cdot p_2$, we obtain the equations

\begin{equation} \lambda_3\lambda_6 = \lambda_2\lambda_7 = \lambda_3^2\lambda_4 = \lambda_5\lambda_6^2 = \lambda_4\lambda_5\label{lambdastor}\end{equation}
\noindent
Inserting (\ref{lambda5}), (\ref{lambda6}) and (\ref{lambda7}) into (\ref{lambdastor}) gives

\begin{equation} 1 = \lambda_2^2\lambda_3\lambda_4 = \lambda_3^2\lambda_4 = \frac{\lambda_2\lambda_4}{\lambda_3^2} = \lambda_2\lambda_4^2\end{equation}
\noindent
hence $\lambda_3 = \lambda_2^2$ and $\lambda_4 = \lambda_2^{-4}$. From (\ref{lambda5}), (\ref{lambda6}) and (\ref{lambda7}) we now obtain

\begin{equation*}\lambda_5 = \lambda_2^{-3}\end{equation*}
\begin{equation*}\lambda_6 = \lambda_2^{-2}\end{equation*}
\begin{equation*}\lambda_7 = \lambda_2^{-1}\end{equation*}
\noindent
Inserting the expressions for $\lambda_3,\ldots,\lambda_7$ into the equation \ref{lambdaequation}, we find that $\lambda_2^7 = 1$.

 We conclude that the subgroup $H$ acting on $X_s$ is $H = \mathbb{Z}/7\mathbb{Z}$, which acts as $x_i \mapsto \xi^{i-1} x_i$, where $\xi$ is a primitive 7th root of 1. This subfamily with the action of $H$ is used in R\o dland's thesis~\cite{roedland} in order to construct a mirror of the general fiber of the full versal family. This is done by orbifolding. The variety $X_s$ has 56 nodes. These are the only singularities. A small resolution of $Y:= X_s/H$ is constructed, and this is the mirror manifold of the general fiber.

\section{The B\"ohm Mirror Construction}\label{boehmmirrorsection}

Consider the versal family we studied in Section~\ref{ex4section}, where the special fiber $X_0$ is the Stanley-Reisner scheme of the simplicial complex labeled $P_4^7$ in Table~\ref{table:pol}. Proposition~\ref{proposition:ex4cohom} states that the Hodge numbers are $h^{1,1}(X) = 1$ and $h^{1,2}(X) = 61$ for the smooth general fiber $X$, hence we have $\chi(X) = 2 (h^{1,1}(X) - h^{1,2}(X)) = 2(1-61) = -120$. In Chapter~\ref{chapter:eulerCharBoehm} we verify that the Euler Characteristic of the B\"ohm mirror candidate is $120$ as expected.

The automorphism group of the complex is isomorphic to $D_4$. On the versal family of deformations, there are 20 orbits under the action of this group. The orbits are listed in table~\ref{table:boehmTabell}, and the number of parameters in each orbit is listed. The invariant family is obtained by equating the parameters contained in the same orbit.

\begin{table}
\begin{center}
\begin{tabular}{|c|c|c|c|}
\hline a & b & Link & $\#$ in $D_7$-orbit\\
\hline
$\{ 1 \}$ & $\{3,4\}$ & cyclic polytope & 4\\
\hline
$\{ 5 \}$ & $\{1,2\}$ & suspension of triangle & 2 \\
\hline
$\{ 5 \}$ & $\{3,4\}$ & " & 2 \\
\hline
$\{ 5 \}$ & $\{3,4,6\}$ & " & 2 \\
\hline
$\{ 5 \}$ & $\{3,6\}$ &    " & 4 \\
\hline
$\{ 6 \}$ & $\{1,2\}$ & octahedron & 2 \\
\hline
$\{ 6 \}$ & $\{5,7\}$ & " & 1 \\
\hline
$\{ 1,2 \}$ & $\{3,4,7\}$ & triangle & 4 \\
\hline
$\{ 1,5 \}$ & $\{3,4,6\}$ & " &  4\\
\hline
$\{ 1,5 \}$ & $\{3,4,6\}$ & " & 4 \\
\hline
$\{ 1,2 \}$ & $\{3,4\}$ & " & 2 \\
\hline
$\{  1,2\}$ & $\{3,7\}$ & " &  4\\
\hline
$\{ 1,5 \}$ & $\{3,4\}$ & " & 4 \\
\hline
$\{ 1,5 \}$ & $\{3,6\}$ & " &  8\\
\hline
$\{ 1,6 \}$ & $\{ 3,4 \} $ & quadrangle & 4\\
\hline
$\{ 1,6 \}$ & $\{ 5,7  \} $ & " & 4\\
\hline
$\{ 1,7  \}$ & $\{  3,4\} $ & " & 4\\
\hline
$\{ 1,7 \}$ & $\{ 2,6\} $ & " & 4\\
\hline
$\{ 5,6 \}$ & $\{ 1,2 \} $ & " & 2\\
\hline
$\{ 5,6 \}$ & $\{ 3,4\} $ & " & 2\\
\hline
\end{tabular}
\end{center}
\caption{$T^1$ is 67 dimensional for the Stanley-Reisner scheme of $P^7_4$}
\label{table:boehmTabell}
\end{table}

Consider the three parameter family where $s_4$ is the invariant parameter corresponding to the orbit represented by $a = \{ 5\}$ and $b = \{ 3,4,6 \}$ ($b$ is the triangle and the link of ${\bf a}$ is the suspension of this triangle), $s_7$ is the invariant parameter corresponding to $a = \{ 6 \}$ and $b = \{ 5,7 \}$. This orbit consists of this single element, the link of ${\bf a}$ is the octahedron (suspension of a quadrangle), and $b$ consists of two adjacent points of the quadrangle, and $s_8$ is the parameter corresponding to the orbit represented by $a = \{ 1,2 \}$ and $b = \{ 3,4,7 \}$. Here the link of ${\bf a}$ is the triangle $b$. From the expressions on page~\pageref{ex4refUttrykk} in the Appendix, we have $s_4 := t_5 = t_{10}$, $s_7 := t_{17}$ and $s_8 := t_{18} = t_{19} = t_{24} = t_{25}$. We set the other $t_i$ to zero. Now the general fiber in this three parameter family is defined by the $4\times 4$ pfaffians of the matrix

\begin{equation}
\begin{bmatrix}
0                & s_4 x_7^2        & x_1 x_2          & -x_3x_4        & -s_4 x_5^2\\
-s_4x_7^2  & 0                   & s_8 (x_3 + x_4) & x_5              & -x_6\\
-x_1x_2      & -s_8(x_3 + x_4) & 0                  & s_7x_6          & x_7 \\
x_3x_4       & -x_5              & -s_7x_6                & 0             & s_8(x_1 + x_2)\\
s_4x_5^2   & x_6           & -x_7                    & -s_8(x_1 + x_2) & 0\\
\end{bmatrix}\label{boehmMatriseMin}\,\,.
\end{equation}
If we construct a one-parameter family with parameter $s:= s_4 = s_7 = s_8$, the matrix is

\begin{equation}
\begin{bmatrix}
0        & sx_7^2        & x_1 x_2      & -x_3x_4       & -s x_5^2\\
-sx_7^2  & 0             & s(x_3 + x_4) & x_5           & -x_6\\
-x_1x_2  & -s(x_3 + x_4) & 0            & sx_6          & x_7 \\
x_3x_4   & -x_5          & -sx_6        & 0             & s(x_1 + x_2)\\
sx_5^2   & x_6           & -x_7         & -s(x_1 + x_2) & 0\\
\end{bmatrix}\label{boehmMatrise}\,\, .
\end{equation}
Let $X_s$ be the variety generated by the principal pfaffians of this matrix. It is defined by the ideal generated by the polynomials

\begin{equation*}p_1 = x_5 x_7 + sx_6^2 - s^2(x_1 + x_2)(x_3 + x_4)\end{equation*}
\begin{equation*}p_2 = x_3 x_4 x_7 + s(x_1 + x_2)x_1x_2 - s^2x_5^2x_6\end{equation*}
\begin{equation*}p_3 = x_3 x_4 x_6 + sx_5^3 - s^2(x_1 + x_2)x_7^2\end{equation*}
\begin{equation*}p_4 = x_1 x_2 x_6 + sx_7^3 - s^2(x_3 + x_4)x_5^2\end{equation*}
\begin{equation*}p_5 = x_1 x_2 x_5 + sx_3x_4(x_3 + x_4) - s^2x_6x_7^2\end{equation*}
\noindent
By a Macaulay 2~\cite{M2} computation, the singular locus of this variety is 0-dimensional, and the degree of the singular locus is 48. This fits nicely with the computation we will perform in chapter~\ref{chapter:eulerCharBoehm}, where we find that there are 4 isolated singularities of type $Q_{12}$.

Other choices of 3 parameters give different results. In most cases, the general fiber has singular locus of dimension greater than zero, but there are several ways to construct families where the general fiber has 0-dimensional singular locus. One is obtained if the nonzero parameters (which we equate) are $s_1$, $s_4$ and $s_{8}$ or $s_1$, $s_4$ and $s_{10}$, where $s_4$ and $s_8$ are as above, and $s_1$ is the invariant parameter corresponding to the link being the cyclic polytope and $s_{10}$ is corresponding to the link being a triangle and $a = \{ 1,2 \}$ and $b = \{3,7\}$. In this case the degree of the singular locus is 79 dimensional. It is expected that a similar computation as that in Chapter~\ref{chapter:eulerCharBoehm} would give the same result also in these cases. In this case, the general fiber in the three parameter family is defined by the $4\times 4$ pfaffians of the matrix

\begin{equation*}
\begin{bmatrix}
0         & s_4x_7^2 & f_{12} & -f_{34}  & -s_4 x_5^2\\
-s_4x_7^2 & 0        & l_2    & x_5      & -x_6\\
-f_{12}   & -l_2     & 0      & 0        & x_7 \\
f_{34}    & -x_5     & 0      & 0        & l_1 \\
s_4x_5^2  & x_6      & -x_7   & -l_1     & 0\\
\end{bmatrix}\label{boehmMatriseMin79}\,\, ,
\end{equation*}
\noindent
where $f_{12} = x_1 x_2 +s_1(x_3^2 + x_4^2)$, $f_{34} = x_3 x_4 +s_1(x_1^2 + x_2^2)$, and $l_1 = s_{8} (x_1 + x_2) + s_{10} (x_3 + x_4)$ and $l_2 = s_8 (x_3 + x_4) + s_{10} (x_1 + x_2)$.

If we include all these four parameters, $s_1$, $s_4$, $s_8$ and $s_{10}$, and equate the first three, say $s := s_1 = s_4 = s_8$ and set $t := s_{10}$, we still get dimension 0 and degree 79. If we equate all these four parameters, we no longer have isolated singularities, since $l_1 = l_2$ in the matrix above in this case.

As in the previous section, there is also a subgroup $H\subset T_7$ of the quotient torus acting on $X_s$. Consider the diagonal scalar matrix

$$\lambda = \begin{bmatrix}
\lambda_1 & 0 & 0 & 0 & 0 & 0 & 0\\
0 & \lambda_2 & 0 & 0 & 0 & 0 & 0\\
0 & 0 & \lambda_3 & 0 & 0 & 0 & 0\\
0 & 0 & 0 & \lambda_4 & 0 & 0 & 0\\
0 & 0 & 0 & 0 & \lambda_5 & 0 & 0\\
0 & 0 & 0 & 0 & 0 & \lambda_6 & 0\\
0 & 0 & 0 & 0 & 0 & 0 & \lambda_7\\
\end{bmatrix}
$$
which acts on $(x_1,\ldots, x_7)$ by

$$\lambda \cdot (x_1,\ldots, x_7) = (\lambda_1 \cdot x_1, \ldots ,\lambda_7\cdot x_7)\,\,\,.$$
\noindent
The subgroup acting on $X_s$ is generated by the $\lambda$ with the property that $\lambda \cdot p_i = c_i p_i$ for $i = 1,\ldots 5$, and $c_i$ a constant. From $\lambda p_1 = c_1\cdot p_1$, we obtain the equations

\begin{equation*}\lambda_5\lambda_7 = \lambda_6^2 = \lambda_1\lambda_3 = \lambda_1\lambda_4 = \lambda_2\lambda_3 = \lambda_2\lambda_4 \,\, .\end{equation*}
Hence $\lambda_1 = \lambda_2$, $\lambda_3 = \lambda_4$. For convenience, we set $\lambda_7 = 1$, and we get the equation

\begin{equation} \lambda_5 = \lambda_6^2 = \lambda_1\lambda_3\label{lambdaBoehm1}\end{equation}
\noindent
From $\lambda p_2 = c_2\cdot p_2$, we obtain the equations

\begin{equation*}\lambda_3\lambda_4\lambda_7 = \lambda_1^2\lambda_2 = \lambda_1\lambda_2^2 = \lambda_5^2\lambda_6\end{equation*}
\noindent
Inserting $x_7 = 1$, $\lambda_2 = \lambda_1$ and $\lambda_4 = \lambda_3$, we get

\begin{equation}\lambda_3^2 = \lambda_1^3 = \lambda_5^2\lambda_6\label{lambdaBoehm2}\end{equation}
\noindent
Combining equation~\ref{lambdaBoehm1} and \ref{lambdaBoehm2} we get

\begin{equation}\lambda_3^2 = \lambda_1^3= \lambda_6^5\label{lambdaBoehm3}\end{equation}
\noindent
From $\lambda p_3 = c_3\cdot p_3$, we obtain the equations

\begin{equation*}\lambda_3\lambda_4\lambda_6 = \lambda_5^3 = \lambda_1\lambda_7^2 = \lambda_2\lambda_7^2\end{equation*}
\noindent
Inserting $\lambda_7 = 1$, $\lambda_2 = \lambda_1$, $\lambda_4 = \lambda_3$ and $\lambda_5 = \lambda_6^2$ we get

\begin{equation}\lambda_3^2 \lambda_6= \lambda_6^6 = \lambda_1 \label{lambdaBoehm4}\end{equation}
\noindent
Combining equation \ref{lambdaBoehm3} and \ref{lambdaBoehm4} we get

\begin{equation*}\lambda_6^{13} = 1\end{equation*}
and
\begin{equation*}\lambda_3^4 = \lambda_6^{10}\,\,\, .\label{lambdaBoehm5}\end{equation*}
\noindent
We conclude that the subgroup $H$ is isomorphic to $\mathbb{Z}/13\mathbb{Z}$ with generator

$$(x_1:x_2: x_3: x_4: x_5: x_6: x_7) \mapsto (\xi^3x_1,\xi^3x_2, \xi^{11}x_3, \xi^{11}x_4, \xi x_5, \xi^7x_6 ,x_7)$$
where $\xi$ is a primitive 13th root of 1.

The mirror is constructed in B\"ohm's thesis~\cite{boehm}, using tropical geometry. It can also be constructed by orbifolding, by a crepant resolution of $X_s/H$. In the next chapter we will verify that the euler characteristic of this mirror candidate is actually $120$, as it should be.

\chapter{The Euler Characteristic of the B\"ohm Mirror}\chaptermark{The Euler Char. of the B\"ohm Mirror}\label{chapter:eulerCharBoehm}

Let $X$ be the smooth general fiber of the versal family we studied in \ref{ex4section}. Proposition~\ref{proposition:ex4cohom} states that the Hodge numbers are $h^{1,1}(X) = 1$ and $h^{1,2}(X) = 61$, hence we have

$$\chi(X) = 2 \cdot(h^{1,1}(X) - h^{1,2}(X)) = 2\cdot(1-61) = -120\,\, .$$
\noindent
In this chapter we verify that the Euler Characteristic of the B\"ohm mirror candidate actually is $120$ as it should be.

Recall from Section \ref{boehmmirrorsection} that $X_s$ is the (singular) general fiber of the given one parameter subfamily of the full versal family, and that $H$ is the group $\mathbb{Z}/13\mathbb{Z}$ which acts on $X_s$. Let $Y_s$ be the quotient space $Y_s := X_s/H$.

In this chapter, we construct a crepant resolution $f: M_s \rightarrow Y_s$ and prove the following result, using toric geometry.

\newtheorem{eulerBoehm}{Theorem}[section]
\begin{eulerBoehm}The Euler characteristic of $M_s$ is 120. \label{theorem:eBoehm} \end{eulerBoehm}

The variety $X_s$ has four isolated singular points at $(1:0:0:0:0:0:0)$, $(0:1:0:0:0:0:0)$, $(0:0:1:0:0:0:0)$ and $(0:0:0:1:0:0:0)$. The group $H$ acts freely on $X_s$ away from 6 fixed points: The four singular points and the two smooth points $(1:-1:0:0:0:0:0)$ and $(0:0:1:-1:0:0:0)$. Locally at the latter points, the quotient space $Y_s$ is the germ $(\mathbb{C}^3/H, 0)$ where the action is generated by the diagonal matrix with entries $(\xi, \xi, \xi^{-2})$. To see this, notice that if we set $y_i := x_i/x_1$, the entry $x_1x_2 = y_2$ in matrix~\eqref{boehmMatrise} is a unit locally around the point $(-1:0:0:0:0:0)$. Since the matrix~\eqref{boehmMatrise} is also the syzygy matrix of the ideal generating $X_s$, the five pfaffians generating this ideal reduce to three:

 $$y_2y_5 + sy_3y_4(y_3 + y_4) - s^2y_6y_7^2$$
 $$y_2y_6 + sy_7^3 - s^2(y_3 + y_4)y_5^2$$
 $$y_3y_4y_7 + sy_2(1+y_2) - s^2y_5^2y_6\,\,.$$
Set $v = y_2$. Then the second equation gives $y_6 = - \frac{s}{v}y_7^3 + \frac{s^2}{v} \left( y_3 + y_4 \right )y_5^2$. Inserting this in the first equation gives
$$ f := w y_5 + sv y_3y_4(y_3 + y_4) + s^3y_7^5$$
where $w$ is the unit $v^2 - s^4(y_3 + y_4)y_5y_7^2$, so locally at the fixed point\linebreak $(1:-1:0:0:0:0:0)$, the quotient $X_s/H$ is $$\text{Spec} \left( \mathbb{C}[y_3, y_4, y_5, y_7]/(f)^H\right) \cong \text{Spec} \left( \mathbb{C}[y_3, y_4, y_7]^H \right) \cong \mathbb{C}^3/H.$$ The group $H$ acts by $$(y_3,y_4,y_7) \mapsto (\xi y_3, \xi y_4, \xi^{-2}y_7)\,\,\, .$$ The situation is similar in the other fixed point $(0:0:1:-1:0:0:0)$.

Now consider the four singular points. One sees that $D_4$ gives isomorphisms of the germs at the singular points. Let $P$ be one of these singular points, by symmetry we can choose $P = (1:0:0:0:0:0:0)$. To see what $(X_s,P)$ look like locally, we consider an affine neighborhood of $P$, so we can assume $x_1 = 1$ with $P$ the origin in this affine neighborhood. Set $y_i = \frac{x_i}{x_1}$. Now $s(x_1 + x_2) = s(1 + y_2)$ is a unit around the origin, and the five pfaffians again reduce to three:

$$y_5y_7 + sy_6² - s²(1 + y_2)(y_3 + y_4)$$
$$y_3y_4y_7 + sy_2(1 + y_2) - s²y_5²y_6$$
$$y_3y_4y_6 + sy_5³ - s²(1 + y_2)y_7²$$
From the second equation we get $y_2 = u(s²y_5²y_6 - y_3y_4y_7)$ where $u$ is a unit locally around the origin. The first and third equations are now

$$y_5y_7 + sy_6² - v(y_3 + y_4)$$
$$y_3y_4y_6 + sy_5³ - v y_7²$$
where $v$ is the unit $s²(1 + y_2)$. Set $z_1 = y_3 + y_4$, $z_2 = y_3 - y_4$, $z_3 = y_5$, $z_4 = y_6$ and $z_5 = y_7$. Then we have

$$z_3z_5 + sz_4^2 - vz_1$$
$$(z_1² - z_2²)z_4 + 4sz_3³ - 4vz_5²$$
Inserting $z_1 = \frac{1}{v}(z_3z_5 + sz_4^2 )$ in the second equation gives

$$z_3²z_4z_5² + 2sz_3z_4³z_5 + s²z_4^5 - v²z_2²z_4 + 4sv²z_3³ - 4v^3z_5²\,\,\, .$$

\noindent After a coordinate change, this polynomial is

$$g = z_5^2 + z_3^3 + z_2^2z_4 + z_4^5 +  w_1 z_3z_4^3z_5 + w_2 z_3^2z_4z_5^2\,\,\, .$$

\noindent where $w_1$ and $w_2$ are $H$ invariant units (since $y_2 = x_2/x_1$ maps to $y_2$ under the action of $H$). The polynomial $g$ has Milnor number 12, and the corank of the Hessian matrix of $g$ is 3. By Arnold's classification of singularities~\cite{arnold} the type of the singularity is $Q_{12}$. The normal form of this singularity is

$$f = z_5^2 + z_3^3 + z_2^2z_4 + z_4^5 \,\, .$$

\noindent In order to show that $f$ and $g$ represent the same germ, we give an $H$ invariant coordinate change taking $g$ to $f$ locally around the origin.

\noindent We first perform the coordinate change

$$z_5 \mapsto z_5 - \frac{1}{2} w_1z_3z_4^3 \,\, ,$$
which maps $g$ to

$$z_5^2 + z_3^3 + z_2^2z_4 + z_4^5 - \frac{1}{4} w_1^2z_3^2z_4^6 + w_2z_3^2z_4z_5^2 - w_1w_2z_3^3z_4^4z_5 + \frac{1}{4}w_1^2w_2z_3^4z_4^7\,\,\, .$$

\noindent This expression may be written

$$u_1 z_5^2 + u_2z_3^3 + z_2^2z_4 + u_3z_4^5 \,\,\, .$$

\noindent where $u_1$, $u_2$ and $u_3$ are $H$ invariant units locally around the origin. After a coordinate change, we obtain the standard form $f$.

Since $H$ now acts as $(z_2, z_3, z_4, z_5)\mapsto(\xi^8z_2, \xi^{-2}z_3, \xi^{4}z_4, \xi^{-3}z_5)$, the polynomial $f$ is {\it semi-invariant} in the sense that $f(\xi^8z_2, \xi^{-2}z_3, \xi^{4}z_4, \xi^{-3}z_5)= \xi^7 f(z_2, z_3, z_4, z_5)$. In fact, it is also a {\it quasi-homogeneous} function of degree 1 and weight $(\alpha_2, \alpha_3, \alpha_4,\alpha_5) = (\frac{2}{5},\frac{1}{3},\frac{1}{5},\frac{1}{2})$, i.e.

$$f(\lambda^{\alpha_2}z_2, \ldots ,\lambda^{\alpha_5}z_5) = \lambda f(z_2,\ldots,z_5)$$ for any $\lambda \geq 0$. We now use also Arnolds notation and set $f = w^2 + x^3 + y^5 + yz^2$. The singularity of $Y_s$ at $0$ is a so called {\it hyper quotient singularity} (hypersurface singularity divided by a group action). The group $H$ acts on $Z := Z(f) := \displaystyle\{ p \mid f(p) = 0 \displaystyle\} \subset \mathbb{C}^4$, and we have $Z/H \cong  \text{Proj} (\mathcal{O}_{\mathbb{C}^4}/(f))^H$.

The quotient $(Z/H,0)$ is Gorenstein. This follows from the following general observation. Let $H$ be a finite subgroup of $GL_n(\mathbb{C})$ and $(Z,0) \subset (\mathbb{C}^n, 0)$ a codimension $r$ Gorenstein singularity with an induced $H$ action. Let $\mathcal{F}$ be a free $\mathcal{O}_{\mathbb{C}^n}$ resolution of $\mathcal{O}_Z$ which is also an $H$ module; i.e. $F_i \cong \mathcal{O}_{\mathbb{C}^n} \otimes V_i$ as $H$ modules with $V_i$ a representation of $H$. Let $V_{\text{det}}$ be the representation $g\mapsto \text{det}(g)$. If $V_k^*\cong V_{r-k}\otimes V_{\text{det}}$ as representations, then $(Z/H, 0)$ is Gorenstein.

 Let $V$ be the singularity $(Z/H,0)\subset (\mathbb{C}^4/H,0)$. It is defined by the ideal $(f)^H$ in $\mathcal{O}^H_{\mathbb{C}^4}$. We wish to construct a crepant resolution $\widetilde{V}\rightarrow V$ of this singularity.

From now on we use freely the notation and results from the book by Fulton \cite{fulton}. We may find a cone $\sigma^{\vee}$ and a lattice $M$ such that

$$\mathbb{C}[w,x,y,z]^H = \mathbb{C}[\sigma^{\vee}\cap M]$$
A monomial $w^\alpha x^\beta y^\gamma z^\delta$ maps to $\xi^{-3\alpha -2\beta + 4\gamma + 8\delta} w^\alpha x^\beta y^\gamma z^\delta$. This monomial is invariant under the action of $H$ if $-3\alpha - 2\beta + 4\gamma + 8\delta = 0 \,(\text{mod}\,13)$. This equation can be written $\alpha +5 \beta + 3\gamma + 6\delta = 0 \,(\text{mod}\,13)$. Let $M$ be the lattice

$$\{ (\alpha, \beta, \gamma, \delta ) | \alpha +5\beta + 3\gamma + 6\delta = 0 \,(\text{mod}\,13)\}$$
and let $\sigma^{\vee}$ be the first octant in $M_{\mathbb{R}}$. Let $N := \text{Hom}(M,\mathbb{Z})$ be the dual lattice, i.e.

$$N = \mathbb{Z}^4 + \frac{1}{13}(1,5,3,6)\mathbb{Z}\,\,\,\, .$$
The dual cone $\sigma$ is the first octant in $N_{\mathbb{R}}$. Let $v_1,\ldots ,v_4$ be the vectors $v_1 = \frac{1}{13}(1,5,3,6)$, $v_2 = (0,1,0,0)$, $v_3 = (0,0,1,0)$ and $v_4 = (0,0,0,1)$. The isomorphism $\oplus\mathbb{Z}v_i \rightarrow N$ given by multiplication by the matrix

$$A:= \begin{bmatrix}
1/13 & 0 & 0 & 0\\
5/13 & 1 & 0 & 0\\
3/13 & 0 & 1 & 0\\
6/13 & 0 & 0 & 1\\
\end{bmatrix}$$
takes the cone generated by $(13,-5,-3,-6)$, $(0,1,0,0)$, $(0,0,1,0)$ and $(0,0,0,1)$ to the cone $\sigma$. The dual isomorphism  $ M\rightarrow   \bigoplus \mathbb{Z}w_i $  is given by multiplication by the transpose $A^T$.

We find a toric resolution $X_{\Sigma} \rightarrow \mathbb{C}^4/H$ with

$$ \XY
  \xymatrix@1{\widetilde{V}\,\ar@{^{(}->}[r] \ar[d] & X_{\Sigma}\ar[d] \\
   V\, \ar@{^{(}->}[r] & \mathbb{C}^4/H }$$
where $\widetilde{V}$ is the strict transform of $V$. A toric resolution of $\mathbb{C}^4/H$ corresponds to a regular subdivision of $\sigma$. This may be computed using the Maple package convex~\cite{convex}. The command {\it regularsubdiv} in convex does not give a resolution with a smooth strict transform, so an additional manual subdivision is made. Table~\ref{table:rays} lists the rays in such a regular subdivision, in the basis $v_1,\ldots,v_4$. On page~\pageref{table:cones} in the Appendix, all the maximal cones of this subdivision $\Sigma$ of $\sigma$ are listed, and they are labeled $\tau_{1},\ldots,\tau_{53}$. Each cone is represented by the four rays spanning it.

\begin{table}
\begin{center}
\begin{tabular}{l}
$[0, 0, 0, 1]$, $[0, 0, 1, 0]$, $[0, 1, 0, 0]$, $[1, 0, 0, 0]$, $[3, -1, 0, -1]$, $[3, 0, 0, -1]$,\\
$[5, -1, -1, -2]$, $[5, -1, 0, -2]$, $[6, -2, -1, -2]$, $[7, -2, -1, -3]$, $[8, -3, -1, -3]$,\\
$[9, -3, -2, -4]$, $[11, -4, -2, -5]$, $[11, -4, -2, -4]$, $[12, -4, -2, -5]$,\\
$[13, -5, -3, -6]$, $[14, -5, -3, -6]$, $[15, -5, -3, -6]$\label{table:rays}\\
\end{tabular}\caption{The rays $\rho$ of a regular subdivision $\Sigma$ of the cone $\sigma$.}
\end{center}
\end{table}

The polynomial $f$ is only semi-invariant, and the ideal $(f)^H$ has many generators in $\mathbb{C}[\sigma^{\vee} \cap M]$. Still, $\widetilde{V}$ is irreducible and codimension 1 in $X_{\Sigma}$ and therefore defined by an irreducible polynomial $\widetilde{f}_{\tau}$ in each $\mathbb{C}[\tau^{\vee}\cap M] = \mathbb{C}[y_1, y_2, y_3, y_4]$ when $\tau \in\Sigma$. The $y_i$ correspond to the four rays of $\tau$, in the order in which they are listed on page \pageref{table:cones}. To compute $\widetilde{f}_{\tau}$, take the image of any generator of $(f)^H$ by the inclusion $\mathbb{C}[\sigma^{\vee}\cap  M] \subset \mathbb{C}[\tau^{\vee}\cap M]$ and remove all factors which are powers of some $y_i$. We can choose the generator $y^8 f \in (f)^H$. The weights of the monomials of $y^8 f $ are $[2,0,8,0]$,$[0,3,8,0]$,$[0,0,13,0]$ and $[0,0,9,2]$.

We will compute $\widetilde{f}_{\tau}$ for a specific $\tau$ to illustrate the idea. Let $\tau := \tau_1$ be the cone in $\Sigma$ generated by the vectors $[13,-5,-3,-6]$, $[0,0,1,0]$, $[3,-1,0,-1]$ and $[8,-3,-1,-3]$ in $\oplus\mathbb{Z}v_i$. Let $B$ be the matrix

$$B:=
\begin{bmatrix}
13 & 0 & 3 & 8\\
-5 & 0 & -1 & -3\\
-3 & 1 & 0 & -1\\
-6 & 0 & -1 & -3\\
\end{bmatrix}\,\,\, .$$
The rays of $\tau^{\vee}$ are generated by the columns of the matrix $(B^{-1})^T$. Thus in $M$, the rays of $\tau^{\vee}$ are generated by the columns of $(A^T)^{-1}\cdot (B^{-1})^T = (B^TA^T)^{-1}$. The image of $y^8f$ by the inclusion is a factor $\widetilde{f}_{\tau}$ of the polynomial ${\bf y}^{B^{T}A^T\cdot [2,0,8,0]} + {\bf y}^{B^{T}A^T\cdot [0,3,8,0]} + {\bf y}^{B^{T}A^T\cdot [0,0,13,0]} + {\bf y}^{B^{T}A^T\cdot [0,0,9,2]}$, where the multi index notation ${\bf y}^{[i_1,\ldots,i_4]}$ means $y_1^{i_1}\cdots y_4^{i_4}$. In this case $\widetilde{f}_{\tau} = y_4y_1^2 + 1 + y_4^4y_3^3y_2^5 + y_4^2y_3y_2$. In this way one checks that all $\tilde{f}$ in fact define smooth hypersurfaces in each chart, i.e. that $\tilde{V}$ is smooth.

Each ray $\rho$ in $\Sigma$, aside from the 4 generating the cone $\sigma$, determines an exceptional divisor $D_{\rho}$ in $X_{\Sigma}$. Hence there are 14 exceptional divisors in $X_{\Sigma}$. For every ray $\rho$, the exceptional divisor $D_{\rho}$ is a smooth, complete toric 3-fold and comes with a fan $\text{Star}(\rho)$ in a lattice $N(\rho)$; we define $N_{\rho}$ to be the sublattice of $N$ generated (as a group) by $\rho \cap N$ and

$$N(\rho) = N/N_{\rho},\,\,\,M(\rho) = M\cap \rho^{\perp}$$
The torus $T_{\rho}\subset D_{\rho}$ corresponding to these lattices is

$$T_{\rho} = \text{Hom}(M(\rho),\mathbb{C}^*) = \text{Spec}(\mathbb{C}[M(\rho)]) = N(\rho)\otimes_{\mathbb{Z}}\mathbb{C}^*\,\, .$$
 The subvariety $\widetilde{V}$ will only intersect 10 of these exceptional divisors $D_{\rho}$. To check this, we compute the fan consisting of all the cones of $\Sigma$ containing the ray $\rho$, realized as a fan in the quotient lattice $N(\rho)$. The quotient map $\mathbb{C}[\tau^{\vee}\cap M] \rightarrow \mathbb{C}[\tau^{\vee}\cap M(\rho)]$ sends $y_i$ to 0 if $y_i$ is the coordinate corresponding to the ray $\rho$. The other three coordinates are unchanged. Let $\overline{f}_{\tau}$ be the image of $\tilde{f}_{\tau}$ under this projection map, i.e. $\overline{f}_{\tau}:= \tilde{f}_{\tau} | (y_i = 0)$.

We consider the cone $\tau = \tau_1$ studied above, and the ray $\rho$ generated by $(3,-1,0,-1)$. In this case, the coordinate $y_3$ is zero, and the polynomial $\overline{f}_{\tau}$ is $y_4y_1^2 + 1$. Hence the ray $\rho$ intersects $\widetilde{V}$ in this chart. This computation can be performed for all the 14 rays. If the strict transform is 1 on all charts containing $D_i$, then there is no intersection. The rays generated by $[3,0,0,-1]$,$[5,-1,0,-2]$,$[8,-3,-1,-3]$ and $[11,-4,-2,-4]$ do not intersect $\tilde{V}$, hence the subvariety $\tilde{V}$ will intersect 10 of the exceptional divisors.

In 9 of these 10 cases the intersection is irreducible and in one case the intersection has 4 components, but one of these is the intersection with another exceptional divisor. All in all the exceptional divisor $E$ in $\widetilde{V}$ has 12 components. We list the 12 components of $E$ in Table~\ref{table:exceptional}.

\begin{table}
\begin{center}
\begin{tabular}{cccc}
Label & $\alpha$ & Type & $\chi$\\
\hline
$E_1$ & $(6,-2,-1,-2)    $  & $\mathbb{P}^2$                  & 3\\
$E_2$ & $(3,-1,0,-1)     $  & $\text{Bl}_1\mathbb{F}_2  $ & 5\\
$E_3$ & $(11,-4,-2,-5)  $  & $\mathbb{F}_5$                   &4\\
$E_4$ & $(7,-2,-1,-3)    $  & $\mathbb{F}_2$                   & 4\\
$E_5$ & $(9,-3,-2,-4)    $  & $\text{Bl}_2\mathbb{F}_2$   & 6\\
$E_6$ & $(9,-3,-2,-4)    $  & $\text{Bl}_2\mathbb{F}_2$   & 6\\
$E_7$ & $(15,-5,-3,-6)  $  & $\text{Bl}_3\mathbb{F}_2$   & 7\\
$E_8$ & $(12,-4,-2,-5)  $  & $\text{Bl}_3\mathbb{F}_2$   & 7\\
$E_9$ & $(14,-5,-3,-6)  $  & $\mathbb{F}_2$                   & 4\\
$E_{10}$ & $(1\, ,0\, ,0\, ,0\,)   $   & $\text{Bl}_3\mathbb{P}^2$ & 6\\
$E_{11}$ & $(9,-3,-2,-4)$  & $\text{Bl}_1\mathbb{F}_4 $  & 5\\
$E_{12}$ & $(5,-1,-1,-2)$  & $\mathbb{F}_3$   & 4 \\
\end{tabular}\caption{Components of $\widetilde{V}\cap E$.}\label{table:exceptional}
\end{center}
\end{table}

We may check that the resolution is crepant using the following formula for the discrepancies of hyperquotient singularities, see the article by Reid~\cite{rei}. Let $\alpha \in N$ be the primitive vector generating $\rho$. Any $m \in M$ determines a rational monomial in the variables $w,x,y,z$ and we write $m \in f$ if the monomial is in $\{ w^2, x^3, yz^2, y^5 \}$. Define $\alpha(f) = \text{min} \{ \alpha(m) \mid m\in f \}$. The result is that components of $\widetilde{V} \cap D_{\rho}$ are crepant if and only if

$$\alpha(1,1,1,1) = \alpha(f) +1\,\,\, .$$
This may easily be checked to be true for all $\rho\in \Sigma$ with $\widetilde{V}\cap D_{\rho} \neq \emptyset$.

To compute the type of $E_i$, several different techniques were needed depending upon the complexity of $D_{\rho}$ and/or $\tilde{f}_{\tau}$, $\rho \subset \tau$. For each $\rho$ we compute the polynomials $\overline{f}_{\tau}$. If for some $\tau$, $\overline{f}_{\tau}$ is on the form $\overline{f}_{\tau} = y_j^{n_j}y_k^{n_k} + 1$ with $(n_j, n_k) \neq (0,0)$, we use Method 1 described below.

\underline{Method 1.} In some cases the intersection $\widetilde{V}\cap T_{\rho}$ is a torus. This torus may be described as $\widetilde{N(\rho)} \otimes \mathbb{C}^{*}$, where $\widetilde{N(\rho)}$ is a rank 2 lattice. The inclusion $\widetilde{V} \cap T_{\rho} \rightarrow T_{\rho}$ may be computed to be induced by a linear map $\phi : \widetilde{N(\rho)} \rightarrow N(\rho)$. Now $\widetilde{V} \cap D_{\rho}$ is the closure of $\widetilde{V}\cap T_{\rho}$ in $D_{\rho}$, so it is the toric variety with fan $\phi^{-1}(\text{Star}(\rho))$.\vskip 4 pt

\noindent\underline{$E_1$.} Consider the case where a primitive vector generating $\rho$ is $(6,-2,-1,-2)$. Let $\tau = \tau_{29}$ be the cone generated by $(13, -5, -3, -6)$, $(0, 0, 0, 1)$, $(6, -2, -1, -2)$ and $(14, -5, -3, -6)$. In this chart, $\widetilde{V}$ is generated by $\tilde{f}_{\tau} = y_1^2y_4 + 1 + y_2y_3^2+ y_3$. Restricted to $y_3 = 0$ (corresponding to the ray $(6,-2,-1,-2)$), this gives $\overline{f}_{\tau} = y_1^2y_4 + 1$. Hence the inclusion $\widetilde{V} \cap T_{N(\rho)} \rightarrow T_{N(\rho)}$ is induced by the inclusion of the sublattice $\widetilde{N(\rho)}\cong \mathbb{Z}^2$ of $N(\rho) \cong \mathbb{Z}^3$ generated by $\pm(1,0,-2)$ and $\pm(0,1,0)$. Hence the map $\phi : \widetilde{N(\rho)}\rightarrow N(\rho)$ is

$$\begin{bmatrix}1 & 0\\
0 & 1\\
-2 & 0\\
\end{bmatrix}$$
and $\widetilde{V}\cap D_{\rho}$ is $\phi^{-1}(\text{Star}(\rho))$. The fan $\text{Star}(\rho)$ consists of 10 maximal cones, and $\phi^{-1}(\text{Star}(\rho))$ is generated by the rays through $(-1,0)$, $(0,1)$ and $(1,-1)$. This fan is drawn in figure~\ref{figure:p2}, and it represents $\mathbb{P}^2$.  In fact, $\phi^{-1}(\text{Star}(\rho))$ can be checked to generate $\mathbb{P}^2$ for all the 10 maximal cones $\tau$ with $\rho$ a ray in $\tau$. In Table~\ref{table:exceptional}, this component of the exceptional divisor $E$ is labeled $E_1$.\vskip 4 pt

\begin{figure}
  \begin{center}
  \includegraphics[width=4cm]{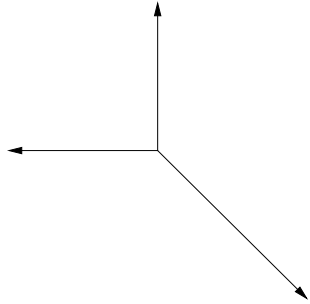}\end{center}
  \caption{A fan representing $\mathbb{P}^2$}\label{figure:p2}
\end{figure}

\noindent\underline{$E_2$.} Now let $\rho$ be generated by the primitive vector $(3, -1, 0, -1)$, and let $\tau = \tau_{48}$. In this chart, $\tilde{V}$ is generated by $\tilde{f}_{\tau} = y_2y_3^2 + 1 + y_1 + y_4^2y_2^2y_1^3$. Restricted to $y_1 = 0$ (corresponding to the ray $(3, -1, 0, -1)$), this gives $\overline{f}_{\tau} = y_2y_3^2 + 1$. By a similar computation as the one above, the fan $\phi^{-1}(\text{Star}(\rho))$ is generated by the rays through the points $(-1,-1)$, $(0,1)$, $(1,1)$, $(1,2)$ and $(2,1)$. This fan is drawn in figure~\ref{figure:bl1f2}. Since $(1,2)$ is the sum of $(0,1)$ and $(1,1)$, the fan represents the blow up of $\mathbb{F}_2$ in a point. On the other hand, $(0,1)$ is the sum of $(-1,-1)$ and $(1,2)$, and the fan represents the blow up of $\mathbb{F}_3$ in a point. These two surfaces are isomorphic.\vskip 4 pt

\begin{figure}
  \begin{center}
  \includegraphics[width=5cm]{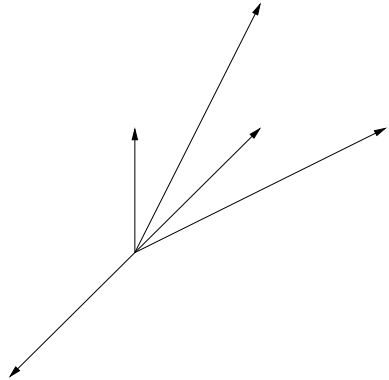}\end{center}
  \caption{A fan representing $\text{Bl}_1\mathbb{F}_2$}\label{figure:bl1f2}
\end{figure}

\noindent\underline{$E_3$.} Let $\tau = \tau_{48}$. In this case $\tilde{f}_{\tau} = y_2y_3^2 + 1 + y_2^2y_4^2y_3^3+ y_1$. Restricted to $y_2 = 0$ (corresponding to the ray $\rho$, this gives $\overline{f}_{\tau} = 1 + y_1$. The fan $\phi^{-1}(\text{Star}(\rho))$ is generated by the rays through the points $(-5, -1)$, $(-1, 0)$, $(0, 1)$ and $(1, 0)$, and is drawn in figure~\ref{figure:f5}. This fan represents $\mathbb{F}_5$.\vskip 4 pt

\begin{figure}
  \begin{center}
  \includegraphics[width=9cm]{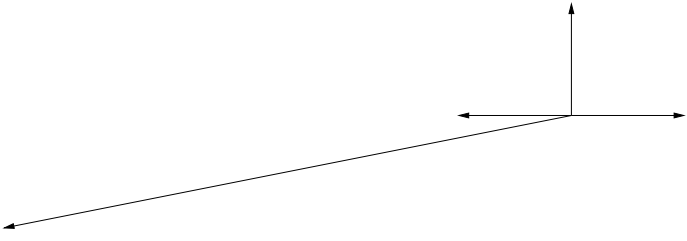}\,\,\,\,\,\,\,\,\,\,\,\,\,\,\,\,\,\,\,\,
  \,\,\,\,\,\,\,\,\,\,\,\,\,\,\,\,\,\,\,\,\,\end{center}
  \caption{A fan representing $\mathbb{F}_5$}\label{figure:f5}
\end{figure}

\noindent\underline{$E_4$.} Let $\tau = \tau_{36}$. In this case $\tilde{f}_{\tau} = y_3^2 + y_2y_1^3y_4^2y_3 + y_2 + 1$. Restricted to $y_3 = 0$ (corresponding to the ray $\rho$, this gives $\overline{f}_{\tau} = 1 + y_2$. The fan $\phi^{-1}(\text{Star}(\rho))$ is generated by the rays through the points $(-1, 2)$, $(0, -1)$, $(0,1)$ and $(1, 0)$, which represents $\mathbb{F}_2$.\vskip 4 pt

\noindent\underline{$E_5$ and $E_6$.} In these two cases it is a component of $\tilde{V}\cap D_{\rho}$ that intersects $T_{N(\rho)}$ in a torus. To see this, consider the case with $\rho$ generated by the primitive vector $(9,-3,-2,-4)$. The fan $\text{Star}(\rho)$ consists of 18 maximal cones. Consider the cone $\tau:= \tau_{13}$. In this chart, $\widetilde{V}$ is generated by $\tilde{f}_{\tau} = y_3 + y_2y_3 + y_1y_2 + y_1^5y_2y_4^2$. Restricted to $y_3 = 0$ (corresponding to the ray $(9,-3,-2,-4)$), this gives the polynomial $\overline{f}_{\tau} = y_1y_2(1 + y_1^4y_4^2)$.  The component $y_1 = 0$ corresponds to the ray $(0,0,1,0)$ and the component $y_2 = 0$ corresponds to $(1,0,0,0)$. These two components are labeled $E_{10}$ and $E_{11}$, and we take a closer look at these in Method 3. Both factors $(1 + iy_1^2y_4)$ and $(1 - iy_1^2y_4)$ of the polynomial $\overline{f}_{\tau}$ give rise to a linear map $\phi : \widetilde{N(\rho)} \rightarrow N(\rho)$ represented by the matrix

$$\begin{bmatrix}
1 & 0\\
0 & 1\\
-2 & 0\\
\end{bmatrix}$$
and $\phi^{-1}(\text{Star}(\rho))$ is generated by the rays through $(-1,-2)$, $(-1,-1)$, $(-1,0)$, $(0,1)$, $(1,2)$ and $(2,3)$. Since $(-1,0) = (-1,-1) + (0,1)$ and $(-1,-1) = (-1,-2) + (0,1)$, the fan represents $\text{Bl}_2\mathbb{F}_2$.\vskip 4 pt

\underline{Method 2.} In 3 cases one sees from the fan $\text{Star}(\rho)$ that $D_{\rho}$ is a locally trivial $\mathbb{P}^1$ bundle over a smooth toric surface.\vskip 4 pt

\noindent\underline{$E_7$.} Consider first the case with $\rho = (15, -5, -3, -6)$. Let $M$ be the matrix

$$M = \begin{bmatrix}
0 & 1 & 9 &  15 \\
0 & 0 & -3 & -5 \\
0 & 0 & -2 & -3 \\
1 & 0 & -4 & -6 \\
\end{bmatrix}$$
and let $\tau$ be the cone generated by the columns of $M$, i.e. $\tau = \tau_{42}$. The fan $\text{Star}(\rho)$ is the image of the 10 maximal cones in $\Sigma$ containing the ray $\rho$ under the projection map $\text{Pr}: N \rightarrow N(\rho)$ given by

$$\text{Pr}:= \begin{bmatrix}
1 & 0 & 0 & 0\\
0 & 1 & 0 & 0\\
0 & 0 & 1 & 0\\
\end{bmatrix} \times M^{-1} =
\begin{bmatrix}
0 & 0 & -2 & 1\\
1 & 3 & 0 & 0\\
0 & 3 & -5 & 0\end{bmatrix}$$
\begin{figure}
  \begin{center}
  \includegraphics[width=4cm]{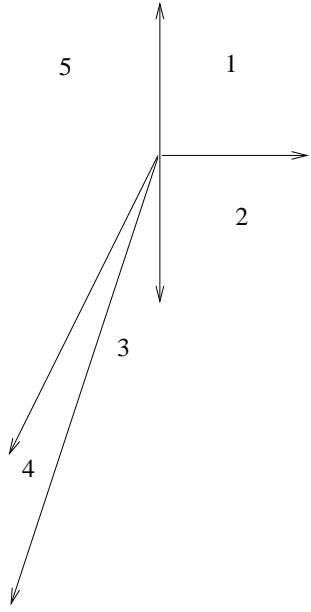}\end{center}
  \caption{A representation of the surface $\text{Bl}_1\mathbb{F}_2$}\label{figure:metode2}
\end{figure}

The fan $\text{Star}(\rho)$ in $N(\rho)$ is generated by the rays $(-1, 0, -3)$, $(-1, 0, -2)$, $(0, -1, 0)$, $(0, 0, -1)$, $(0, 0, 1)$, $(0, 1, 0)$, $(1, 0, 0)$. Let $\Delta^{'}$ be the fan generated by $1$ and $-1$ in the lattice $\mathbb{Z}$, and let  $\Delta^{''}$ be the fan generated by $(-1, -3)$, $(-1, -2)$, $(0, -1)$, $(0, 1)$ and $(1, 0)$ in the lattice $\mathbb{Z}^2$. There is an exact sequence of lattices

$$0 \rightarrow \mathbb{Z} \rightarrow N(\rho) \rightarrow \mathbb{Z}^2 \rightarrow 0$$
where the map $\mathbb{Z} \rightarrow N(\rho)$ is the inclusion $n \mapsto (0,n,0)$ and the map $N(\rho) \rightarrow \mathbb{Z}^2$ is the projection $(x,y,z) \mapsto (x,z)$. This exact sequence gives rise to mappings

$$X(\Delta^{'}) \rightarrow X(\text{Star}(\rho)) \rightarrow X(\Delta^{''})$$
where $X(\Delta^{''})$ is the blow up of $\mathbb{F}_2$ in a point, and $X(\Delta^{'})$ is $\mathbb{P}^1$. Thus we have a trivial $\mathbb{P}^1$ bundle over $\text{Bl}_1\mathbb{F}_2$.

We can find out what $\tilde{V} \cap D_{\rho}$ is by a local look at the fibers of the map $\tilde{V}\cap D_{\rho} \rightarrow \text{Bl}_1\mathbb{F}_2$. Over the charts labeled 2, 3 and 4 in Figure~\ref{figure:metode2}, the map is an isomorphism. Over the intersection of the charts 1 and 5, we have an isomorphism except over two points in $\text{Bl}_1\mathbb{F}_2$, where the inverse image is a line. Hence, $\tilde{V}\cap D_{\rho}$ is isomorphic to $\text{Bl}_3\mathbb{F}_2$.\vskip 4 pt

\noindent\underline{$E_8$.} In this case the fan $\text{Star}(\rho)$ represents a locally trivial $\mathbb{P}^1$ bundle over $\text{Bl}_2\mathbb{F}_2$, and $\tilde{V}\cap D_{\rho}$ is isomorphic to $\text{Bl}_3\mathbb{F}_2$.

To see this, let $\rho = (12, -4, -2, -5)$, and let $M$ be the matrix

$$M = \begin{bmatrix}
9  & 11 & 13 &  12 \\
-3 & -4 & -5 & -4 \\
-2 & -2 & -3 & -2 \\
-4 & -5 & -6 & -5 \\
\end{bmatrix}$$
and let $\tau$ be the cone generated by the columns of $M$, i.e. $\tau = \tau_{47}$. The fan $\text{Star}(\rho)$ is the image of the 12 maximal cones in $\Sigma$ containing the ray $\rho$ under the projection map $\text{Pr}: N \rightarrow N(\rho)$ given by

$$\text{Pr}:= \begin{bmatrix}
1 & 0 & 0 & 0\\
0 & 1 & 0 & 0\\
0 & 0 & 1 & 0\\
\end{bmatrix} \times M^{-1} =
\begin{bmatrix}
0  & 3 & -1 & -2\\
-1 & 1 & 2 & -4\\
0 & -2 & -1 & 2\end{bmatrix}\,\, .$$
The fan $\text{Star}(\rho)$ in $N(\rho)$ is generated by the rays $(-1, 0, 0)$, $(-1, 2, -1)$, $(0, -2, 1)$, $(0, -1, 0)$, $(0, -1, 1)$, $(0, 0, 1)$, $(0, 1, 0)$ and $(1,0,0)$. Let $\Delta^{'}$ be the fan generated by $1$ and $-1$ in the lattice $\mathbb{Z}$, and let  $\Delta^{''}$ be the fan generated by $(2, -1)$, $(-2, 1)$, $(-1, 0)$, $(-1, 1)$, $(0,1)$ and $(1, 0)$ in the lattice $\mathbb{Z}^2$. There is an exact sequence of lattices

$$0 \rightarrow \mathbb{Z} \rightarrow N(\rho) \rightarrow \mathbb{Z}^2 \rightarrow 0$$
where the map $\mathbb{Z} \rightarrow N(\rho)$ is the inclusion $n \mapsto (n,0,0)$ and the map $N(\rho) \rightarrow \mathbb{Z}^2$ is the projection $(x,y,z) \mapsto (y,z)$. This exact sequence gives rise to mappings

$$X(\Delta^{'}) \rightarrow X(\text{Star}(\rho)) \rightarrow X(\Delta^{''})$$
where $X(\Delta^{''})$ is the blow up of $\mathbb{F}_2$ in two points, and $X(\Delta^{'})$ is $\mathbb{P}^1$. Thus we have a locally trivial $\mathbb{P}^1$ bundle over $\text{Bl}_2\mathbb{F}_2$.

We can find out what $\tilde{V} \cap D_{\rho}$ is by a local look at the fibers of the map $\tilde{V}\cap D_{\rho} \rightarrow \text{Bl}_1\mathbb{F}_2$. By a similar computation as in $E_8$, we find that the inverse image is a line in one point, and otherwise an isomorphism. Hence, $\tilde{V}\cap D_{\rho}$ is isomorphic to $\text{Bl}_3\mathbb{F}_2$.\vskip 4 pt

\noindent\underline{$E_9$.} Let $\rho$ be generated by the primitive vector $(14, -5, -3, -6)$, and let $M$ be the matrix

$$M = \begin{bmatrix}
0 &  9 & 15 &  14 \\
0 & -3 & -5 & -5 \\
0 & -2 & -3 & -3 \\
1 & -4 & -6 & -6 \\
\end{bmatrix}$$
and let $\tau$ be the cone generated by the columns of $M$, i.e. $\tau = \tau_{50}$. The fan $\text{Star}(\rho)$ is the image of the 10 maximal cones in $\Sigma$ containing the ray $\rho$ under the projection map $\text{Pr}: N \rightarrow N(\rho)$ given by

$$\text{Pr}:= \begin{bmatrix}
1 & 0 & 0 & 0\\
0 & 1 & 0 & 0\\
0 & 0 & 1 & 0\\
\end{bmatrix} \times M^{-1} =
\begin{bmatrix}
0 & 0 & -2 & 1\\
0 & 3 & -5 & 0\\
1 & 1 &  3 & 0\end{bmatrix}\,\, .$$

The fan $\text{Star}(\rho)$ in $N(\rho)$ is generated by the rays $(-1, -3, 2)$, $(-1, -2, 2)$, $(0, -1, 1)$, $(0, 0, -1)$, $(0, 0, 1)$, $(0, 1, 0)$, and $(1,0,0)$. Let $\Delta^{'}$ be the fan generated by $1$ and $-1$ in the lattice $\mathbb{Z}$, and let  $\Delta^{''}$ be the fan generated by $(-1, -3)$, $(-1, -2)$, $(0, -1)$, $(0,1)$ and $(1, 0)$ in the lattice $\mathbb{Z}^2$. There is an exact sequence of lattices

$$0 \rightarrow \mathbb{Z} \rightarrow N(\rho) \rightarrow \mathbb{Z}^2 \rightarrow 0$$
where the map $\mathbb{Z} \rightarrow N(\rho)$ is the inclusion $n \mapsto (0,0,n)$ and the map $N(\rho) \rightarrow \mathbb{Z}^2$ is the projection $(x,y,z) \mapsto (x,y)$. This exact sequence gives rise to mappings

$$X(\Delta^{'}) \rightarrow X(\text{Star}(\rho)) \rightarrow X(\Delta^{''})$$
The image of $E_9$ under this last projection is a rational curve on a toric surface, and $E_9$ is a $\mathbb{P}^1$ bundle over this curve, but a local computation as in $E_7$ does not tell us what $\tilde{V} \cap D_{\rho}$ looks like.

We look at the 10 charts of $X_{\Sigma}$ containing $D_{\rho}$. Four of these cover $\tilde{V} \cap D_{\rho}$. A covering is given by the cones $\tau_{22}$, $\tau_{24}$, $\tau_{50}$ and $\tau_{51}$. In these four maps, the polynomial $\overline{f}_{\tau}$ is of the form $1 + x + y^2$, hence $\tilde{V} \cap D_{\rho}$ is a union of four copies of $\mathbb{C}^2$. They glue together to form $\mathbb{F}_2$.

\underline{Method 3.} In two cases $E_i$ is an orbit closure in $X_{\Sigma}$ corresponding to a 2-dimensional cone in $\Sigma$.\vskip 4 pt

\noindent\underline{$E_{10}$.} This component is $D_{\rho_1}\cap D_{\rho_2}$, where $\rho_{1}$ is generated by $(1, 0, 0, 0)$ and $\rho_{2}$ is generated by $(9,-3,-2,-4)$. To see this, consider the chart given by the cone $\tau:= \tau_{42}$. In this chart, $\widetilde{V}$ is generated by $\tilde{f}_{\tau} = y_2y_3 + y_3 + y_1^2y_2 + y_2$. Restricted to $y_2 = 0$ (corresponding to the ray $(1,0,0,0)$), this gives $\overline{f}_{\tau} = y_3$ (corresponding to the ray $(9,-3,-2,-4)$). We now define $N_{\rho_1 ,\rho_2}$ to be the sublattice of $N$ generated (as a group) by $(\rho_1 \cap N)\times (\rho_2 \cap N)$ and

$$N(\rho_1 ,\rho_2) = N/N_{\rho_1,\rho_2},\,\,\,M(\rho_1, \rho_2) = M\cap \rho_1^{\perp}\cap \rho_2^{\perp}$$
Let $M$ be the matrix with columns the vectors generating $\tau$, i.e.

$$\begin{bmatrix}
0 & 15 & 9 & 1\\
0 & -5 & -3 & 0\\
0 & -3 & -2 & 0\\
1 & -6 & -4 & 0\\
\end{bmatrix}\,\, .$$
and let $\text{Pr}: N \rightarrow N(\rho_1 ,\rho_2)$ be the projection map

$$\text{Pr}:= \begin{bmatrix}
1 & 0 & 0 & 0\\
0 & 1 & 0 & 0\\
\end{bmatrix} \times M^{-1} =
\begin{bmatrix}
0 & 0 & -2 & 1\\
0 & -2 & 3 & 0\\ \end{bmatrix}\,\, .$$
The set of cones in $\Sigma$ containing both $\rho_1$ and $\rho_2$ is defined by its set of images in $N(\rho_1 ,\rho_2)$ under $\text{Pr}$. There are 6 maximal cones in $\Sigma$ containing both $(1, 0, 0, 0)$ and $(9, -3, -2, -4)$, and they project down to the fan generated by the rays $(-2, 3)$, $(-1, 1)$, $(-1, 2)$, $(0, -1)$, $(0, 1)$ and $(1, 0)$. This fan represents $\text{Bl}_3 \mathbb{P}_2$.\vskip 4 pt

\noindent\underline{$E_{11}$.} This is the intersection of $D_{\rho}$, $\rho$ generated by $(9,-3,-2,-4)$, and the non-exceptional divisor corresponding to the ray $(0,0,1,0)$. The computation is similar as for $E_{10}$, with

$$\text{Pr}= \begin{bmatrix}
1 & -1 & 0 & 3\\
0 & 4 & 0 & -3\\ \end{bmatrix}$$
and the fan generated by the rays $(-1, 4)$, $(0, -1)$, $(0, 1)$, $(1, -1)$ and $(1, 0)$ represents $\text{Bl}_1(\mathbb{F}_4)$.\vskip 4 pt

\underline{Method 4.} We need the following definitions. Let $P$ be a polytope in $\mathbb{R}^d$. For every nonempty face $F$ of $P$ we define

$$N_F := \{ c \in (\mathbb{R}^d)^*  | F\subset \{ x \in P | cx\geq cy \,\, \forall y\in P \} \} \,\, .$$
We define the {\it normal fan} $\mathcal{N}_P$ as

$$\mathcal{N}_P = \{ N_F | \text{$F$ is a face of $P$} \} \,\, .$$

\noindent\underline{$E_{12}$} is computed by finding a polytope $\Delta$ in $M(\rho)_{\mathbb{R}}$ which has $\text{Star}(\rho)$ as normal fan. Now the ray $\rho$ is generated by the vector  $(5, -1, -1, -2)$. Consider the local chart given by the cone $\tau$, where $\tau$ is spanned by the vectors $(0, 1, 0, 0)$, $(9, -3, -2, -4)$, $(7, -2, -1, -3)$,$(5, -1, -1, -2)$. Let M be the matrix with columns the vectors generating $\tau$, i.e.

$$M = \begin{bmatrix}
0 & 9 &  7 & 5\\
1 & -3 & -2 & -1\\
0 & -2 & -1 & -1\\
0 & -4 & -3 & -2\\
\end{bmatrix}\,\, .$$
\noindent
The projection map $\text{Pr}: N \rightarrow N(\rho)$ is given by

$$\text{Pr}:= \begin{bmatrix}
1 & 0 & 0 & 0\\
0 & 1 & 0 & 0\\
0 & 0 & 1 & 0\\
\end{bmatrix} \times M^{-1} =
\begin{bmatrix}
-1 & 1 & 0 & -3\\
-1 & 0 & -1 & -2\\
0 & 0 & 2 & -1\end{bmatrix}$$
The fan $\text{Star}(\rho)$ in $N(\rho)$ is generated by the rays $(-3, -2, -1)$, $(-1, -1, 0)$, $(0, 0, 1)$, $(0, 1, 0)$ and $(1, 0, 0)$. Up to translation, the polytope with vertices

$$(0, 0, 0), (1, 0, 0), (0, 1, 0), (1, 0, 1), (0, 0, 4), (0, 1, 2)$$
has $\text{Star}(\rho)$ as normal fan, see figure~\ref{figure:polytope}.

\begin{figure}
  \begin{center}
  \includegraphics[width=3cm]{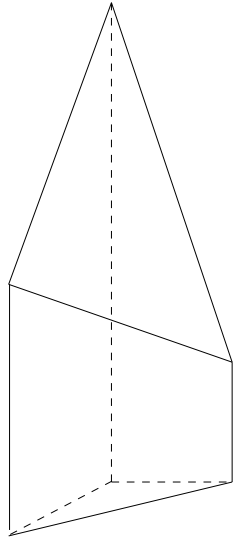}\end{center}
  \caption{The polytope with $\text{Star}(\rho)$ as a normal fan.}\label{figure:polytope}
\end{figure}

We let $S_{\Delta}$ be the graded $\mathbb{C}$-algebra generated by monomials $t^k\chi^m$, where $m$ is an element of the Minkowski sum $k\Delta$ and $\chi^m = x^{m_1}y^{m_2}z^{m_3}$ for $m = (m_1, m_2, m_3)$. The 10 lattice points contained in the polytope, the 6 vertices above and 4 interior lattice points, give us the equations

$$
\begin{tabular}{l}
$z_0 = t$\\
$z_1 = tx$\\
$z_2 = ty$\\
$z_3 = tz$\\
$z_4 = txz$\\
$z_5 = tz^2$\\
$z_6 = tz^3$\\
$z_7 = tz^4$\\
$z_8= tyz$\\
$z_9 = tyz^2$\\
\end{tabular}$$
defining an embedding of $D_{\rho}$ in $\mathbb{P}^{9}$. Its equations are given by the $2\times2$ minors of the matrix

$$\begin{bmatrix}\begin{array}{cccc|c|cc}
z_3 & z_5 & z_6 & z_7 & z_4 & z_8 & z_9\\
z_0 & z_3 & z_5 & z_6 & z_1 & z_2 & z_8\\
\end{array}
\end{bmatrix}$$
Now consider the vertex $(1,0,1)$, corresponding to the variable $z_4$. Consider the chart $z_4 \neq 0$ with $x_i = z_i/z_4$. The remaining coordinates are $x_1$, $x_7$ and $x_9$, corresponding to the cone with vectors $(-1,0,3)$, $(-1,1,1)$ and $(0,0,-1)$. This is the dual of the cone in $\text{Star}(\rho)$ with rays generated by the vectors $(-3, -2, -1)$, $(-1, -1, 0)$ and $(0, 1, 0)$. In this chart, $\tilde{V}\cap D_{\rho}$ is given by the equation $x_7 + x_9 + x_1^2x_7$. In the torus coordinates, this is $z/x\cdot (1 + y + z^2)$. In fact, in every chart $\tilde{V}\cap D_{\rho}$ is given by the equation $u\cdot (1 + y + z^2)$, where $u$ is an invertible element in $\mathbb{C}[x,y,z, 1/x, 1/y, 1/z]$. Since $(1 + y + z^2) = z_0 + z_2 + z_5$, the equations for $\tilde{V}\cap D_{\rho}$ reduce to the $2\times2$ minors of the matrix

$$\begin{bmatrix}\begin{array}{cccc|c}
z_3 & z_5 & z_6 & z_7 & z_4 \\
z_0 & z_3 & z_5 & z_6 & z_1 \\
\end{array}
\end{bmatrix}$$
This is a $(4:1)$ rational scroll in this embedding; i.e. $E_{12} = \mathbb{F}_3$.

The space E is a normal crossing divisor. We may therefore describe the intersections of components with a simplicial complex (the {\it dual complex} or {\it intersection complex}). The vertices $\{ i \}$ correspond to the components $E_i$, and $\{i_1, \ldots i_k \}$ is a face if $E_{i_1} \cap \cdots \cap E_{i_k} \neq \emptyset$.

The intersection complex may be computed by looking at the various $\tilde{V}\cap D_{\rho_1} \cap D_{\rho_2}$ and $\tilde{V}\cap D_{\rho_1} \cap D_{\rho_2}\cap D_{\rho_3}$. We list here the facets of the complex.
\begin{align*}
\{1,2,7\}, \{2,7,8\}, \{3,8,11\}, \{4,10,11\}, \{4,10,12\}, \{5,7,9\}, \{5,7,10\},\,\,\, \\
\{5,10,12\},\{6,7,9\}, \{6,7,10\}, \{6,10,12\}, \{7,8,9\}, \{7,8,10\}, \{8,10,11\}.\\
\end{align*}
We see that there are 14 facets, corresponding to 14 intersection points of 3 components and 25 edges corresponding to 25 projective lines which are the intersections of 2 components.

\newtheorem{euler25}[eulerBoehm]{Lemma}
\begin{euler25}
The Euler characteristic of $E$ is 25.
\end{euler25}\label{proposition:e25}

\begin{proof}
Using the inclusion-exclusion principle and the fact that $E$ has normal crossings, we may compute the Euler characteristic $\chi(E)$ by

$$\chi(E) = \underset{i}{\textstyle\sum \, } \chi(E_i) - \underset{i<j}{\textstyle\sum \, }\chi(E_i\cap E_j) + \underset{i<j<k}{\textstyle\sum \, }
 \chi(E_i\cap E_j \cap E_k)\,\,\, .$$
Now from Table~\ref{table:exceptional} we count $\Sigma \chi (E_i) = 61$. We have $\chi (\mathbb{P}^1) = 2$. Thus from the intersection complex, we compute $\chi(E) = 61 - 25\times 2 + 14 = 25$.
\end{proof}

For the other 2 (quotient) singularities we may use the McKay correspondence as conjectured by Miles Reid and proved by Batyrev in~\cite{bat}, Theorem 1.10. In a crepant resolution of $\mathbb{C}^n/H$, $H$ a finite subgroup of $SL_n$, the Euler characteristic of the exceptional divisor will be the number of conjugacy classes in $H$. In our case this is 13. Denote these exceptional divisors $E'$.

We have constructed a resolution $M_s \rightarrow Y_s$, where $Y_s$ is the quotient $X_s/H$. Let $U$ be the complement of the 6 singular points in $Y_s$.

\newtheorem{ueuler}[eulerBoehm]{Lemma}
\begin{ueuler}
The Euler characteristic of $U$ is $-6$.
\end{ueuler}\label{lemma:ueuler}

\begin{proof}The singular variety $X_s$ smooths to the general degree 13 Calabi-Yau 3-fold in $\mathbb{P}^6$. A Macaulay 2 computation shows that the total space of a general one parameter smoothing is smooth. The smooth fiber has Euler characteristic -120, see Proposition~\ref{proposition:ex4cohom}. The Milnor fiber of the 3-dimensional $Q_{12}$ singularity is a wedge sum of 12 3-spheres. Thus $\chi(X_s) = -120 + 4\times12 = -72$. Hence

$$\chi(U) = \frac{\chi(X_s \setminus \{ 6\, \text{points} \} )}{13} = \frac{-72 - 6}{13} = -6$$
\end{proof}

\noindent We can now put all this together to prove the main result of this section.\\

{\it Proof of Theorem~\ref{theorem:eBoehm}}.  Since the resolution $M_s \rightarrow Y_s$ is an isomorphism away from the 6 points, four points with exceptional divisor $E$ and two points with exceptional divisor $E^{\prime}$, we have $\chi(M_s) = \chi(U) + 4\chi(E) +2\chi(E') =  -6 + 4\cdot25 + 2\cdot 13 = 120$.

\begin{flushleft}

\appendix

\chapter{Computer Calculations}
The following is a Macaulay 2 code for computing $T^1_{X}$ of an variety $X = \text{Proj}( T/p)$, for an ideal $p$ in a ring $T$.\\ \ \\ \

\begin{table}[h]
\begin{center}
\begin{tabular}{l}
A = resolution(p, LengthLimit => 3)\\
rel = transpose(A.dd$\_$2)\\
dp = transpose  jacobian(A.dd$\_$1)\\
R = T/p\\
Rel = substitute(rel,R)\\
Dp = substitute(dp,R)\\
Der = image Dp\\
N = kernel Rel\\
N0 = image basis(0,N)\\
Der0 = image basis(0,Der)\\
isSubset(Der0,N0)\\
T1temp = N0/Der0\\
T1 = trim T1temp\\
T1mat = gens T1
\end{tabular}
\end{center}
\end{table}\label{sideM2code}
\end{flushleft}
\begin{table}[tp]\footnotesize
\begin{center}
\begin{tabular}{ll}
$\tau_1:$ & (13, -5, -3, -6), (0, 0, 1, 0), (3, -1, 0, -1), (8, -3, -1, -3)\\
$\tau_2:$ & (13, -5, -3, -6), (3, -1, 0, -1), (6, -2, -1, -2),(11, -4, -2, -4)\\
$\tau_3:$ & (0, 0, 0, 1), (6, -2, -1, -2), (14, -5, -3, -6), (15, -5, -3, -6)\\
$\tau_4:$ & (13, -5, -3, -6), (3, -1, 0, -1), (6, -2, -1, -2),(14, -5, -3, -6)\\
$\tau_5:$ & (0, 1, 0, 0), (0, 0, 1, 0), (9, -3, -2, -4), (11, -4, -2, -5)\\
$\tau_6:$ & (0, 1, 0, 0), (0, 0, 1, 0), (7, -2, -1, -3), (5, -1, 0, -2)\\
$\tau_7:$ & (0, 0, 1, 0), (1, 0, 0, 0), (3, 0, 0, -1), (5, -1, 0, -2)\\
$\tau_8:$ & (0, 1, 0, 0), (0, 0, 0, 1), (9, -3, -2, -4), (5, -1, -1, -2)\\
$\tau_9:$ &  (0, 1, 0, 0), (1, 0, 0, 0), (7, -2, -1, -3), (5, -1, -1, -2)\\
$\tau_{10}:$ & (0, 0, 0, 1), (1, 0, 0, 0), (6, -2, -1, -2), (15, -5, -3, -6)\\
$\tau_{11}:$ & (1, 0, 0, 0), (3, -1, 0, -1), (6, -2, -1, -2), (15, -5, -3, -6)\\
$\tau_{12}:$ & (0, 0, 1, 0), (1, 0, 0, 0), (3, -1, 0, -1), (12, -4, -2, -5)\\
$\tau_{13}:$ & (0, 0, 1, 0), (1, 0, 0, 0), (9, -3, -2, -4), (12, -4, -2, -5)\\
$\tau_{14}:$ & (0, 1, 0, 0), (0, 0, 0, 1), (0, 0, 1, 0), (1, 0, 0, 0)\\
$\tau_{15}:$ & (0, 0, 0, 1), (1, 0, 0, 0), (3, -1, 0, -1), (6, -2, -1, -2)\\
$\tau_{16}:$ & (9, -3, -2, -4), (1, 0, 0, 0), (7, -2, -1, -3), (5, -1, -1, -2)\\
$\tau_{17}:$ & (0, 1, 0, 0), (0, 0, 0, 1), (1, 0, 0, 0), (5, -1, -1, -2)\\
$\tau_{18}:$ & (1, 0, 0, 0), (7, -2, -1, -3), (3, 0, 0, -1), (5, -1, 0, -2)\\
$\tau_{19}:$ & (0, 1, 0, 0), (0, 0, 1, 0), (3, 0, 0, -1), (5, -1, 0, -2)\\
$\tau_{20}:$ & (13, -5, -3, -6), (0, 0, 1, 0), (3, -1, 0, -1), (11, -4, -2, -5)\\
$\tau_{21}:$ & (3, -1, 0, -1), (14, -5, -3, -6), (15, -5, -3, -6),(12, -4, -2, -5)\\
$\tau_{22}:$ & (9, -3, -2, -4), (14, -5, -3, -6), (15, -5, -3, -6),(12, -4, -2, -5)\\
$\tau_{23}:$ & (3, -1, 0, -1), (13, -5, -3, -6), (14, -5, -3, -6),(12, -4, -2, -5)\\
$\tau_{24}:$ & (9, -3, -2, -4), (13, -5, -3, -6), (14, -5, -3, -6),(12, -4, -2, -5)\\
$\tau_{25}:$ & (0, 0, 0, 1), (3, -1, 0, -1), (6, -2, -1, -2), (11, -4, -2, -4)\\
$\tau_{26}:$ & (0, 0, 0, 1), (0, 0, 1, 0), (3, -1, 0, -1), (8, -3, -1, -3)\\
$\tau_{27}:$ & (13, -5, -3, -6), (0, 0, 0, 1), (3, -1, 0, -1), (8, -3, -1, -3)\\
$\tau_{28}:$ & (13, -5, -3, -6), (0, 0, 0, 1), (6, -2, -1, -2), (11, -4, -2, -4)\\
$\tau_{29}:$ & (13, -5, -3, -6), (0, 0, 0, 1), (6, -2, -1, -2), (14, -5, -3, -6)\\
$\tau_{30}:$ & (3, -1, 0, -1), (6, -2, -1, -2), (14, -5, -3, -6),(15, -5, -3, -6)\\
$\tau_{31}:$ & (13, -5, -3, -6), (0, 1, 0, 0), (9, -3, -2, -4), (11, -4, -2, -5)\\
$\tau_{32}:$ & (0, 0, 1, 0), (9, -3, -2, -4), (11, -4, -2, -5), (12, -4, -2, -5)\\
$\tau_{33}:$ & (0, 0, 1, 0), (3, -1, 0, -1), (11, -4, -2, -5), (12, -4, -2, -5)\\
$\tau_{34}:$ & (0, 0, 1, 0), (1, 0, 0, 0), (7, -2, -1, -3), (5, -1, 0, -2)\\
$\tau_{35}:$ & (0, 1, 0, 0), (1, 0, 0, 0), (7, -2, -1, -3), (3, 0, 0, -1)\\
$\tau_{36}:$ & (0, 1, 0, 0), (9, -3, -2, -4), (7, -2, -1, -3), (5, -1, -1, -2)\\
$\tau_{37}:$ & (0, 0, 1, 0), (9, -3, -2, -4), (1, 0, 0, 0), (7, -2, -1, -3)\\
$\tau_{38}:$ & (0, 0, 0, 1), (0, 0, 1, 0), (1, 0, 0, 0), (3, -1, 0, -1)\\
$\tau_{39}:$ & (13, -5, -3, -6), (0, 1, 0, 0), (0, 0, 0, 1), (9, -3, -2, -4)\\
$\tau_{40}:$ & (1, 0, 0, 0), (9, -3, -2, -4), (15, -5, -3, -6), (12, -4, -2, -5)\\
$\tau_{41}:$ & (1, 0, 0, 0), (3, -1, 0, -1), (15, -5, -3, -6), (12, -4, -2, -5)\\
$\tau_{42}:$ & (0, 0, 0, 1), (1, 0, 0, 0), (9, -3, -2, -4), (15, -5, -3, -6)\\
$\tau_{43}:$ & (0, 1, 0, 0), (0, 0, 1, 0), (9, -3, -2, -4), (7, -2, -1, -3)\\
$\tau_{44}:$ & (0, 0, 0, 1), (9, -3, -2, -4), (1, 0, 0, 0), (5, -1, -1, -2)\\
$\tau_{45}:$ & (0, 1, 0, 0), (0, 0, 1, 0), (1, 0, 0, 0), (3, 0, 0, -1)\\
$\tau_{46}:$ & (0, 1, 0, 0), (7, -2, -1, -3), (3, 0, 0, -1), (5, -1, 0, -2)\\
$\tau_{47}:$ & (9, -3, -2, -4), (11, -4, -2, -5), (13, -5, -3, -6), (12, -4, -2, -5)\\
$\tau_{48}:$ & (3, -1, 0, -1), (11, -4, -2, -5), (13, -5, -3, -6),(12, -4, -2, -5)\\
$\tau_{49}:$ & (13, -5, -3, -6), (0, 1, 0, 0), (0, 0, 1, 0), (11, -4, -2, -5)\\
$\tau_{50}:$ & (0, 0, 0, 1), (9, -3, -2, -4), (14, -5, -3, -6), (15, -5, -3, -6)\\
$\tau_{51}:$ & (13, -5, -3, -6), (0, 0, 0, 1), (9, -3, -2, -4), (14, -5, -3, -6)\\
$\tau_{52}:$ & (13, -5, -3, -6), (0, 0, 0, 1), (3, -1, 0, -1), (11, -4, -2, -4)\\
$\tau_{53}:$ & (13, -5, -3, -6), (0, 0, 0, 1), (0, 0, 1, 0), (8, -3, -1, -3)\\
\end{tabular}\label{table:cones}\caption{The maximal cones of the subdivision $\Sigma$.}\end{center}\end{table}

\begin{flushleft}
\chapter{Explicit Expressions for the Varieties in Chapter 3}\chaptermark{Explicit Expressions}
In the $P^7_1$ case, studied in Section~\ref{ex1section}, the linear entries of $M^{1}$ are\label{ex1refUttrykk}

\[
\begin{array}{l}
l_{1} = x_{4}\\
l_{2} = t_{65}x_{1} + t_{57}x_{2}+ t_{61}x_{3} + t_{30}x_{5}\\
l_{3} = -x_{5} \\
l_{4} = x_{6}\\
l_{5} = -t_{55}x_{1} - t_{59}x_{2}-t_{63}x_{3}-t_{33}x_{6}\\
l_{6} = -x_{7}
\end{array}
\]

and the cubic terms are

\[
\begin{array}{l}
g_{1} =  -t_{15}x_{7}^{3}-t_{17}x_{3}x_{7}^2
-t_{18}x_{2}x_{7}^2 -t_{19}x_{1}x_{7}^{2}
-t_{22}x_{2}x_{3}x_{7} -t_{24}x_{1}x_{3}x_{7}\\
-t_{25}x_{1}x_{2}x_{7}
-t_{36}x_{1}^3 - t_{37}x_{1}^2x_{7} -t_{42}x_{2}^2x_{7}
-t_{43}x_{2}^3 -t_{46}x_{3}^3
-t_{47}x_{3}^2x_{7}\\
-t_{67}x_{1}^2x_{3} -
t_{68}x_{1}^2x_{2} - t_{76}x_{2}^2x_{3}
-t_{77}x_{1}x_{2}^2 -t_{82}x_{2}x_{3}^2 -t_{83}x_{1}
x_{3}^2\\ \ \\
g_{2} = -t_{1}x_{1}^3 -t_{2}x_{2}^3 -t_{3}x_{3}^3
-t_{6}x_{4}^3 -t_{9}x_{3}x_{4}^2 -t_{10}x_{2}x_{4}^2
-t_{11}x_{1}x_{4}^2 \\
-t_{16}x_{7}^3 -t_{20}x_{3}x_{7}^2
-t_{21}x_{2}x_{7}^2 -t_{23}x_{1}x_{7}^2 -t_{26}x_{5}^3
-t_{27}x_{2}x_{5}^2 -t_{28}x_{3}x_{5}^2\\
-t_{29}x_{1}x_{5}^2 -t_{31}x_{6}^3 -t_{32}x_{3}x_{6}^2
-t_{34}x_{1}x_{6}^2 -t_{35}x_{2}x_{6}^2
-t_{48}x_{4}^2x_{5}
-t_{49}x_{4}x_{5}^2\\
 -t_{50}x_{5}^2x_{6} -t_{51}x_{5}x_{6}^2 -t_{52}x_{6}^2x_{7}
-t_{53}x_{6}x_{7}^2 -t_{54}x_{1}^2x_{6}
-t_{56}x_{2}^2x_{5} -t_{58}x_{2}^2x_{6}\\
-t_{60}x_{3}^2x_{5} -t_{62}x_{3}^2x_{6}
-t_{64}x_{1}^2x_{5} -t_{66}x_{1}^2x_{7}
-t_{69}x_{1}^2x_{4} -t_{72}x_{2}^2x_{4}
-t_{75}x_{2}^2x_{7}\\
 -t_{78}x_{3}^2x_{4} -t_{81}x_{3}^2x_{7}
-t_{84}x_{3}x_{4}x_{5}
-t_{85}x_{2}x_{4}x_{5} -t_{86}x_{1}x_{4}x_{5}
-t_{87}x_{3}x_{5}x_{6}\\
 -t_{88}x_{2}x_{5}x_{6} -t_{89}x_{1}x_{5}x_{6}
-t_{90}x_{3}x_{6}x_{7} -t_{91}x_{2}x_{6}x_{7}
-t_{92}x_{1}x_{6}x_{7} -x_{1}x_{2}x_{3}\\ \ \\
g_{3} =  -t_{4}x _{4}^3  -t_{5}x _{3} x _{4}^2 -t_{6}x
_{3}
x _{4}^2 + t_{6}x _{4}^2 x _{5} -t_{7}x _{1} x _{4}^2 -t _{8}x
_{2} x _{4}^2  -t _{12}x _{2}x _{3}x _{4}\\
 -t _{13}x_{1}x_{3} x_{4} -t _{14}x _{1} x _{2}x _{4} -t_{38}x _{1}^2 x
_{4} -t_{39}x_{1}^3 -t_{40}x _{2}^3 -t_{41}x _{2}^2 x
_{4} -t_{44}x _{3}^3\\
-t_{45}x_{3}^2x_{4} -t_{70}x_{1}^2 x _{3} -t_{71}x _{1}^2 x _{2}
-t_{73}x _{2}^2 x _{3} -t_{74}x _{1}x _{2}^2 -t_{79}x _{2}
x_{3}^2 -t _{80}x _{1} x _{3}^2 \end{array} \]

In the $P^7_2$ case, studied in Section~\ref{ex2section}, the cubic $g$ is

\[
\begin{array}{l}g = t_{15}x_7^3 + t_{16}x_6x_7^2 + t_{17}x_4x_7^2 + t_{18}x_3x_7^2 + t_{19}x_2x_7^2 + t_{21}x_3x_6x_7 + t_{22}x_3x_4x_7\\
\,\,+ t_{23}x_2x_6x_7 + t_{24}x_2x_4x_7 + t_{50}x_2^3 + t_{51}x_2^2x_7 + t_{52}x_2^2x_6 + t_{53}x_2^2x_4 + t_{55}x_3^3\\
\,\,+ t_{56}x_3^2x_7 + t_{57}x_3^2x_6 + t_{58}x_3^2x_4 + t_{65}x_4^3 + t_{66}x_4^2x_7 + t_{68}x_3x_4^2 + t_{69}x_2x_4^2\\
\,\,+ t_{75}x_6^3 + t_{76}x_6^2x_7 + t_{78}x_3x_6^2 + t_{79}x_2x_6^2\label{ex2refUttrykk} \,\, ,\end{array} \]

and the quadrics $q_1,\ldots,q_4$ are

\[
\begin{array}{l}
q_1 = t_{10}x_5^2 + t_{12}x_3x_5 + t_{13}x_2x_5 + t_{34}x_2^2 + t_{36}x_3^2 + t_{60}x_4^2 + t_{61}x_4x_5\\
\,\,+ t_{62}x_3x_4 + t_{63}x_2x_4 + t_{70}x_5x_6 + t_{71}x_6^2 + t_{72}x_3x_6 + t_{73}x_2x_6\\ \ \\
q_2 = x_4x_6 + t_1x_2^2 + t_2x_3^2 + t_8x_1^2 + t_{11}x_5^2 + t_{25}x_7^2 + t_{29}x_1x_5 + t_{32}x_2x_3 + t_{35}x_2x_5\\
\,\, + t_{37}x_3x_5 + t_{43}x_1x_2 + t_{48}x_1x_3 + t_{54}x_2x_7 + t_{59}x_3x_7\\ \ \\
q_3 = x_2x_3 + t_3x_4^2 + t_4x_6^2 + t_6x_1^2 + t_{14}x_5^2 + t_{20}x_7^2 + t_{26}x_1x_4 + t_{28}x_1x_5 + t_{30}x_1x_6\\
\,\,+ t_{39}x_4x_6 + t_{64}x_4x_5 + t_{67}x_4x_7 + t_{74}x_5x_6 + t_{77}x_6x_7\\ \ \\
q_4 = t_5x_1^2 + t_7x_1x_6 + t_9x_1x_4 + t_{27}x_4^2 +t_{31}x_6^2 + t_{40}x_1x_2 + t_{41}x_2^2 + t_{42}x_2x_6\\
\,\,+ t_{44}x_2x_4 + t_{45}x_1x_3 + t_{46}x_3^2 +t_{47}x_3x_6 + t_{49}x_3x_4\, \, \, .\end{array}\]

In the $P^7_3$ case, studied in Section~\ref{ex3section}, the entries of the syzygy matrix $M^{1}$ are given by

\[
\begin{array}{l}
g =  x_1x_2x_3 + t_1x_1^3 + t_4x_2^3 + t_7x_3^3 + t_{10}x_4^3 + t_{12}x_3x_4^2 + t_{13}x_2x_4^2 + t_{14}x_1x_4^2\\ \label{ex3refUttrykk}
\,\,+ t_{15}x_5^3 + t_{17}x_3x_5^2 + t_{18}x_2x_5^2 + t_{19}x_1x_5^2 + t_{20}x_6^3 + t_{21}x_7^3 + t_{23}x_3x_6^2 + t_{24}x_2x_6^2\\
\,\,+ t_{25}x_1x_6^2 + t_{27}x_3x_7^2 + t_{28}x_2x_7^2 + t_{29}x_1x_7^2 + t_{36}x_1^2x_4 + t_{38}x_1^2x_5 + t_{40}x_1^2x_6\\
\,\, + t_{42}x_1^2x_7 + t_{44}x_2^2x_4 + t_{46}x_2^2x_5 + t_{48}x_2^2x_6 + t_{50}x_2^2x_7 + t_{52}x_3^2x_4 + t_{54}x_3^2x_5\\
\,\,+ t_{56}x_3^2x_6 + t_{58}x_3^2x_7 + t_{60}x_4^2x_6 + t_{61}x_3x_4x_6 + t_{62}x_2x_4x_6 + t_{63}x_1x_4x_6\\
\,\,+ t_{64}x_4^2x_7 + t_{65}x_3x_4x_7 + t_{66}x_2x_4x_7 + t_{67}x_1x_4x_7 + t_{68}x_5^2x_6 + t_{69}x_3x_5x_6\\
\,\, + t_{70}x_2x_5x_6 + t_{71}x_1x_5x_6 + t_{72}x_5^2x_7 + t_{73}x_3x_5x_7 + t_{74}x_2x_5x_7 + t_{75}x_1x_5x_7\\
\,\,+ t_{76}x_4x_6^2 + t_{77}x_4x_7^2 + t_{78}x_5x_6^2 + t_{79}x_5x_7^2\end{array}\]
\[
\begin{array}{l}
q_1 = x_4x_5 + t_2x_1^2 + t_5x_2^2 + t_8x_3^2 + t_{22}x_6^2 + t_{26}x_7^2 + t_{30}x_1x_2 + t_{32}x_1x_3 + t_{34}x_2x_3\\
\,\, + t_{41}x_1x_6 + t_{43}x_1x_7 + t_{49}x_2x_6 + t_{51}x_2x_7 + t_{57}x_3x_6 + t_{59}x_3x_7\end{array} \]
\[
\begin{array}{l}
q_2 = x_6x_7 + t_3x_1^2 + t_6x_2^2 + t_9x_3^2 + t_{11}x_4^2 + t_{16}x_5^2 + t_{31}x_1x_2 + t_{33}x_1x_3 + t_{35}x_2x_3\\
\,\, + t_{37}x_1x_4 + t_{39}x_1x_5 + t_{45}x_2x_4 + t_{45}x_2x_5 + t_{53}x_3x_4 + t_{55}x_3x_5\,\, .\end{array} \]\pagebreak

In the $P^7_4$ case, studied in Section~\ref{ex4section}, the quadrics are given by

\[
\begin{array}{l}
q_1 = t_{20}x_1^2 + t_{22}x_2^2 + t_{35}x_1x_3 + t_{38}x_2x_3 + t_{57}x_3^2 + t_{34}x_1x_4 + t_{37}x_2x_4 \\
\,\,  + t_{67}x_4^2 + t_{21}x_1x_5 + t_{23}x_2x_5 + t_{8}x_3x_5 + t_{7}x_4x_5 + t_{5}x_5^2\\ \ \\
q_2 =  t_{1}x_1^2+t_{30}x_1x_2+t_2x_2^2 + x_3x_4+t_{33}x_1x_5 + t_{36}x_2x_5 +t_6x_5^2+t_{48}x_1x_6\\
\,\, +t_{52}x_2x_6 +t_{63}x_5x_6+t_{16}x_6^2 +t_{50}x_1x_7+t_{55}x_2x_7+t_{65}x_6x_7+t_{14}x_7^2\\ \ \\
q_3 =  -x_1x_2-t_3x_3^2 -t_{39}x_3x_4-t_4x_4^2 -t_{56}x_3x_5 -t_{66}x_4x_5 -t_9x_5^2-t_{58}x_3x_6\\
\,\,  -t_{60}x_4x_6-t_{62}x_5x_6-t_{15}x_6^2 -t_{42}x_3x_7 -t_{45}x_4x_7-t_{64}x_6x_7-t_{11}x_7^2\\ \ \\
q_4 = -t_{51}x_1^2-t_{54}x_2^2-t_{44}x_1x_3 -t_{43}x_2x_3 -t_{26}x_3^2 -t_{47}x_1x_4 -t_{46}x_2x_4\\
\,\, -t_{28}x_4^2 -t_{13}x_1x_7 -t_{12}x_2x_7 -t_{27}x_3x_7 -t_{29}x_4x_7-t_{10}x_7^2\,\, ,\end{array} \]\label{ex4refUttrykk}

and the linear forms are given by

\[
\begin{array}{l}
l_1 = t_{18}x_1+t_{19}x_2+t_{32}x_3+t_{31}x_4\\
l_2 =  x_7\\
l_3 =  -x_6\\
l_4 =  t_{49}x_1+t_{53}x_2+t_{59}x_3+t_{61}x_4+t_{17}x_6\\
l_5 =  x_5\\
l_6 =  t_{41}x_1+t_{40}x_2+t_{24}x_3+t_{25}x_4\,\, .
\end{array} \]\\ \ \\ \ \\

In the $P^7_5$ case, studied in Section~\ref{ex5section}, the entries of the syzygy matrix $M^1$ are\label{ex5refUttrykk}
\[
\begin{array}{l}
l_1 = t_8x_1 + t_9x_3 + t_{24}x_2\\
l_2 = t_1x_1+t_{23}x_3+t_{26}x_6 + t_{44}x_4+t_{45}x_5\\
l_3 = t_5x_5 + t_{36}x_3 + t_{42}x_7 + t_{46}x_1 + t_{48}x_2\\
l_4 = t_{16}x_3+t_{17}x_5 + t_{35}x_4\\
l_5 = t_{10}x_1 + t_{11}x_6 + t_{25}x_7\\
l_6 = t_6x_6 + t_{27}x_1 + t_{39}x_4 +t_{50}x_2+t_{52}x_3\\
l_7 =  t_3x_3 + t_{22}x_1 + t_{34}x_5 + t_{51}x_6 + t_{54}x_7\\
l_8 = t_{18}x_4 +t_{19}x_6 +t_{38}x_5\\
l_9 = t_4x_4 + t_{30}x_2 + t_{37}x_6 + t_{43}x_1 + t_{56}x_7\\
l_{10} = t_{12}x_2 + t_{13}x_4 + t_{29}x_3 \\
l_{11} = t_2x_2 + t_{28}x_4 + t_{31}x_7 + t_{47}x_5 + t_{49}x_6\\
l_{12} = t_{14}x_2+t_{15}x_7 + t_{33}x_1\\
l_{13} = t_7x_7 + t_{32}x_2 + t_{41}x_5 + t_{53}x_3 + t_{55}x_4\\
l_{14} = t_{20}x_5 + t_{21}x_7 + t_{40}x_6\,\, .
\end{array} \]

\end{flushleft}

\bibliography{referanser}
\end{document}